\documentclass[11pt]{article}
\usepackage{graphicx,subcaption} 
\usepackage{amsmath,amssymb,color,xcolor,bm,amsthm,enumerate}
\usepackage{epstopdf}
\usepackage{tikz}
\usetikzlibrary{arrows.meta,angles,quotes}
\usetikzlibrary{calc} 
\usepackage[percent]{overpic}
\usepackage{tabularx}
\usepackage{multirow}
\usepackage{booktabs}
\usepackage{mathrsfs}
\usepackage{comment}
\usepackage{float}
\usepackage[normalem]{ulem}

\usepackage[hmargin={25mm,25mm},vmargin={25mm,30mm}]{geometry}

\newcommand{\bb}{\mathbf{b}}

\newcommand{\bx}{\mathbf{x}}
\newcommand{\bz}{\mathbf{z}}
\newcommand{\bn}{\mathbf{n}}

\newcommand{\wt}{\widetilde}

\newtheorem{theorem}{Theorem}[section]
\newtheorem{corollary}{Corollary}[section]
\newtheorem{lemma}[theorem]{Lemma}
\newtheorem{remark}[theorem]{Remark}
\newtheorem{proposition}{Proposition}[section]
\newtheorem{assumption}{Assumption}[section]
\newtheorem{definition}{Definition}[section]

\newtheorem{property}{Property}[section]
\newtheorem{example}{Example}[section]

\makeatletter\@addtoreset{equation}{section}\makeatother

\newcommand{\red}[1]{{\color{red}#1}}
\newcommand{\blue}[1]{{\color{blue}#1}}

\def\mesh{{\cal T}_h}

\def\IVE{{\cal I}_E}

\def\IVh{{\cal I}_h}

\def\Pol{{\mathbb P}}

\def\mesh{{{\cal T}_h}}
\newcommand{\V}{\mathbb{V}}

\newcommand{\VE}{{V_E}}
\newcommand{\WE}{{W_E}}
\newcommand{\VdE}{{B_{\partial E}}}
\newcommand{\edgeE}{{\mathcal{E}_E}}
\newcommand{\nodesE}{{\mathcal{V}_E}}
\newcommand{\PiE}{{\Pi^\nabla_E}}
\newcommand{\Vmesh}{V_{h}}
\newcommand{\Wmesh}{W_{h}}
\newcommand{\edge}{\mathcal{E}}

\newcommand{\enorm}[2][2,\Omega]{\left|\negthinspace\left|\negthinspace\left|{#2}%
\right|\negthinspace\right|\negthinspace\right|_{#1}}

\title{A Virtual Element Method for elliptic problems \\ on trimmed background meshes}
\author{L. Beir\~ao da Veiga\thanks{Dipartimento di Matematica e Applicazioni,
Universit\`a degli Studi di Milano Bicocca,
Via Roberto Cozzi 55 - 20125 Milano, Italy 
  ({\it lourenco.beirao@unimib.it})}
\and C. Canuto  \thanks{Dipartimento di Scienze Matematiche G.L. Lagrange, Politecnico di Torino,
Corso Duca degli Abruzzi 24 - 10129 Torino, Italy
  ({\it claudio.canuto@polito.it})}
\and R. H. Nochetto  \thanks{Department of Mathematics and Institute
for Physical Science and Technology,
University of Maryland, College Park - 20742, MD, USA 
  ({\it rhn@umd.edu})}
\and G. Vacca \thanks{Dipartimento di Matematica,
Universit\`a degli Studi di Bari,
Via Edoardo Orabona 4  - 70125 Bari, Italy 
  ({\it giuseppe.vacca@uniba.it})}
\and M. Verani \thanks{MOX-Laboratory for Modeling and Scientific Computing, Dipartimento di Matematica, Politecnico di Milano, Piazza Leonardo da Vinci 32 - 20133 Milano, Italy
  ({\it marco.verani@polimi.it})}
  }
\date{\today}

\begin{document}

\maketitle

\begin{abstract}
\noindent We consider a two-dimensional piecewise $C^2$ domain that cuts through a quasi-uniform fixed polygonal background mesh, for instance made of quadrilaterals. A simple procedure based on convex hulls gives rise to a rather small number of polygonal boundary elements of various shapes, including elements with small edges and large aspect ratios; this is the computational mesh for a virtual element method (VEM), a trimmed background mesh. We classify all possible geometric configurations and study their stability and approximability properties. This entails deriving robust stabilization mechanisms and interpolation estimates for anisotropic elements and elements with small cuts, as well as a weak maximum principle for enhanced virtual elements; these contributions have intrinsic interest for VEM theory on geometric flexibility. We prove that the resulting VEM is uniformly stable in $H^1$, and also show optimal order-regularity error estimates in $H^1$ and $L^2$. Insightful numerical experiments corroborate and complement our theory. The proposed method is suitable for treating ALE formulations of problems in moving domains.

\end{abstract}

\section{Introduction}\label{S:introduction}

Accurately discretizing partial differential equations (PDEs) in domains with complex geometries or interfaces is crucial for many applications. One important example is fluid-structure interaction problems, which are ubiquitous in engineering. They are typically governed by (nonlinear) PDEs posed on domains that evolve in time and exhibit large deformations 
The Arbitrary Lagrangian Eulerian (ALE) formulation is a popular computational technique that approximates the PDEs on a domain-fitted shape-regular mesh (Eulerian approach) but allows the mesh to deform together with the domain (Lagrangian approach). This marriage of conceptually distinct approaches requires extending the velocity of the domain boundary inside the domain so as to avoid mesh distortion. This is a delicate matter because none of the existing techniques can guarantee robust mesh regularity and often require remeshing, especially in dealing with large domain defomations. 

The Virtual Element Method (VEM) is a relatively new discretization paradigm for PDEs which, in contrast to popular Finite Element Methods (FEMs), allows for general polytopal meshes \cite{volley, hitchhiker, vem-acta, vem-special-issue, vem-ep-deformations}. 
Ever since its inception, the VEM has been recognized in both the mathematics and engineering communities for its flexibility and robustness with respect to mesh design and handling \cite{vem-esher, vem-noncov, vem-deformations, wriggers2024virtual, vem-dfn, vem-imati, vem-imati-genova, vem-nn}.

Our overall goal is to exploit this geometric flexibility of VEMs for ALE formulations in 2d. We assume that we have a fixed polygonal background mesh, for instance made of rectangular cells, that is sufficiently fine to resolve features of the domain boundary; we do not include possible adaptive mesh refinement in our study. We further assume that the domain boundary is composed by $C^2$ arcs connected via break points, thereby forming corners, and that their evolution is given a priori by a smooth geometric law; we do not consider topological changes. 
At each time instant, the domain boundary cuts the background mesh into elements of various shapes, including anisotropic and relatively small objects. Such elements give rise to boundary polygons, by a simple procedure based on convex hulls; a finite, rather small, number of distinct configurations may arise, as illustrated in Fig. 1 below. These polygons together with the elements of the background mesh fully contained in $\bar{\Omega}$ form our computational mesh, which we call a {\it trimmed background mesh}.
The mesh nodes fit the $C^2$ parts of the boundary, not necessarily the break points.

The advantage of this VEM-ALE approach, relative to traditional FEM-ALE approaches, is that the only elements undergoing geometric distortions produce boundary elements, which can be classified and handled separately in terms of stability and approximability. This avoids the two bottlenecks of FEM-ALE approaches, namely suitably extending the boundary velocity to maintain mesh regularity and dealing with remeshing. However, the success of our approach hinges crucially on quantitative understanding of the asymptotic behavior of VEM under extreme geometric situations such as anisotropy, small elements and small edges.  
While nowadays there is a strong theoretical background on dealing with small edges \cite{BLRstab, brenner2018, smalledges, vem-steklov},
rigorous results for anisotropic polygons are still limited \cite{Chen_anisotr:2018, Chen_anisotr_nc}, in particular not sufficient for our purposes of studying stability and accuracy of VEM-ALE.


Therefore, in this paper we fill in this gap and examine simpler stationary problems that exhibit the main geometric difficulties of time-dependent ALE formulations, and develop tools applicable to VEM-ALE, which we will be the focus of our paper \cite{NostroMovingDomain}. To this end, we consider the elliptic coercive boundary value problem
\begin{subequations}\label{eq:cpb-intro}
  \begin{align}
    - \nabla \cdot (\mu \nabla u) + \bb\cdot \nabla u + \sigma u &= f \quad \mathrm{in~}\Omega,\label{cpb:01}\\
    u&=0 \quad \mathrm{on~}\partial\Omega,\label{cpb:02}
  \end{align}
\end{subequations}
whose domain boundary $\partial\Omega$ is made of $N$ $C^2$-arcs $\gamma_j$ connected by break points $\bz_j$. The low regularity of $u$ near corners suggests a discretization of the problem by virtual elements of order one.
Our main contributions are as follows:

\begin{itemize}
    \item {\it Cutting rule}: We propose a simple recipe based on convex hulls to cut elements of the structured background mesh traversed by $\partial\Omega$. This rule allows for a sistematic classification of shapes and extends to 3d.

    \item {\it Anisotropic elements}: Mesh cuts can give rise to polygonal elements with very large aspect ratios. Exploiting that at least one node belongs to $\partial\Omega$ we develop novel sharp $H^1$ and $L^2$ interpolation error estimates.

    \item {\it Stabilization}: We develop a new stabilization technique for quadrilaterals which is robust with respect to anisotropies.

    \item {\it Weak maximum principle}: We show that the elementwise $L^\infty$-norm of functions in an `enhanced' VEM space is controlled up to a multiplicative constant $C$ by the $L^\infty$-norm on the element boundary, where $C$ may depend on the element shape. We provide examples of
classes of polygons for which this constant is shape-independent. $L^\infty$ estimates are instrumental for the developments in this paper, but are also of intrinsic general interest for VEM theory.

    \item {\it Realistic regularity}: The $L^2$-based regularity of the solution $u$ of \eqref{eq:cpb-intro} on a 2d domain $\Omega$ with corners $\bz_j$ is known. We exploit asymptotic expressions for $u$ near $\bz_j$ to derive sharp local interpolation estimates in $H^1$ and $L^2$. It is worth stressing that the VEM mesh is not conforming with respect to the break points $\bz_j$, due to the cutting strategy, thereby leading to errors rarely quantified in the literature.
    
    \item {\it Stability and approximation}: We develop an analysis based on abstract assumptions, which are subsequently verified for our VEM. This leads to $L^2$ and $H^1$ error estimates for the Galerkin solutions which turn out to be sharp relative to the underlying regularity of the solution $u$ and boundary $\partial\Omega$.

    \item {\it Numerical experiments}: We present computational results that not only corroborate theory but also extend it. We discuss quantitative estimates of the effect of anisotropy on the error estimates and condition number of the resulting linear systems.
    
\end{itemize}

Dealing with structured (e.g. cartesian) grids cut by $\partial\Omega$ is not a new idea. It is central to {\it unfitted} discretization methods which allow the boundary or interface to be represented independently of the computational mesh. A brief literature overview of unfitted methods follows:
immersed boundary
method (IBM) \cite{McCrackenPeskin1980,Peskin2002} and \cite{BoffiGastaldiHeltai2007,BoffiCavalliniGastaldi2015};
fictitious domain method \cite{GlowinskiPanPeriaux1994,GiraultGlowinski1995};
boundary penalty method: \cite{BarrettElliott1986};
volume penalty method \cite{Maury2009};
extended finite element method (XFEM) \cite{MoesDolbowBelytschko1999,FriesBelytschko2010};
immersed finite element method (IFEM) \cite{ZhangGerstenbergerWangLiu2004};
Cut-FEM \cite{BeckerHansboStenberg2003,
BurmanHansbo2010} 
and the overviews \cite{BurmanHansboLarsonMassing2015,Cut_Fem_Acta:2025};
transfer path method \cite{CockburnSolano2012};
shifted boundary method \cite{MainScovazzi2018,
AtallahCanutoScovazzi:2021}
$\phi$-FEM \cite{DuprezLlerasLozinski2023};
finite cell method \cite{ParvizianDusterRank2007}.
All these methods deal with a background grid, and simplify mesh generation, but bring additional
challenges such as accurate imposition of boundary and interface conditions, integration over
elements intersected by the boundary, and ensuring stability and well-conditioned linear systems.

Our VEM method for \eqref{eq:cpb-intro} is not formally an unfitted discretization because boundary nodes belong to $\partial\Omega$, except for break points $\bz_j$ which generically are not nodes. Our cut elements are admissible within the VEM framework and do not require special treatment regarding integration. Moreover, essential boundary conditions can be imposed on the VEM space thereby bypassing weak imposition of Dirichlet conditions via Nitsche's method, penalty method, or the method of multipliers. However, the robustness of VEM with respect to extreme geometric conditions, such as large anisotropies and small cuts, is essential for its competitiveness and range of applicability. Studying these issues is the principal objective of this paper.

The paper is organized as follows. Section \ref{sec:pde} presents the boundary-value problem of interest, describes the structure of the solution in the presence of corners, and discusses its extension outside the domain where the problem is set, needed to handle the discrepancy between the exact and computational domains. In Sect. \ref{sec:discrete-abstract} we describe how the virtual element mesh is derived from a given background mesh, next we introduce the virtual element spaces and forms, and finally present our discretization method. Sect. \ref{stability-convergence-abstract} is devoted to the study of the stability and convergence of the method, under abstract assumptions on the robustness with respect to the element shapes of the VEM stabilization form and the interpolation and projection errors. Optimal and robust error estimates are established in both $H^1$ and $L^2$ norms; a key ingredient in the analysis is the study of the domain error. The three subsequent sections lead to the verification of the abstract assumptions in the important case of a Cartesian background mesh. Precisely, Sect. \ref{sec:maximum-principle} proves a generalized maximum principle for `enhanced' virtual element functions, a result instrumental to the subsequent analysis but of intrinsic interest in the theory of Virtual Element Methods. Sect.~\ref{sec:stab:X} suggests the choice of shape-robust stabilization forms for all the possible types of elements in our meshes, whereas in Sect. \ref{sec:checking} we prove robust approximation results, in both cases of regular and non-regular solutions. Sect. \ref{sec:numerics} presents various numerical results, which enrich our theoretical analysis with a quantitative insight, and add further information on other features of the method.
Conclusions and perspectives are contained in Sect. \ref{sec:conclusions}.

\section{The Elliptic Problem}\label{sec:pde}

Let $\Omega \subset \mathbb{R}^2$ be a bounded domain whose Lipschitz boundary satisfies the following conditions:
\begin{itemize}
\item $\partial \Omega$ is a piecewise-$C^2$ manifold: precisely, it is the union of $N\geq 1$ $C^2$ arcs of curve $\boldsymbol{\gamma}_j, \ 1 \leq  j \leq N$, such that
\begin{equation}\label{eq:curvature}
B := \max_j \max_s \| \boldsymbol{\gamma}_j'' (s)\|  < \infty,
\end{equation}
where differentiation is taken with respect to the arclength $s \in [0,\ell_j]$ of $\boldsymbol{\gamma}_j$. It is convenient to set $\boldsymbol{\gamma}_{N+1}:=\boldsymbol{\gamma}_1$. If $N>1$ then arcs $\boldsymbol{\gamma}_j$ and $\boldsymbol{\gamma}_{j+1}$ meet at a break point ${\bm z}_j$, forming an angle $\omega_j \in (0,2\pi)$ containing a portion of $\Omega$. If $N=1$ then $\boldsymbol{\gamma}_1$ is a closed arc describing the whole boundary $\partial\Omega$, which may be globally $C^2$ or may contain one break point  ${\bm z}_1$ corresponding to the endpoints of $\boldsymbol{\gamma}_1$. The break points $\boldsymbol{z}_j$ are the only allowed intersections between arcs.


\end{itemize}

\medskip

\noindent
We are interested in solving the following model boundary value problem:
\begin{subequations}\label{eq:cpb}
  \begin{align}
    - \nabla \cdot (\mu \nabla u) + \bb\cdot \nabla u + \sigma u &= f \quad \mathrm{in~}\Omega,\label{cpb:1}\\
    u&=0 \quad \mathrm{on~}\partial\Omega,\label{cpb:2}
  \end{align}
\end{subequations}
with coefficients $\mu \in L^\infty(\Omega)$, $\bb \in (L^\infty(\Omega))^2$ satisfying $\nabla \cdot \bb \in L^\infty(\Omega)$, $\sigma \in L^\infty (\Omega)$, and forcing $f\in L^2(\Omega)$. We make the standard assumptions
\begin{equation}\label{eq:assumptions-coefficients}
\mu \geq \mu_0, \qquad \sigma - \tfrac 1 2 \nabla\cdot \bb \geq 0 \qquad \mathrm{a.e. \  in~}\Omega
\end{equation}
for some constant $\mu_0 > 0$. These assumptions guarantee the well-posedness of the problem in $H^1_0(\Omega)$, since the bilinear form 
\begin{equation}\label{eq:form-B1}
{\cal B}(u, v) := \int_\Omega \mu \nabla  u \cdot \nabla v + \int_\Omega (\bb\cdot \nabla u ) v + \int_\Omega \sigma uv
\end{equation}
turns out to be continuous and coercive in this norm; precisely, the unique solution $u$ satisfies the bound
\begin{equation}\label{eq:bound-u}
C_P^{-1}\| u \|_{0,\Omega} + | u |_{1,\Omega} \leq \frac{C_P}{\mu_0} \| f \|_{0,\Omega},
\end{equation}
where $C_P>0$ denotes the Poincar\'e constant in $H^1_0(\Omega)$.

In view of the numerical discretization, we write the convective term in skew-symmetric form, obtaining
\begin{equation}\label{eq:form-B2}
{\cal B}(u, v) = a(u,v) + b(u,v) + c(u,v)
\end{equation}
with
\[
a(u,v):=\int_\Omega \mu \nabla  u \cdot \nabla v, 
\quad b(u,v):=\tfrac12 \int_\Omega (\bb\cdot \nabla u ) v - \tfrac12 \int_\Omega (\bb\cdot \nabla v ) u, \quad c(u,v) := \int_\Omega (\sigma - \tfrac 1 2 \nabla\cdot \bb)\, u\, v.
\]

Furthermore, we assume $\mu \in W^{1}_\infty(\Omega)$, by which we derive from \eqref{eq:cpb} that $u$ is the solution of the homogeneous Dirichlet problem for the Poisson equation 
\begin{subequations}\label{eq:cpbr}
  \begin{align}
    - \Delta u &= g \quad \mathrm{in~}\Omega,\label{cpbr:1}\\
    u&=0 \quad \mathrm{on~}\partial\Omega,\label{cpbr:2}
  \end{align}
\end{subequations}
with forcing
\begin{equation}\label{eq:def-g}
g:= \frac1{\mu}\left( f + (\nabla \mu -\bb)\cdot\nabla u -\sigma u \right)
\end{equation}
satisfying
\begin{equation}\label{eq:bound-g}
\| g \|_{0,\Omega} \leq \frac1{\mu_0} \left(
1 + \frac{C_P}{\mu_0}(\| \nabla \mu \|_{(L^\infty(\Omega))^d} + \| \bb \|_{(L^\infty(\Omega))^2}) + \frac{C_P^2}{\mu_0} \| \sigma \|_{L^\infty(\Omega)})
\right ) \|f\|_{0,\Omega}.
\end{equation}

\subsection{Structure of the solution}\label{sec:solution-structure}
%
Grisvard's regularity theory (\cite{Grisvard}) applied to \eqref{eq:cpbr} yields $u \in H^2(\Omega)$ if $\partial\Omega$ contains no reentrant corner. On the other hand, if at the break point $\bz_j$ a reentrant corner with $ \pi < \omega_j < 2\pi$ exists, then in a neighborhood ${\cal N}_j \cap \Omega$ of $\bz_j$ which is far away from the other break points, $u$ is the sum of a regular function and a singular function:
\begin{equation}\label{eq:split-u-Nj}
u = \psi_j + \lambda_j \zeta_j
\end{equation}
where  $\psi_j \in H^2({\cal N}_j \cap \Omega)$ and vanishes in a neighborhhod of $\bz_j$, $\lambda_j \in \mathbb{R}$ is the intensity factor of the singularity,
and
\begin{equation}\label{eq:def-zeta-j}
\zeta_j(x) := \zeta_j^0(T_j(x)),
\end{equation}
where $T_j$ is a ${\cal C}^2$-diffeomorphism that maps ${\cal N}_j \cap \Omega$ onto a straight sector pointed at the origin
\[
S_j :=\{ y=r{\rm e}^{i \theta} : 0 < r < r_0, \ 0 < \theta < \omega_j \},
\]
and the singular function $\zeta_j^0$ is given by
\begin{equation}\label{eq:def-zeta_j-0}
\zeta_j^0(y) :=  r^{\gamma_j}  \sin(\gamma_j \theta)  \phi_j(r), \qquad \gamma_j =\tfrac{\pi}{\omega_j} \in \left(\tfrac12, 1 \right).
\end{equation}
Here, $\phi_j$ is a smooth cut-off function satisfying $\phi_j=1$ in a neighborhood of the origin, and $\phi_j=0$ for $r\geq r_0$.
The following result describes the smoothness of the singular part.
\begin{lemma}\label{lem:smooth-zeta_j}
The function $\zeta_j$ satisfies
\[
\zeta_j \in H^1_0({\cal N}_j \cap \Omega) \cap W^2_p({\cal N}_j \cap \Omega) \qquad \text{for all \ } p < \frac2{2-\gamma_j} <2.
\]
\end{lemma}
\proof If $D^k$ denotes any partial derivative of order $k \in \{1,2\}$ with respect to the variables $y$, it is easily seen that
\[
D^1 \zeta^0_j = r^{\gamma_j -1}\xi(r,\theta), \qquad D^2 \zeta^0_j = r^{\gamma_j -2}\eta(r,\theta),
\]
where $\xi$, $\eta$ are smooth functions in $S_j$. Since
\[
\int_0^{r_0} r^{2(\gamma_j -1)} r {\rm d}r <\infty, \qquad \int_0^{r_0} r^{p(\gamma_j -2)} r {\rm d}r <\infty \quad \text{iff} \quad p < \frac2{2-\gamma_j},
\]
we obtain $\zeta^0_j \in H^1(S_j)\cap W^2_p(S_j)$. The thesis follows from the smoothness of the mapping $T_j$. 
\endproof

Going from local to global, the solution $u$ can be represented in $\Omega$ as
\begin{equation}\label{eq:split-u}
u = \psi +\zeta,
\end{equation}
where $\psi \in H^2(\Omega) \cap H^1_0(\Omega)$, whereas $\zeta$ is zero if there are no reentrant corners, or 
\begin{equation}\label{eq:def-zeta}
\zeta = \sum_{j=1}^J \lambda_j \zeta_j
\end{equation}
if there are $J\geq 1$ reentrant corners (that we assume to be numbered first). For convenience, we will set $J=0$, whence $u=\psi$, to include in the notation the case of no reentrant corners.
Furthermore, introducing the norm
\begin{equation}\label{eq:def-norm-u}
\enorm{u} := \| \psi \|_{2, \Omega} + \sum_{j=1}^J |\lambda_j|,
\end{equation}
one can prove \cite{Grisvard,Nochetto_Liao} the regularity bound
\begin{equation}\label{eq:reg-bound-u}
\enorm{u} \leq C_R \| g \|_{0,\Omega},
\end{equation}
for some constant $C_R>0$ depending only on the domain $\Omega$. 

The global regularity result for $u$ follows from Lemma \ref{lem:smooth-zeta_j} and \eqref{eq:split-u}-\eqref{eq:reg-bound-u}.
\begin{proposition}\label{prop:reg-u}
The solution $u$ satisfies $u \in H^1_0(\Omega) \cap W^2_p(\Omega)$, where $p=2$ if $\partial\Omega$ does not contain reentrant corners, or $p$ is any number satisfying
\begin{equation}\label{eq:cond-p-gamma}
1 \leq p < \frac2{2-\bar{\gamma}}, \qquad \bar{\gamma}:= \min_j \gamma_j,
\end{equation}
if $\partial\Omega$ contains $J \geq 1$ reentrant corners with exponents $\gamma_j$. Furthermore, there exists $\bar{C}>0$ depending only on the domain $\Omega$ such that
\[
\| u \|_{W^2_p(\Omega)} \leq \bar{C} \enorm{u}  \lesssim \| g \|_{0,\Omega}.
\]
\end{proposition}\endproof


\subsection{Extensions}\label{sec:extension}
Since we wish to numerically solve our problem in a domain not entirely contained in $\Omega$, we need to extend data to a (slightly) larger domain, say $\widetilde{\Omega}$, containing $\overline{\Omega}$; let $\delta_0>0$ be a lower bound for the Hausdorff distance between $\partial\Omega$ and $\partial\widetilde{\Omega}$.

Suppose that the equation data $\mu, \bb, \sigma, f$ are restrictions to $\Omega$ of functions $\widetilde{\mu}, \widetilde{\bb}, \widetilde{\sigma}, \widetilde{f}$ defined in $\widetilde{\Omega}$, and satisfying therein the same hypotheses we have made in $\Omega$. 
To define an extension $\widetilde{u}$ of the solution $u$ to $\widetilde{\Omega}$, 
we use the representation \eqref{eq:split-u}-\eqref{eq:def-zeta} of $u$, and extend each function appearing therein:
\begin{equation}\label{eq:split-u-ext}
\widetilde{u} := \widetilde{\psi} + \widetilde{\zeta} =\widetilde{\psi} + \sum_{j=1}^J \lambda_j \widetilde{\zeta}_j.
\end{equation}
Precisely, $\widetilde{\psi}$ is a Calder\'on extension of $\psi$, satisfying $\| \widetilde{\psi} \|_{2,\widetilde{\Omega}} \lesssim \| {\psi} \|_{2,\Omega}$ (\cite{Adams:1975}), whereas $\widetilde{\zeta}_j$ is defined, according to \eqref{eq:def-zeta-j}, as
\begin{equation}\label{eq:def-zeta-j-ext}
\widetilde{\zeta}_j(x) = \widetilde{\zeta}_j^0(\widetilde{T}_j(x));
\end{equation}
here, $\widetilde{T}_j$ is an extension of $T_j$, which maps ${\cal N}_j$ onto the disk $\widetilde{S}_j:=\{ y=r{\rm e}^{i \theta} : 0 \leq r < r_0, \ 0 \leq \theta < 2\pi \}$ diffeomorphically, while $\widetilde{\zeta}_j^0$ is defined, according to \eqref{eq:def-zeta_j-0}, as
\begin{equation}\label{eq:def-zeta_j-0-ext}
\widetilde{\zeta}_j^0(y) =  r^{\gamma_j}  s_j(\theta)  \phi_j(r),
\end{equation}
where $s_j(\theta)$ is a smooth $2\pi$-periodic function that extends $\sin(\gamma_j \theta)$ outside the interval $[0,\omega_j]$. 

Since the radial singularity in $\widetilde{\zeta}_j^0$ is the same as the one in $\zeta_j^0$, the function $\widetilde{u}$ has the same regularity as $u$, i.e., $\widetilde{u} \in H^1(\wt{\Omega}) \cap W^2_p(\wt{\Omega})$, with 
\begin{equation}\label{eq:reg-u-tilde} 
    \| \widetilde{u} \|_{W^2_p(\wt{\Omega})} \lesssim \| \widetilde{\psi} \|_{2,\widetilde{\Omega}} + \sum_{j=1}^J |\lambda_j| \lesssim
    \enorm{u}
    \lesssim \| g \|_{0,\Omega}.
\end{equation}    

We do not expect $\widetilde{u}$ to satisfy the extended equation in $\widetilde{\Omega}$, so we introduce the function
\begin{equation}\label{eq: def-hatf}
\hat{f} := - \nabla \cdot (\widetilde{\mu}\, \nabla \widetilde{u}) + \widetilde{\bb}\cdot \nabla \widetilde{u} + \widetilde{\sigma} \widetilde{u} \, \in L^p(\widetilde{\Omega}),
\end{equation}
which satisfies $\hat{f}=f$ in $\Omega$. Equivalently, the extension $\wt{u}$ satisfies in $\wt{\Omega}$ the equation
\begin{equation}\label{eq: extended-equation}
- \nabla \cdot (\widetilde{\mu}\, \nabla \widetilde{u}) + \widetilde{\bb}\cdot \nabla \widetilde{u} + \widetilde{\sigma} \widetilde{u} = \widetilde{f} + F,
\end{equation}
with $F:=\hat{f}-\widetilde{f} \in L^p(\widetilde{\Omega})$. Note that $F$ vanishes in $\Omega$, or equivalently
\[
\text{supp}\, F \subset \delta\Omega :=\widetilde{\Omega}\setminus \Omega.
\]

\section{The discrete problem}\label{sec:discrete-abstract}

In this section, we introduce a background mesh of size $h>0$ and we define the elements of an associated polygonal mesh obtained by intersecting the background mesh with the domain $\Omega$. The union of such elements gives rise to the computational domain $\Omega_h$, which may differ from the true domain $\Omega$ by at most a region of width $O(h)$. A space of virtual functions is built on the polygonal mesh in $\Omega_h$, by which we define discrete bilinear and linear forms that are used to set up a Galerkin discretization of Problem \eqref{eq:cpb}.

\subsection{Meshes}\label{sec:meshes}

Given a parameter $h>0$, we consider a conforming tessellation ${\cal T}_h^B$ of $\mathbb{R}^2$ made of convex polygons, and a universal constant $\lambda \in (0,1)$, with the following properties: for any polygon $K \in {\cal T}_h^B$
\begin{itemize}
\item its diameter $h_K$ satisfies  $\lambda h \leq h_K \leq \lambda^{-1} h$;
\item $K$ is star-shaped with respect to a ball of diameter $\geq \lambda h_K$;
\item each edge $e_K$ of $K$ has length $\geq \lambda h_K$;
\end{itemize}
We denote by ${\cal V}_K$ the set of vertices of $K$, and we note that its cardinality is bounded by a constant independent of $h$. We call ${\cal T}_h^B$ a \emph{background mesh} of size $h$.

\begin{example}[Cartesian background mesh]\label{exa:background-mesh} {\rm For any $\mathbf{j} \in \mathbb{Z}^2$, we define the Cartesian polytope $K_{\mathbf{j}}:=h\mathbf{j}+[0,h]^2$. The collection ${\cal T}_h^B :=\{K_{\mathbf{j}} : \mathbf{j} \in \mathbb{Z}^2 \}$ of such elements will be called the \emph{Cartesian background mesh} of size $h$.  $\qquad \square$
}
\end{example}

Next, we restrict ourselves to the \emph{finite} background mesh formed by those elements having at least one vertex inside $\Omega$:
\[
{\cal T}_h^B(\Omega) := \{ K \in {\cal T}_h^B \, | \, {\cal V}_K \cap \Omega \not= \emptyset \};
\]
furthermore, from now on we assume $h$ to satisfy $h \leq \lambda \, \delta_0$ ($\delta_0$ being defined in Sect. \ref{sec:extension}), which implies $K \subset \widetilde{\Omega}$ for all $K \in {\cal T}_h^B(\Omega)$. 

A further restriction on the mesh size guarantees local separation of break points, which have finite cardinality.

\begin{property}[mesh resolution]\label{prop:mesh-restrict} There exists $h_0 \leq \lambda \, \delta_0$ and $R>1$ such that for all $h \leq h_0$ there is at most one break point of $\partial\Omega$ in the neighborhood of radius $R h$ of each $K \in {\cal T}_h^B(\Omega)$.
\end{property}

We are going to define the corresponding mesh ${\cal T}_h$ that will be used for the VEM discretization of problem \eqref{eq:cpb}. 
Given any $h \leq h_0$ and any $K \in {\cal T}_h^B(\Omega)$, we generate a convex polytope $E_K \subseteq K$ as follows.

\begin{definition}[definition of $E_K$]\label{def:E-K}
The set $E_K$ is the convex hull of the collection of points in ${\cal V}_K \cap \Omega$ and in
$\partial K_\# \cap \partial\Omega$, where $\partial K_\#$ denotes the union of the edges of $K$ which have at least one endpoint in $\Omega$.
\end{definition}

Note that $E_K \subseteq K$ because $K$ is convex. Furthermore, since
${\cal V}_K \cap \Omega$ is not empty, $E_K$ contains the intersection of $K$ with a $d$-dimensional ball centered at an internal vertex of $K$, hence it is truly $2$-dimensional.

If we consider an edge $e$ of $K$, three situations may occur: \emph{i)} both endpoints are in $\Omega$, in which case by convexity $e$ is also an edge of $E_K$; \emph{ii)} exactly one endpoint is in $\Omega$, in which case $e$ has non-empty intersection with $\partial\Omega$, thus it contains a closed segment which is an edge of $E_K$; \emph{iii)} no endpoint is in $\Omega$, in which case no edge of $E_K$ is contained in $e$ (regardless of whether $e \cap \partial\Omega$ is empty or not).

\begin{remark}[degrees of freedom in $E_K$]\label{rem:dofs-in-E_K}
{\rm
In view of the subsequent definition of virtual functions in $E_K$, it is worth stressing that only the geometric vertices of $E_K$ (i.e., the extremal points of the convex hull) will carry a degree of freedom. Points in $\partial K \cap \partial\Omega$ sitting inside an edge of $E_K$ will not be considered vertices in the sense of VEMs. 
}    
\end{remark}

\begin{example}[shapes on Cartesian mesh]\label{ex:element-geometries}

Consider the Cartesian background mesh ${\cal T}_h^B$ introduced in Example \ref{exa:background-mesh}, and fix a square element $K \in {\cal T}_h^B(\Omega)$.
Let us investigate which are the possible shapes of the element $E_K$. Five cases may occur, which are illustrated in Fig. \ref{fig:geometries-1}.

\begin{figure}[ht!]
\centering
\resizebox{0.6\textwidth}{!}{ 
\begin{minipage}{\textwidth}
\vspace{-12cm}
\begin{subfigure}{0.12\textwidth}
\centering
\begin{tikzpicture}[scale=1, line cap=round, line join=round]
\useasboundingbox (0,0) rectangle (4,4);
\draw[thick] (0,0) rectangle (4,4);
\fill[gray!15]
(0,0) -- (0,4)
plot[smooth, tension=1.1]
    coordinates {(0,2.8) (1.8,1.6) (3.6,0)}
-- (3.6,0) -- (0,0) -- cycle;
\draw[blue, thick]
    plot [smooth, tension=1]
        coordinates {(-0.4,3.04) (0,2.8) (1.8,1.6) (3.6,0) (4,-0.4)};
\draw[red, thick] (0,0) -- (0,2.8) -- (3.6,0) -- cycle;
\filldraw[red] (0,2.8) circle (0.1);
\filldraw[red] (3.6,0) circle (0.1);
\filldraw[red] (0,0) circle (0.06);
\draw[red, thick] (0,0) circle (0.16);
\end{tikzpicture}
\caption*{(A)}
\end{subfigure}
\hfill
\begin{subfigure}{0.12\textwidth}
\centering
\begin{tikzpicture}[scale=4, line cap=round, line join=round]
\useasboundingbox (0,0) rectangle (4,4);
\draw[thick] (0,0) rectangle (1,1);
\coordinate (TopLeft) at (0,1);
\coordinate (BottomLeft) at (0,0);
\coordinate (Iup) at (0.25,1);   
\coordinate (Idown) at (0.55,0); 

\fill[gray!15]
(0,0) -- (0,1) --
 plot [smooth, tension=1]
        coordinates {(Iup) (0.38,0.45) (Idown) };
 -- (0,0) -- cycle;

\draw[red, thick]
    (TopLeft) -- (Iup) -- (Idown) -- (BottomLeft) -- cycle;

\draw[blue, thick]
    plot [smooth, tension=1]
        coordinates {(0.2,1.2) (Iup) (0.38,0.45) (Idown) (0.69,-0.22)};

\filldraw[red] (TopLeft) circle (0.015);
\filldraw[red] (BottomLeft) circle (0.015);
\filldraw[red] (Iup) circle (0.025);
\filldraw[red] (Idown) circle (0.025);

\draw[red, thick] (TopLeft) circle (0.04);
\draw[red, thick] (BottomLeft) circle (0.04);
\end{tikzpicture}
\caption*{(B)}
\end{subfigure}
\hfill
\begin{subfigure}{0.12\textwidth}
\begin{tikzpicture}[line cap=round,line join=round,thick,scale=1,
  declare function={
    f(\x) = 4.2 - 0.25*(\x-2)^2;
  }]
\useasboundingbox (0,0) rectangle (4,4);
\centering

\fill[gray!15]
(4,0) -- (0,0) -- (0,{f(0)})
plot[domain=0:4,samples=50] (\x,{f(\x)})
-- (4,0) -- cycle;
\draw[black,thick] (0,0) rectangle (4,4);

\draw[blue,thick,domain=-0.3:4.3,samples=50] plot (\x,{f(\x)});

\coordinate (L) at (0,{f(0)});   
\coordinate (R) at (4,{f(4)});   
\coordinate (T1) at (1,{f(1)});     
\coordinate (T2) at (3,{f(3)});     

\coordinate (B1) at (0,0);
\coordinate (B2) at (4,0);

\draw[red,thick] (B1)--(L)--(R)--(B2)--cycle;

\foreach \p in {L,R}
  \fill[red] (\p) circle (3pt);

\foreach \p in {B1,B2}
  \draw[red,thick] (\p) circle (4pt);
\foreach \p in {B1,B2}
  \filldraw[red] (\p) circle (0.05);

\end{tikzpicture}
\caption*{(C)}
\end{subfigure}
\hfill
\begin{subfigure}{0.18\textwidth}
\begin{tikzpicture}[line cap=round,line join=round,thick,scale=1,
  declare function={
    H(\x) = (\x < 2 ? 0 : 1);
    f(\x) = 4.5-0.2*(\x-2)^2-(\x-2)*(1.5*H(\x)-1);
  }]
\useasboundingbox (0,0) rectangle (4,4);
\centering

\fill[gray!15]
(4,0) -- (0,0) -- (0,{f(0)})
plot[domain=0:4,samples=50] (\x,{f(\x)})
-- (4,0) -- cycle;

\draw[black,thick] (0,0) rectangle (4,4);

\draw[blue,thick,domain=-0.3:4.3,samples=50] plot (\x,{f(\x)});

\coordinate (L) at (0,{f(0)});   
\coordinate (R) at (4,{f(4)});   
\coordinate (T1) at (1.55,{f(1.55)});     
\coordinate (T2) at (2.45,{f(2.45)});     

\coordinate (B1) at (0,0);
\coordinate (B2) at (4,0);

\draw[red,thick] (B1)--(L)--(R)--(B2)--cycle;

\foreach \p in {L,R}
  \fill[red] (\p) circle (3pt);

\foreach \p in {B1,B2}
  \draw[red,thick] (\p) circle (4pt);
\foreach \p in {B1,B2}
  \filldraw[red] (\p) circle (0.05);
  
\fill[blue] (2,{f(2)}) circle (3pt);
\end{tikzpicture}
\caption*{(D)}
\end{subfigure}

\vspace{-1cm}

\hspace{-20mm}
\begin{subfigure}{0.001\textwidth}
\centering
\begin{tikzpicture}[scale=2,thick]
\useasboundingbox (0,0) rectangle (4,4);
\tikzset{
  point/.style={fill=red, draw=red, circle, radius=2pt},
  vertex/.style={draw=red, thick, circle, inner sep=3pt},}

\coordinate (A) at (0,0);
\coordinate (B) at (2,0);
\coordinate (C) at (0,2);
\coordinate (D) at (2,2);

\coordinate (P1) at (0.7,2);
\coordinate (P2) at (2,0.5);

\coordinate (P3) at (1,0);
\coordinate (P4) at (0,1);

\fill[gray!15]
  (2,2) --
 plot [smooth, tension=1]
        coordinates {(P2) (1.5,0.35) (P3)};
-- (P3) -- (0,0) -- cycle;

\fill[gray!15]
  (2,2) -- (P3) -- (0,0) -- cycle;
  
\fill[gray!15]
(0,0) --
plot [smooth, tension=0.9]
coordinates {(P4) (0.5,1.5) (P1)};
 (2,2) -- cycle;

\fill[gray!15]
  (0,0) -- (P1) -- (2,2) -- cycle;
\draw[black, thick] (A)--(C)--(D) --(B) --(A);

\draw[red, thick] (A)--(P4)--(P1)--(D)--(P2)--(P3)--(A);

\foreach \p in {A,D}
  \draw[red] (\p) circle (0.080);

\draw[blue, thick]
    plot [smooth, tension=0.9]
        coordinates {(0.67,2.2) (P1) (0.5,1.5) (P4) (-0.2,0.9)};

\draw[blue, thick]
    plot [smooth, tension=1]
        coordinates {(2.2,0.46) (P2) (1.5,0.35) (P3) (0.9,-0.2)};

\filldraw[red] (P1) circle (0.05);
\filldraw[red] (P2) circle (0.05);
\filldraw[red] (P4) circle (0.05);
\filldraw[red] (P3) circle (0.05);

\filldraw[red] (A) circle (0.03);
\filldraw[red] (D) circle (0.03);
\end{tikzpicture}
\caption*{(E)}
\end{subfigure}
\hfill
\begin{subfigure}{0.001\textwidth}
\centering
\begin{tikzpicture}[scale=2,thick]
\useasboundingbox (0,0) rectangle (4,4);
\tikzset{
  point/.style={fill=red, draw=red, circle, radius=2pt},
  vertex/.style={draw=red, thick, circle, inner sep=3pt},}

\coordinate (A) at (0,0);
\coordinate (B) at (2,0);
\coordinate (C) at (0,2);
\coordinate (D) at (2,2);

\coordinate (P1) at (0.7,2);
\coordinate (P2) at (2,0.5);

\coordinate (P3) at (1,0);
\coordinate (P4) at (0,1);

\fill[gray!15]
(0,0) -- 
 plot [smooth, tension=1]
        coordinates {(P4) (0.5,0.5) (P3)};
-- cycle;

   \fill[gray!15]
(0,0) -- 
 plot [smooth, tension=1]
        coordinates {(P4) (0.55,0.55) (P3)};
-- cycle;

 \fill[gray!15]
(D) -- 
 plot [smooth, tension=1]
        coordinates {(P1) (1.25,1.25) (P2)};
-- cycle;

\draw[black, thick] (A)--(C)--(D) --(B) --(A);

\draw[red, thick] (A)--(P4)--(P1)--(D)--(P2)--(P3)--(A);

\foreach \p in {A,D}
  \draw[red] (\p) circle (0.080);

\draw[blue, thick]
    plot [smooth, tension=0.9]
        coordinates {(0.5,2.3) (P1) (1.25,1.25) (P2) (2.1,0.4)};

\draw[blue, thick]
    plot [smooth, tension=1]
        coordinates {(-0.2,1.1) (P4) (0.55,0.55) (P3) (1.1,-0.2)};

\filldraw[red] (P1) circle (0.05);
\filldraw[red] (P2) circle (0.05);
\filldraw[red] (P4) circle (0.05);
\filldraw[red] (P3) circle (0.05);

\filldraw[red] (A) circle (0.03);
\filldraw[red] (D) circle (0.03);
\end{tikzpicture}
\caption*{(F)}
\end{subfigure}
\hfill
\begin{subfigure}{0.001\textwidth}
\centering
\begin{tikzpicture}[scale=2,thick]
\useasboundingbox (0,0) rectangle (4,4);
\tikzset{
  point/.style={fill=red, draw=red, circle, radius=2pt},
  vertex/.style={draw=red, thick, circle, inner sep=3pt},}

\coordinate (A) at (0,0);
\coordinate (B) at (2,0);
\coordinate (C) at (0,2);
\coordinate (D) at (2,2);

\coordinate (P1) at (0.7,2);
\coordinate (P2) at (2,0.5);

\fill[gray!15]
(2,0) -- (0,0) -- (0,2) --
 plot [smooth, tension=1]
        coordinates {(P1) (1.35,1.1) (P2)};
 -- (2,0) --  cycle;
\draw[red, thick] (P1)--(P2)--(B)--(A)--(C)--(P1);

\foreach \p in {A,B,C}
  \draw[red] (\p) circle (0.080);

\draw[red, thick] (P1) -- (P2);

\draw[black, thick] (P1)-- (D) -- (P2);

\draw[blue, thick]
    plot [smooth, tension=0.8]
        coordinates {(0.4,2.4) (P1) (1.35,1.1) (P2) (2.2,0.3)};

\filldraw[red] (P1) circle (0.05);
\filldraw[red] (P2) circle (0.05);

\filldraw[red] (A) circle (0.03);
\filldraw[red] (B) circle (0.03);
\filldraw[red] (C) circle (0.03);
\end{tikzpicture}
\caption*{(G)}
\end{subfigure}
\hfill
\begin{subfigure}{0.001\textwidth}
\begin{tikzpicture}[line cap=round,scale=1,line join=round,thick]
\centering
\useasboundingbox (0,0) rectangle (4,4);
\draw[black,thick] (0,0) rectangle (4,4);
\coordinate (M) at (1.8,0.8); 
\coordinate (Pright) at (4,1); 

\fill[gray!15]
(4,0) -- (0,0) -- (0,4) -- (1.0,4)
.. controls (0.85,4.8) and (1.3,3.8) .. (M)
.. controls (2.8,0.5) and (3.4,0.7) .. (Pright)
-- (4,0) -- cycle;

\draw[blue,thick]
  (0.3,4.8) .. controls (0.85,4.8) and (1.3,3.8) .. (M)
  .. controls (2.8,0.5) and (3.4,0.7) .. (Pright)
  .. controls (4.4,1.2) and (5,1.4) .. (4.1,1.05); 

\coordinate (Ptop) at (1.0,4);

\coordinate (A) at (0,0);  
\coordinate (B) at (4,0);  
\coordinate (D) at (0,4);  

\draw[red,thick] (D)--(A)--(B)--(Pright)--(Ptop)--cycle;

\foreach \p in {Ptop,Pright}
  \fill[red] (\p) circle (3pt);

\foreach \p in {A,B,D}
  \draw[red,thick] (\p) circle (4pt);
  \foreach \p in {A,B,D}
\filldraw[red] (\p) circle (0.05);
 \fill[blue] (M) circle (3pt);
\end{tikzpicture}
\caption*{(H)}
\end{subfigure}
\hfill
\begin{subfigure}{0.001\textwidth}
\begin{tikzpicture}[line cap=round,thick,scale=1,
  declare function={
    f(\x) = sqrt(abs(\x-2))+3;
  }]
\useasboundingbox (0,0) rectangle (4,4);
\centering

\fill[gray!15]
(4,0) -- (0,0) -- (0,4) -- (1,{f(1)}) -- (1.5,{f(1.5)}) -- (1.7,{f(1.7)}) -- (1.9,{f(1.9)})-- (2,{f(2)}) -- (2.1,{f(2.1)})-- (2.5,{f(2.5)})-- (3,{f(3)}) -- (4,4) -- cycle;

\draw[black,thick] (0,0) rectangle (4,4);

\draw[blue,thick,domain=-0.3:4.3,samples=150] plot (\x,{f(\x)});

\coordinate (B1) at (0,0);
\coordinate (B2) at (4,0);

\coordinate (B3) at (0,4);
\coordinate (B4) at (4,4);
\draw[red,thick] (B1)--(B2)--(B4)--(B3)--cycle;

\foreach \p in {B1,B2}
  \draw[red,thick] (\p) circle (4pt);
\foreach \p in {B1,B2}
  \filldraw[red] (\p) circle (0.05);

\foreach \p in {B3,B4}
  \draw[red,thick] (\p) circle (4pt);
\foreach \p in {B3,B4}
  \filldraw[red] (\p) circle (0.05);

\fill[blue] (2,{f(2)}) circle (3pt);
\end{tikzpicture}
\caption*{(I)}
\end{subfigure}
\end{minipage}
}
\caption{Examples of geometries of elements $E_K$ based on a Cartesian background mesh. 
We represent an element $K$ of the background mesh in black, the associated element $E_K$ in red, the boundary $\partial\Omega$ in blue; empty circles denote internal nodes, full circles denote boundary nodes. The grey region 
identifies $\Omega \cap K$.\\
Plots are ordered by increasing number of internal vertices of $K$. Plot A: one internal vertex, triangular $E_K$; plots B-C-D: two consecutive internal vertices, quadrilateral $E_K$; plots E-F: two opposite internal vertices, hexagonal $E_K$; plots G-H: three internal vertices, pentagonal $E_K$; plot I: four internal vertices, $E_K=K$}
\label{fig:geometries-1}
\end{figure}
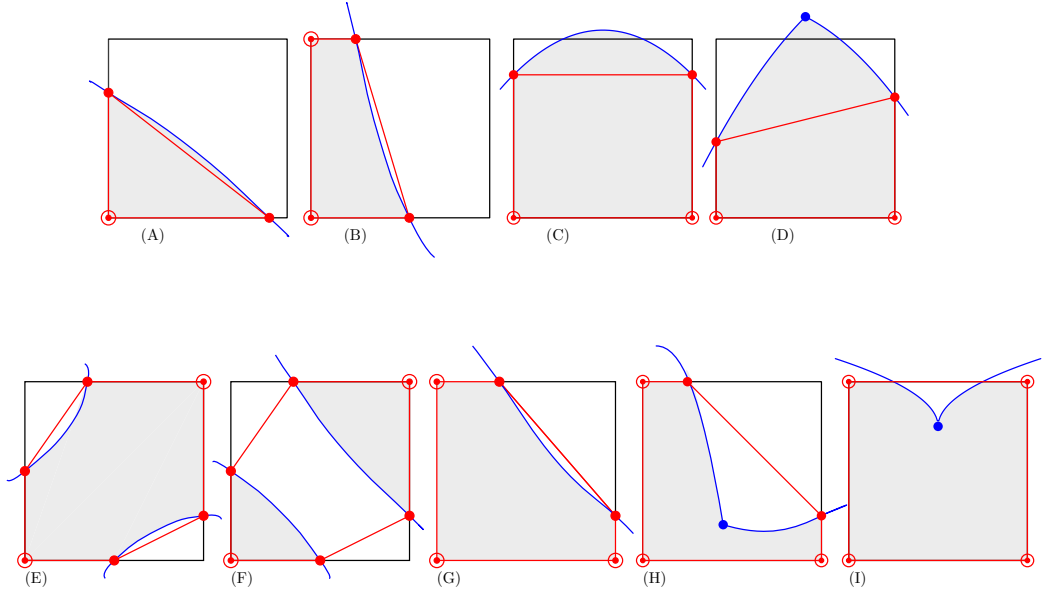

\begin{itemize}
\item \emph{One vertex of $K$ is in $\Omega$.} Then, each of the two edges of $K$ meeting at such vertex contains an edge of $E_K$, whereas the two other edges of $K$ do not intersect $E_K$. By convexity, $E_K$ is a \emph{triangle} (see Fig. \ref{fig:geometries-1}, plot A). 
    \item \emph{Two consecutive vertices of $K$ are in $\Omega$.} Then, the edge of $K$ containing these vertices is also an edge of $E_K$; each of the two edges of $K$ having a single internal vertex contains an edge of $E_K$, whereas the remaining edge of $K$ does not intersect $E_K$. By convexity, $E_K$ is a \emph{quadrilateral} (see Fig. \ref{fig:geometries-1}, plots B, C, and D: denoting by $\hat{e}$ the edge of $K$ opposite to the edge with two internal endpoints, then $\hat{e} \cap \partial\Omega = \emptyset$ in B and $\hat{e} \cap \partial\Omega \not = \emptyset$ in C and D). 

    \item \emph{Two non-consecutive (i.e., opposite) vertices of $K$ are in $\Omega$.} Then, each edge of $K$ contains an edge of $E_K$, which shares with it one of the two internal vertices of $K$. By convexity, $E_K$ is a \emph{hexagon} (see Fig. \ref{fig:geometries-1}, plots E ($\partial\Omega$ oriented NE-SW) and F ($\partial\Omega$ oriented NW-SE)). In situations like F, where $K \cap \Omega$ is not connected, we allow the possibility of replacing the hexagon by the two disjoint triangles which have a vertex in $\Omega$ and two vertices on $\partial K \cap \partial\Omega$.

    \item\emph{Three vertices of $K$ are in $\Omega$.} Then, two contiguous edges of $K$ are also edges of $E_K$, whereas each of the remaining edges of $K$ carry another edge of $E_K$.  By convexity, $E_K$ is a \emph{pentagon} (see Fig. \ref{fig:geometries-1},  plot G (smooth boundary) and H (corner boundary)).

    \item \emph{All the vertices of $K$ are in $\Omega$.} Then, by convexity $E_K=K$ (see Fig. \ref{fig:geometries-1},  plot I). 
\end{itemize}

We refer to Sect. \ref{sec:checking} for a thorough discussion of the properties of the different types of  elements from the point of view of stability and approximation. \hskip 7.2cm$\square$

\end{example}

At this point, we are able to introduce the mesh ${\cal T}_h$ and the domain $\Omega_h$. Let $h_0$ be defined in Property \ref{prop:mesh-restrict} (mesh resolution).
\begin{definition}[definition of the mesh ${\cal T}_h$ and the domain $\Omega_h$]\label{def:T-h}
For any $h \leq h_0$, we set
\[
{\cal T}_h :=\{E : E=E_K \text{\rm \  for some } K \in {\cal T}_h^B(\Omega)\}
\]
and
\[
\Omega_h := \text{\rm interior} \left(\bigcup \{E : E \in {\cal T}_h \} \right)\,.
\] 
\end{definition}
\noindent It is easily seen that ${\cal T}_h$ is a conforming polygonal partition of $\overline{\Omega}_h$, which is a proper subset of the extended domain $\widetilde{\Omega}$ but need not be contained in $\overline{\Omega}$ (except for specific geometric situations, such as, e.g., $\overline{\Omega}$ convex). Moreover, no element $E\in {\cal T}_h$ contains a break point $\bz_j$ in its interior, in particular not a re-entrant corner of $\partial\Omega$ because $E$ is convex.

Given $E \in {\cal T}_h$, let  $K_E$ denote the unique element of the background mesh ${\cal T}_h^B$ containing $E$. Furthermore,
 let $\nodesE$ be the set of nodes $\nu$ sitting on $\partial E$. The set of all edges $e$ of $E$ is denoted by $\edgeE$.

\subsection{VEM spaces}

In order to define a space of discrete functions in $\Omega_h$ associated with $\mesh$, for each element $E \in \mesh$ let us first introduce the space of continuous, piecewise affine functions on $\partial E$
\begin{equation}\label{eq:cond-VdE}
\VdE:= \{ v \in {\cal C}^0(\partial E) : v_{|e} \in \mathbb{P}_1(e) \ \forall e \in \edgeE \}.
\end{equation}

For the purpose of this paper we introduce the following well-known VEM spaces: the basic VEM space of \cite{volley}
\begin{equation}\label{vem:basic} 
W_E := \big\{ w \in H^1(E) \ : \ w|_{\partial E} \in \VdE , \ \Delta w =0 \big\} 
\end{equation}
and the more advanced ``enhanced'' space from \cite{projectors,BBMR16} 
\begin{equation}\label{vem:choice:2} 
V_E := \big\{ v \in H^1(E) \ : \ v|_{\partial E} \in \VdE , \ \Delta v \in \mathbb{P}_1(E) \, ,
\int_E (v - \PiE v) q_1 = 0 \ \forall q_1 \in \mathbb{P}_1(E)
\big\} \,,
\end{equation}
where the projector $\PiE : H^1(E) \to \mathbb{P}_1(E)$ is defined by the conditions
\begin{equation}\label{eq:def-PinablaE}
(\nabla (v - \PiE v), \nabla q)_E = 0 \quad \forall q \in \mathbb{P}_1(E), \qquad  \int_{\partial E} (v-\PiE v) = 0 \, .
\end{equation}
An important observation is  that, due to the peculiar definition of the enhanced space \cite{projectors}, the operator $\Pi^\nabla_E$ coincides with the $L^2(E)$-projection operator $\Pi_E^0$ from $\VE$ onto $\Pol_1(E)$.

It is easy to check that 
$\VE,\WE \subset {\cal C}^0(E)$ and the following properties hold:
\begin{equation}\label{eq:cond-VE}
\text{dim}\, \VE = \text{dim}\, \WE = |\nodesE|\,, \qquad  \mathbb{P}_1(E) \subseteq \VE,\WE \,, \qquad  \tau_{\partial E}(\VE) = \tau_{\partial E}(\WE)= \VdE \,,
\end{equation}
where $ \tau_{\partial E}$ is the trace operator on the boundary of $E$. Note that  functions in $\VE,\WE$   are uniquely identified by their trace on $\partial E$ (and hence by the values at the vertices of $E$), but their  values in the interior of $E$ must be defined.

For any function $v\in \VE$ (or $\WE$), $\text{dofs}_E(v)$ stands for the vector of degrees of freedom of $v$, i.e., the vector collecting the values of $v$ at the vertices of $E$.

Once the local spaces $\VE,\WE$ are defined, we introduce the global discrete spaces
\begin{equation}\label{eq:def-VT}
\Vmesh := \{ v \in H^1_0(\Omega_h) :  \ v_{|E} \in \VE \ \ \forall E \in \mesh \}\,,
\end{equation} 
\begin{equation}\label{eq:def-WT}
\Wmesh := \{ w \in H^1_0(\Omega_h) :  \ w_{|E} \in \WE \ \ \forall E \in \mesh \}\,.
\end{equation} 
Note that functions in $\Vmesh$ (and $\Wmesh$) are piecewise affine on the skeleton $\edge_h$ of ${\cal T}_h$, and indeed they are uniquely determined by their values therein and are globally continuous. 

 Finally, given a continuous function $w$ defined in $E$, we denote by $\IVE w$ the VEM interpolant of $w$ in $\VE$, i.e., the function $\IVE w \in \VE$ such that $\text{dofs}_E(\IVE w)= \text{dofs}_E(w)$. Coherently, for a continuous function $w$ defined in $\overline{\Omega}_{h}$, $\IVh w$ denotes the continuous piecewise VEM interpolant of $w$, i.e., the function satisfying 
 \begin{equation}\label{eq:def-Ih}
 (\IVh w)_{|E}= \IVE(w_{|E}), \qquad \text{for all \ } E\in \mesh.
 \end{equation}
 Such a function vanishes on $\partial\Omega_h$, hence, $\IVh w \in \Vmesh$. Indeed, by construction any vertex of ${\cal T}_h$ sitting on $\partial\Omega_h$ belongs to $\partial\Omega$, hence $\IVh w$ vanishes therein; furthermore, $\IVh w$ is linear on each edge of ${\cal T}_h$.

\subsection{The discrete method}

\newcommand{\PiNew}{\widetilde{\Pi}}

To define a Galerkin approximation of Problem \eqref{eq:cpb}, let us introduce suitable discrete counterparts of the bilinear and linear forms introduced in Sect. \ref{sec:pde}. We start by posing a crucial assumption on the stabilization form, a typical ingredient of VEM discretizations. 
In order to keep the generality required to deal with anisotropic elements near the boundary, in the following we will denote by 
\begin{equation}\label{eq:def-Pnew}
\PiNew_E : \VE \rightarrow {\cal S}_1(E) \, , \quad E \in \mesh \, ,
\end{equation}
a suitable linear projection operator, with ${\cal S}_1(E) \subseteq \Pol_1(E)$ a linear subspace.
We require that $\PiNew_E$ is computable from the DoF values and continuous in $H^1(E)$, uniformly with respect to $E \in \{ \mesh \}_h$. 
Note that, for the large majority of the elements $E$, such operator will be taken as the standard $\Pi^\nabla_E$ projection.

\begin{assumption}[stabilization form]\label{ass:stab-form}
For any $E \in {\cal T}_h$, there exists  a symmetric bilinear form $s_E: \VE \times \VE \to \mathbb{R}$, depending only upon the values of its arguments on the boundary $\partial E$, which
satisfies the conditions
\begin{equation}\label{eq:prop-stabform}
c_s |v|_{1,E}^2 \leq s_E(v,v) \leq C_s |v|_{1,E}^2 \qquad \forall v \in \VE \textrm{ with } \, \PiNew_E v = 0.
\end{equation}
where $C_s \geq c_s >0$ are constants independent of $E$.
\end{assumption}
\noindent In Sect. \ref{sec:stab:X} we will discuss how to construct such a form in the case of a 2D Cartesian background mesh.

To continue, let $\mu_E \geq \mu_0$ denote the integral average of the diffusion coefficient $\mu$ over $E$, and let us introduce the following local discrete bilinear and linear forms: for all $u,v \in V_h(E)$, set
\begin{equation}\label{eq:dis:forms}
\begin{aligned}
& a_E(u,v) := \int_E \mu \, (\nabla \Pi^\nabla_E u)\cdot (\nabla \Pi^\nabla_E v) \, + \mu_E \, 
s_E(u - \PiNew_E u, v - \PiNew_E v) \, ,   \\
& b_E(u,v) := \tfrac12 \int_E [{\bf b}\cdot(\nabla \Pi^\nabla_E u)] \, \Pi^\nabla_E v - \tfrac12 \int_E [{\bf b}\cdot(\nabla \Pi^\nabla_E v)] \, \Pi^\nabla_E u \,    \\
& c_E(u,v) := \int_E (\sigma-\tfrac12 \nabla \cdot \bb) (\Pi^\nabla_E u) (\Pi^\nabla_E v)  
\\
& {\cal F}_E(v) := \int_E f \, \Pi^\nabla_E v \,.
\end{aligned}
\end{equation}
\begin{remark}[extended data]\label{rem:ext-coeff}{\rm
Since by construction the element $E$ may contain a portion outside $\Omega$, the extensions $\tilde{\mu}, \tilde{\bf b}, \tilde{\sigma}$, $\tilde{f}$ of data $\mu, {\bf \beta}, \sigma$, $f$ introduced in Sec. \ref{sec:extension} are implicitly used therein. Here and in the sequel, we avoid adding the superscript $\wt{\ }$ in order to keep notation simpler.  \qquad \qquad \qquad $\square$
}
\end{remark}

The corresponding global discrete forms
read as follows: for all $u,v \in V_h$, set 
$$
\begin{aligned}
& a_h(u,v) := \sum_{E\in {\cal T}_{h}} a_E(u_{|E},v_{|E}) \,, \qquad b_h(u,v) := \sum_{E\in {\cal T}_{h}} b_E(u_{|E},v_{|E}) \,, \\
& c_h(u,v) := \sum_{E\in {\cal T}_{h}} c_E(u_{|E},v_{|E}) \,, \qquad {\cal F}_h(v) := \sum_{E\in {\cal T}_{h}} {\cal F}_E(v_{|E}) \,.
\end{aligned}
$$
We recall that the assumptions on $\mu$ and \eqref{eq:prop-stabform} imply the bilinear form $a_h(u,v)$ to be uniformly coercive and continuous in the $H^1$-norm (see e.g. \cite{BBMR16}): there exist constants $0 < \mu_\star \leq \mu_0$ and $A_\mu >0$ independent of $h$ such that for any $u,v \in V_h$
\begin{equation}\label{eq:coerc-ah}
\mu_\star | v |_{1,\Omega_{h}}^2 \leq a_h(v,v), \qquad |a_h(u,v)| \leq A_\mu | u |_{1,\Omega_{h}} | v |_{1,\Omega_{h}}.
\end{equation}
Indeed, it is easy to check that the presence of the more general  operator $\PiNew_E$ can be handled by trivial modifications of the classical proof (also exploiting that, from the definition of $\Pi^\nabla_E$, it holds $|v - \Pi^\nabla_E v |_{H^1(E)} \le |v - \PiNew_E v |_{H^1(E)}$ for all $E \in \mesh$ and $v \in H^1(E)$).

Finally, we approximate the cumulative bilinear form ${\cal B}$ defined in \eqref{eq:form-B2} by setting for all $u,v \in V_h$
$$
{\cal B}_h(u, v):= a_h(u,v)+b_h(u,v)+c_h(u,v)\,.
$$
Using the skew-symmetry of the form $b_h$ together with inequalities \eqref{eq:assumptions-coefficients}, which by assumption holds on the extended domain $\wt{\Omega}$ as well, we get the coercivity of the form ${\cal B}_h$ in $V_h$, namely
\begin{equation}\label{eq:discrete-coercivity}
{\cal B}_h(v_h, v_h) \geq a_h(v_h, v_h) \geq \mu_\star|v_h|_{1,\Omega_h}^2, \qquad \forall v_h \in \Vmesh. 
\end{equation}
Then, Problem \eqref{eq:cpb} is approximated as follows.
\begin{definition}[Galerkin discretization]\label{def:Galerkin}
Let $u_h \in \Vmesh$ be the solution of the variational problem
\begin{equation}\label{eq:Galerkin}
{\cal B}_h(u_h, v_h) = {\cal F}_h(v_h) \qquad \forall v_h \in \Vmesh.
\end{equation}
\end{definition}

\section{Stability and convergence analysis under abstract assumptions}
\label{stability-convergence-abstract}
Numerical stability is granted by the following result.
\begin{proposition}[stability]\label{prop:stability}
Let $u_h \in \Vmesh$ be the solution of \eqref{eq:Galerkin}. Then
\begin{equation}\label{eq:stab-Galerkin}
|u_h|_{1,\Omega_h} \leq \frac{\tilde{C}_P}{\mu_\star} \|f\|_{0,\wt{\Omega}},
\end{equation}
where $\mu_\star$ is given in \eqref{eq:coerc-ah}, whereas $\tilde{C}_P>0$ is the Poincar\'e constant of the extended domain $\wt{\Omega}$.
\end{proposition}
\proof
 We use inequality \eqref{eq:discrete-coercivity} with $v_h=u_h$ in the right-hand side of \eqref{eq:Galerkin}.
On the other hand, recalling that $\Pi_E^\nabla = \Pi_E^0$ in each $\VE$, we have
\begin{equation*}
\begin{split}
{\cal F}_h(u_h) &= \sum_{ E \in {\cal T}_h} \int_E f \, \Pi_E^0 u_h = \sum_{ E \in {\cal T}_h} \int_E \Pi_E^0 f \, u_h \\
&\, \leq 
\|f\|_{0,\Omega_h}\|u_h\|_{0,\Omega_h} \leq 
\|f\|_{0,\wt{\Omega}}\, \|u_h\|_{0,\wt{\Omega}}
\leq \tilde{C}_P \|f\|_{0,\wt{\Omega}} \, |u_h|_{1,\wt{\Omega}}
= \tilde{C}_P \|f\|_{0,\wt{\Omega}} \, |u_h|_{1,\Omega_h},
\end{split}
\end{equation*}
whence the result.
\endproof

From now on, we will focus on the convergence analysis. To this end, we make a second crucial assumption on the VEM interpolation operator $\IVh$ defined above. Let us introduce the notation 
\begin{equation}\label{eq:domain-error-2}
\gamma := \begin{cases} 1 & \text{if there are no reentrant corners in }\partial\Omega, \\
\bar{\gamma}=\min_{1 \leq j \leq J} \gamma_j <1 & \text{if } \bz_1, \dots, \bz_J \text{ are the reentrant corners of } \partial\Omega.
\end{cases}
\end{equation}
\begin{assumption}[interpolation error]\label{ass:interp} 
Let $u$ be the solution of the Poisson problem \eqref{eq:cpbr} for some forcing $g \in L^2(\Omega)$.
The interpolation operator $\IVh$ satisfies the following error estimate:  there exists a constant $C_\gamma>0$ independent of $u$ and $\mesh$ such that
\begin{equation}\label{eq:interp-err}
h^{-1}\|u-\IVh u \|_{0,\Omega_h} +  | u-\IVh u |_{1,\Omega_h} \leq C_\gamma h^\gamma \enorm{u}\,,
\end{equation}
where the norm $\enorm{u}$ is defined in \eqref{eq:def-norm-u} and satisfies \eqref{eq:reg-bound-u}.
\end{assumption}
\noindent In Sect. \ref{sec:checking}  we will establish this result for the case of a 2D Cartesian background mesh.

We also require the following polynomial approximation result. 
\begin{assumption}[polynomial approximation]\label{ass:polapprox} 
Let $u$ be as in Assumption \ref{ass:interp}.
There exists a constant $C_\gamma'>0$ independent of $u$ and $\mesh$, and a piecewise linear (discontinuous) polynomial function ${\cal P}_h^1 u$ on $\mesh$ such that
\begin{equation}\label{eq:polapprox-err}
h^{-1}\|u-{\cal P}_h^1 u \|_{0,\Omega_h} +  | u-{\cal P}_h^1 u |_{1,\Omega_h} \leq C_\gamma' h^\gamma \enorm{u}\,.
\end{equation}
On the elements $E$ where $\PiNew_E \not= \Pi^\nabla_E$ we require that the restriction ${\cal P}^1_h u|_E$ falls in the subspace ${\cal S}_1(E)$ introduced in \eqref{eq:def-Pnew}.
\end{assumption}

Differently from Assumptions \ref{ass:stab-form} and \ref{ass:interp}, the verification of Assumption \ref{ass:polapprox} is much simpler and will be addressed in Remarks \ref{rem:polapprox:H2} and \ref{rem:polapprox:sing}.

\subsection{Convergence in $H^1(\Omega_h)$}\label{sec:converg-H1}
Let $u$ be the solution of our original problem \eqref{eq:cpb}, extended to $\wt{\Omega}$ as described in Sect. \ref{sec:extension}. To proceed, we consider the extended equation \eqref{eq: extended-equation} (where for simplicity we keep neglecting the superscript $\widetilde{\ }$\,), we multiply it by $v \in H^1_0(\Omega_h)$, integrate over $\Omega_h$ and integrate by parts the diffusion term. Denoting by ${\cal B}_{\Omega_h}(u,v)$ the bilinear form defined as in \eqref{eq:form-B1} but with $\Omega$ replaced by $\Omega_h$, the (extended) exact solution $u$ satisfies
\begin{equation}\label{eq:ext-eqn}
{\cal B}_{\Omega_h}(u,v) = (f,v)_{\Omega_h} + (F, v)_{\delta \Omega_h} \qquad \forall v \in H^1_0(\Omega_h),
\end{equation}
where 
\begin{equation}\label{eq:def-delta-Omegah}
\delta\Omega_h := \Omega_h \setminus \Omega \subset \delta\Omega.
\end{equation}
Note that we have used the property $F \equiv 0$ in $\Omega$. Then,
$\IVh u \in \Vmesh$ satisfies
\begin{equation}\label{eq:equation-Ihu}
{\cal B}_h(\IVh u,v_h) = (f,v_h)_{\Omega_h} + (F, v_h)_{\delta \Omega_h} + {\cal B}_h(\IVh u,v_h) -{\cal B}_{\Omega_h}(u,v_h) \qquad \forall v_h \in \Vmesh.
\end{equation}
On the other hand, the Galerkin solution satisfies
\[
{\cal B}_h(u_h, v_h) = (f,v_h)_{\Omega_h} + {\cal F}_h(v_h) -(f,v_h)_{\Omega_h} \qquad \forall v_h \in \Vmesh.
\]
By subtracting this equation from \eqref{eq:equation-Ihu}, we see that the discrete error $e_h := \IVh u - u_h \in \V_h$ satisfies the equations
\begin{equation}\label{eq:error-eq}
{\cal B}_h(e_h, v_h) = {\cal E}_h(v_h) \qquad \forall v_h \in \Vmesh,    
\end{equation}
where the consistency error on the right-hand side is made up of three contributions
\begin{equation}
  \begin{split}
 {\cal E}_h(v_h) &= \underbrace{(F, v_h)_{\delta \Omega_h}}_{:={\cal E}_{h,1}(v_h)} - (\underbrace{{\cal B}_{\Omega_h}(u,v_h)- {\cal B}_h(\IVh u,v_h)}_{:={\cal E}_{h,2}(v_h)})
 + \underbrace{(f,v_h)_{\Omega_h} - {\cal F}_h(v_h)}_{:={\cal E}_{h,3}(v_h)}
  \end{split}
\end{equation}
corresponding to the approximation of the domain, the bilinear form, and the forcing term.

We now provide bounds on the three addends ${\cal E}_{h,i}$.

\begin{proposition}[domain error]\label{prop:domain-error}
There exists a constant $C_D>0$ independent of $u$ and $\mesh$ such that
\begin{equation}\label{eq:domain-error}
 | {\cal E}_{h,1}(v_h) | \leq C_D h^\gamma {\cal N}_D({\mu}, {\bb}, {\sigma},{f}) | v_h |_{1,\Omega_h}  \qquad \forall v_h \in \Vmesh,
\end{equation}
where
\begin{equation}\label{eq:def-ND}
{\cal N}_D({\mu}, {\bb}, {\sigma}, {f}) := \left(\| {\mu} \|_{W^1_\infty(\wt{\Omega})} + \| {\bb} \|_{(L^\infty(\wt{\Omega}))^d} + \| {\sigma}\|_{L^\infty(\wt{\Omega})}\right)\|g\|_{0,\Omega} + \|{f} \|_{0,\wt{\Omega}}.
\end{equation}

\end{proposition}
\proof 
Preliminarly, it is worth investigating the structure of the set $\delta\Omega_h$. To this end, let us introduce the set ${\cal T}_h^b \subset {\cal T}_h$ of elements whose interior intersects $\delta\Omega_h$. If $E \in {\cal T}_h^b$ does not contain a break point in its interior, then $E \cap \partial\Omega$ is a ${\cal C}^2$ arc (or a union of arcs), hence $E\cap \delta\Omega_h$ is a region of width $O(h^2)$ (see, e.g., elements (A)-(D) in Fig. \ref{fig:geometries-1}). Conversely, if $E \in {\cal T}_h^b$ contains a break point in its interior, then $E\cap \delta\Omega_h$ may have width $O(h)$ (see, e.g., element (E) in Fig. \ref{fig:geometries-1};
note that the break point in this figure has a re-entrant corner). It may exists an element $E\in {\cal T}_h^b$ containing in its interior a break point with acute corner, but this is a non-asymptotic situation, namely after enough mesh refinements that break point will sit on $\partial\Omega_h$. For this reason, in the sequel we will only consider the case of break points with re-entrant corners.

It is also important to identify the regions in $\delta\Omega_h$ where $u$ (and consequently $F$) is smooth or not. At first, we define a region $R_A \subseteq \delta\Omega_h$ whose distance from any break point is $\geq A$, where $A>0$ is a fixed constant independent of $h$. Next, we fix a constant $a$ close to $1$ satisfying $ah_0 < A$, and we consider the region $R_{a,j} \subseteq \delta\Omega_h$ of the points with distance from $\bz_j$ smaller than $A$ but larger than $ah$. Finally, we let $R_{h,j}$ be the intersection of $\delta\Omega_h$ with the neighborhood of $\bz_j$ of radius $ah$.

In this proof, for better clarity we prefer to keep the notation $\wt{\ }$ to denote quantities living in the extended domain $\wt{\Omega}$. At first, recalling \eqref{eq:split-u-ext}, we observe that 
\begin{equation*}
\begin{split}
F  &= -\wt{\mu} \Delta \wt{u} + (\wt{\bb} - \nabla \wt{\mu})\cdot \nabla \wt{u} + \wt{\sigma} \wt{u} - \wt{f} \\
& = -\wt{\mu} \Delta \wt{\zeta} + \big( - \wt{\mu} \Delta \wt{\psi} +  (\wt{\bb} - \nabla \wt{\mu})\cdot \nabla \wt{u} + \wt{\sigma} \wt{u} - \wt{f} \big)
 =: -\wt{\mu} \Delta \wt{\zeta} + \phi,
\end{split}
\end{equation*}
with $\phi \in L^2(\wt{\Omega})$ and $\wt{\mu} \Delta \wt{\zeta} \in L^p(\wt{\Omega})$, $p$ being defined in Proposition \ref{prop:reg-u}.

Let $E \in {\cal T}_h^b$. In bounding the quantity, we distinguish three cases.

\smallskip
\noindent \emph{Case 1: $E \cap \delta\Omega_h$ is contained in $R_A$}. Then, $\wt{u}\in H^2(E)$ and $F \in L^2(E)$, with uniformly bounded norms.
Using the Poincar\'e inequality for $v_h$, and the fact that $E \cap \delta\Omega_h$ has width $O(h^2)$, we have
\begin{equation}\label{eq:bound-skin-1}
|(F,v_h)_{E \cap \delta\Omega_h}| \lesssim h_E^2 \| F \|_{0,E} | v_h|_{1,E}.
\end{equation}

\smallskip
\noindent \emph{Case 2: $E \cap \delta\Omega_h$ is contained in $R_{a,j}$ for some reentrant corner $\bz_j$}. Again, \eqref{eq:bound-skin-1} is valid since $F \in L^2(E)$, however this norm may blow-up while $E$ gets closer and closer to $\bz_j$, due to the presence of the term $\wt{\mu} \Delta \wt{\zeta}$. Recalling the structure of the singular part of $\wt{u}$, this norm can be controlled as follows
\[
\| F \|_{0,E} \lesssim \| \phi \|_{0,E} + |\lambda_j| \left( \int_{ah}^A r^{2(\gamma_j -2)} r {\rm d}r \right)^{1/2}
\lesssim \ \| \phi \|_{0,E} + |\lambda_j| \,h_E^{\gamma_j-1},
\]
whence 
\begin{equation}\label{eq:bound-skin-2}
|(F,v_h)_{E \cap \delta\Omega_h}|  \lesssim h_E^{\gamma_j+1} (\| \phi \|_{0,E} + |\lambda_j|) | v_h|_{1,E}.
\end{equation}

\smallskip
\noindent \emph{Case 3: $E \cap \delta\Omega_h$ intersects $R_{h,j}$ for some reentrant corner $\bz_j$}. In this case, we just exploit the $L^p$-smoothness of $F$ in $E$, so we use the bound
\[
|(F,v_h)_{E \cap \delta\Omega_h}| \leq \| F \|_{L^p(E)}\| v_h\|_{L^q(E)}, \qquad \tfrac1p+\tfrac1q=1,
\]
with
\begin{equation}\label{eq:skin-52}
\begin{split}
\| F \|_{L^p(E)} &\lesssim \| \phi \|_{L^p(E)} + |\lambda_j| \left( \int_0^{ah} r^{p(\gamma_j -2)} r {\rm d}r \right)^{1/p}\\
&\lesssim h_E^{\frac2p-1}\| \phi \|_{0,E} + |\lambda_j| \,h_E^{\gamma_j-2 +\frac2p}
\ \lesssim \ (\| \phi \|_{0,E} + |\lambda_j|) \,h_E^{\gamma_j-2 +\frac2p}.
\end{split}
\end{equation}
and the scaled Poincar\'e inequality for $v_h$
\[
\| v_h\|_{L^q(E)} \lesssim h_E^{\frac12+\frac1q} | v_h|_{1,E}.
\]
Since $\gamma_j -2 +\frac2p + \frac12 +\frac1q > \gamma_j$,
we get
\begin{equation}\label{eq:bound-skin-3}
|(F,v_h)_{E \cap \delta\Omega_h}|  \lesssim h_E^{\gamma_j} (\| \phi \|_{0,E} + |\lambda_j|) | v_h|_{1,E}.
\end{equation}
Putting together the bounds for the three cases yields the thesis.
\endproof

\begin{proposition}[error on bilinear and linear forms]\label{prop:forms-error} Assume that $\bb \in (W^1_\infty(\wt{\Omega}))^d$.
There exist constants $C_B>0$ and $C_F>0$ independent of $u$ and $\mesh$ such that for all $v_h \in \Vmesh$
\begin{align}
 | {\cal E}_{h,2}(v_h) | &\leq C_B h^\gamma {\cal N}_B(\mu,\bb, \sigma) )\enorm{u} | v_h |_{1,\Omega_h}, \\
 | {\cal E}_{h,3}(v_h) | & \leq C_F h \, \| f \|_{0,\Omega_h}\, | v_h |_{1,\Omega_h},
\end{align} 
with 
${\cal N}_B({\mu}, {\bb}, {\sigma}) := \left(\| {\mu} \|_{W^1_\infty(\wt{\Omega})} + \| {\bb} \|_{(W^1_\infty(\wt{\Omega}))^d} + \| {\sigma}\|_{L^\infty(\wt{\Omega})}\right)$.
\end{proposition}
\proof
The bound for term ${\cal E}_{h,2}(v_h)$ is derived easily by suitable manipulations, orthogonality plus approximation properties of the projection operators and the triangle inequality. The interested reader can check, for example, the derivations in \cite{BBMR16}, taking into account that the polynomial consistency of the method is now restricted to the piecewise polynomials $\{q  \in L^2(\Omega) \, : \, q|E \in {\cal S}_1(E) \ \forall E \in \mesh \}$. 

Concerning ${\cal E}_{h,3}(v_h)$, the estimate follows immediately from 
\[
{\cal E}_{h,3}(v_h) = \sum_{E \in {\cal T}_h} \int_E f (v_h - \Pi^\nabla_E v_h)
\]
and the bound $\|v_h - \Pi^\nabla_E v_h\|_{0,E} \lesssim h_E |v_h|_{1,E}$.
\endproof

We are ready to estimate the error between $u$ and $u_h$ in the energy norm. We recall that $u$ denotes the solution of Problem \eqref{eq:cpb}, possibly extended from $\Omega$ to $\wt{\Omega}$ as detailed in Sect. \ref{sec:extension}.

\begin{theorem}[Galerkin error in energy]\label{th:energy-error}
Let Assumptions \ref{ass:stab-form}, \ref{ass:interp} and \ref{ass:polapprox} be valid. Suppose that data satisfy $\mu \in W^1_\infty(\wt{\Omega})$, $\bb \in (W^1_\infty(\wt{\Omega}))^d$, $\sigma \in L^\infty(\wt{\Omega})$, and $f \in L^2(\wt{\Omega})$. 
There exists a constant $C_G>0$, depending upon $\mu_0$, $\mu_\star$ and the norms of $\mu$, $\bb$, $\sigma$ in these spaces but independent of $u$ and $\mesh$, such that the Galerkin solution $u_h$ defined in \eqref{eq:Galerkin} satisfies
\begin{equation}\label{eq:errorGal-H1}
| u-u_h |_{1,\Omega_h} \leq C_G h^\gamma \| f \|_{0,\wt{\Omega}}\,.
\end{equation}
\end{theorem}
\proof 
Apply the triangle inequality to $u-u_h = (u-\IVh u)+e_h$, invoking the interpolation error estimate \eqref{eq:interp-err} for the first addend. Next,
choose $v_h=e_h$ in \eqref{eq:error-eq}, use the coercivity of the form ${\cal B}_h$ (see \eqref{eq:discrete-coercivity}) to get the inequality
\[
\mu_\star |e_h|_{1,\Omega_h}^2 \leq |{\cal E}_{h}(v_h)|,
\]
and apply Propositions \ref{prop:domain-error} and \ref{prop:forms-error} to bound the right-hand side. Finally, bound the norm $\enorm{u}$ by inequality \eqref{eq:reg-bound-u}, and the norm $\| g \|_{0,\Omega}$ by inequality \eqref{eq:bound-g}. 
\endproof

\subsection{Convergence in $L^2(\Omega_h)$}\label{sec:converg-L2}
Hereafter, we derive an estimate of the error $\|u-u_h\|_{0,\Omega_h}$, by applying a classical Aubin-Nitsche duality argument. To this end, we define 
\begin{equation}
e := \begin{cases} u-u_h & \text{ in } \overline{\Omega}_h, \\
0 & \text{ in } \Omega \setminus \overline{\Omega}_h,
\end{cases}
\end{equation}
which satisfies $\| e \|_{0,\Omega} \leq \| e \|_{0,\Omega_h}$.
Then, we solve the adjoint variational problem
\begin{equation}
z \in H^1_0(\Omega) \quad : \quad   {\cal B}(v,z) = (e,v)_{0,\Omega} \qquad \forall v \in H^1_0(\Omega).  
\end{equation}
It is easily seen that $z$ is the solution of the Dirichlet problem
\begin{subequations}\label{eq:cpbd}
  \begin{align}
    - \nabla \cdot (\mu \nabla z) - \bb\cdot \nabla z + (\sigma- \nabla \cdot \bb)z &= e \quad \mathrm{in~}\Omega,\label{cpbd:1}\\
    z&=0 \quad \mathrm{on~}\partial\Omega,\label{cpbd:2}
  \end{align}
\end{subequations}
hence, in particular, $z$ solves in $\Omega$ the Poisson equation $-\Delta z = \eta$ with  $\eta:= {\mu}^{-1}( e + (\nabla \mu +\bb)\cdot\nabla z - (\sigma - \nabla \cdot \bb) z) \in L^2(\Omega)$. It follows that $z$ (and its extension to $\wt{\Omega}$) share with $u$ the same features in terms of structure and smoothness, as detailed in Sects. \ref{sec:solution-structure} and \ref{sec:extension}. In particular, it holds
\begin{equation}\label{eq:reg-z}
\enorm{z} \lesssim \| \eta \|_{0,\Omega} \lesssim \|e \|_{0,\Omega} \leq \|e \|_{0,\Omega_h}.
\end{equation}

A few words about extensions to the computational domain $\Omega_h$ are in order.  Let $\wt{z}$ be an extension of $z$ which preserves regularity, in analogy with the extension of $u$ in \eqref{eq:split-u-ext}; thus, $\wt{z} \in W^2_p(\Omega_h)$ with $p$ given in Proposition \ref{prop:reg-u}, and $\| \wt{z} \|_{W^2_p(\Omega_h)} \lesssim \| z \|_{W^2_p(\Omega)}$. Setting
\[
\hat{\vartheta} := - \nabla \cdot (\widetilde{\mu}\, \nabla \widetilde{z}) - \widetilde{\bb}\cdot \nabla \widetilde{z} + (\widetilde{\sigma}-\nabla \cdot \wt{\bb}) \widetilde{z} \, \in L^p(\Omega_h),
\]
the extension $\wt{z}$ satisfies in $\Omega_h$ the equation
\[
- \nabla \cdot (\widetilde{\mu}\, \nabla \widetilde{z}) - \widetilde{\bb}\cdot \nabla \widetilde{z} + (\widetilde{\sigma}-\nabla \cdot \wt{\bb}) \widetilde{z} =
e + D,
\]
with $D:=\hat{\vartheta}-e \in L^p(\Omega_h)$, identically vanishing in $\Omega$. 

For simplicity, from now on let us drop the symbol $\wt{\ }$ from all functions. If we multiply this equation by any $v \in H^1(\Omega_h)$, integrate over $\Omega_h$ and integrate by parts, we obtain
\begin{equation*}
\begin{split}
&\int_{\Omega_h} {\mu} \, \nabla {z} \cdot \!\nabla v - \frac12 \int_{\Omega_h} ({\bb} \cdot \!\nabla {z}) v +  \frac12 \int_{\Omega_h} ({\bb} \cdot \!\nabla v) {z} + 
\int_{\Omega_h} (\sigma - \tfrac12 \nabla \cdot \bb) {z} \, v  \\
& \qquad \qquad \qquad + \int_{\partial\Omega_h}(\tfrac12 {\bb} \cdot \bn \, {z} - {\mu} \nabla {z} \cdot \bn)v
= \int_{\Omega_h} {e} \,v + \int_{\delta \Omega_h} D \,v \,.
\end{split} 
\end{equation*}
Setting $S:=(\tfrac12 {\bb} {z} - {\mu} \nabla {z}) \cdot \bn$ and choosing $v=e$ (note that $e$ need not vanish on $\partial\Omega_h$, explaining our choice of test functions in $H^1(\Omega_h)$), we get the following expression for the $L^2$-norm of $e$:
\begin{equation}\label{eq:identity-e}
\| e \|_{0,\Omega_h}^2 ={\cal B}_{\Omega_h}(e,z) - (D, e)_{0,\delta \Omega_h} +(S,e)_{0,\partial\Omega_h},
\end{equation}
where as usual ${\cal B}_{\Omega_h}$ denotes the bilinear form defined as in \eqref{eq:form-B2} but with $\Omega$ replaced by $\Omega_h$.

To proceed, we express ${\cal B}_{\Omega_h}(e,z)$ as done in 
\cite[Theorem 5.2]{BBMR16}, using the equations satisfied by $u$ (see \eqref{eq:ext-eqn}) and by $u_h$ (see \eqref{eq:Galerkin}). We get
\begin{equation}
\begin{split}
{\cal B}_{\Omega_h}(e,z) &= \underbrace{{\cal B}_{\Omega_h}(u-u_h,z-\IVh z)}_{T_1} + \underbrace{{\cal B}_h(u_h,\IVh z) - {\cal B}_{\Omega_h}(u_h,\IVh z)}_{T_2}  \\
& \qquad  + \underbrace{(f,\IVh z)_{0,\Omega_h} - {\cal F}_h(\IVh z)}_{T_3} + \underbrace{(F, \IVh z)_{0,\delta\Omega_h}}_{T_4}
\end{split}
\end{equation}
Setting $T_5 := - (D, e)_{0,\delta \Omega_h}$ and $ T_6 :=(S,e)_{0,\partial\Omega_h}$, identity \eqref{eq:identity-e} implies
\begin{equation}\label{eq:split-norm-e}
\| e \|_{0,\Omega_h}^2 \leq \sum_{i=1}^6 |T_i|\,,
\end{equation}
and we are left with the task of bounding each addend. To this end, we observe that Assumption \ref{ass:interp} (interpolation error) applies to $z$ as well, yielding the error bound 
\begin{equation}\label{eq:interp-z}
h^{-1}\| z-\IVh z \|_{0,\Omega_h} +
     | z-\IVh z |_{1,\Omega_h} \leq C_\gamma h^\gamma \enorm{z} \lesssim h^\gamma \|e \|_{0,\Omega_h}\,.
\end{equation}

\medskip
\noindent \underline{\sl Bound on $|T_1|$.} Continuity of ${\cal B}_{\Omega_h}$ and estimates \eqref{eq:errorGal-H1} and \eqref{eq:interp-z} together with \eqref{eq:reg-z} yield
\begin{equation}\label{eq:bound-T1}
|T_1| \lesssim h^{2 \gamma} \|f \|_{0,\wt{\Omega}} \| e \|_{0,\Omega_h}\,.
\end{equation}
\medskip
\noindent \underline{\sl Bound on $|T_2|$.} Denote by $\Pi_{k,h}^0 u$ the $L^2$-orthogonal projection of $u$ upon the piecewise-$\mathbb{P}_k$ functions on $\mesh$, and observe that ${\cal B}_{\Omega_h}(\Pi_{k,h}^0 u,v_h)$ is meaningful since ${\cal B}_{\Omega_h}$ is defined elementwise on $\mesh$. 
Then, using the identity
\[
T_2 = [{\cal B}_{\Omega_h}(u_h-\Pi_{1,h}^0 u,\IVh z)- {\cal B}_h(u_h-\Pi_{1,h}^0 u,\IVh z)] +
[{\cal B}_{\Omega_h}(\Pi_{1,h}^0 u,\IVh z)- {\cal B}_h(\Pi_{1,h}^0 u,\IVh z)]
\]
and proceeding as in the proof of \cite[Theorem 5.2]{BBMR16}, one gets again
\begin{equation}\label{eq:bound-T2}
|T_2| \lesssim h^{2 \gamma} \|f \|_{0,\wt{\Omega}} \| e \|_{0,\Omega_h}\,.
\end{equation}
\medskip
\noindent \underline{\sl Bound on $|T_3|$.} Let $\Pi_h^\nabla v$ be defined by $(\Pi_h^\nabla v)_{|E}= \Pi_E^\nabla z_{|E}$ for all $E\in \mesh$. Then, recalling that the enhanced $\Pi_h^\nabla$ operator is $L^2$-stable as it locally coincides with the $L^2$-projection operator, one has 
\[
|T_3| = |(f, \IVh z - \Pi_h^\nabla \IVh z)_{0,\Omega_h}| \lesssim \|f\|_{0,\Omega_h} \|\IVh z - \Pi_{0,h}^0 z\|_{0,\Omega_h},
\]
from which (recalling Assumption \ref{ass:polapprox}) the bound
\begin{equation}\label{eq:bound-T3}
|T_3| \lesssim h^{1+\gamma} |f |_{0,\wt{\Omega}} \| e \|_{0,\Omega_h}\
\end{equation}
easily follows. 

\medskip
\noindent \underline{\sl Bound on $|T_4|$.} We go back to the proof of Proposition \ref{prop:domain-error} (domain error), in which we set $v_h=\IVh z$, and observe that in the regions $R_A$ or  $R_{a,j} \subseteq \delta\Omega_h$ therein defined, inequalities \eqref{eq:bound-skin-1} or \eqref{eq:bound-skin-2} imply an error decay of order at least $2\gamma$, provided $\IVh z$ has bounded energy norm. Such boundedness follows from the triangle inequality and \eqref{eq:interp-z}, which yield $|\IVh z|_{1,\Omega_h} \lesssim \enorm{z} \lesssim \|e\|_{0,\Omega_h}$.

It remains to consider the regions $R_{h,j}$ near a re-entrant break point. Let $E\in {\cal T}_h^b$ have non-empty intersection with some $R_{h,j}$. By \eqref{eq:bound-skin-2} we know that $\| F \|_{L^p(E)}\lesssim \ (\| \phi \|_{0,E} + |\lambda_j|) \,h_E^{\gamma_j-2 +2/p}$  provided $p$ satisfies \eqref{eq:cond-p-gamma}. To bound the $L^q$-norm of $\IVh z$ (with $1/q+1/p=1$),
we use the inequality $\| \IVh z \|_{L^q(E)} \leq |E|^{1/q} \| \IVh z \|_{L^\infty(E)}$, together with a generalized maximum principle for virtual element functions, presented in Sect. \ref{sec:maximum-principle}, which gives
\[
\| \IVh z \|_{L^\infty(E)} \lesssim \| \IVh z \|_{L^\infty(\partial E)} \leq \max_{\nu \text{ vertex of } E} |z(\nu)| \leq \| z \|_{L^\infty(E)}.
\]
If $z = \psi^\star + \lambda_j^\star \zeta_j$ is the decomposition of $z$ into regular and singular part, for $h$ small enough we have $z= \lambda_j^\star \zeta_j$ in $E$, whence 
$\| z \|_{L^\infty(E)} \lesssim |\lambda_j^\star| h_E^{\gamma_j} \lesssim h_E^{\gamma_j} \enorm{z}$. This gives
\[
\| \IVh z \|_{L^q(E)} \lesssim h_E^{\gamma_j+2/q} \enorm{z} = h_E^{\gamma_j+2-2/p} \enorm{z},  
\]
from which we get $|(F,\IVh z)_{0,R_{h,j}}| \lesssim h^{2\gamma_j} \enorm{u} \enorm{z}$.  In conclusion, we obtain
\[
|T_4| \lesssim h^{2\gamma} 
\|f \|_{0,\wt{\Omega}} \| e \|_{0,\Omega_h}\,.
\]

\medskip
\noindent \underline{\sl Bound on $|T_5|$.} At first note that structure and properties of $D$ are similar to those of $F$; hence, we can apply Proposition \ref{prop:domain-error} (domain error) to it, after replacing $u$ with $z$ and $f$ with $e$. Splitting $e = \IVh e + (e-\IVh e) = (\IVh u - u_h) + (u-\IVh u)$, and using \eqref{eq:interp-err} and \eqref{eq:errorGal-H1}, we get
\[
|(D,\IVh e)_{0,\delta\Omega_h}| \lesssim h^\gamma \| e \|_{0,\Omega_h} |\IVh u - u_h|_{1,\Omega_h} \lesssim h^{2\gamma} \|f \|_{0,\wt{\Omega}} \| e \|_{0,\Omega_h}\,.
\]
To deal with $|(D,e-\IVh e)_{0,\delta\Omega_h}| = |(D,u-\IVh u)_{0,\delta\Omega_h}|$, we follow again the lines of the proof of Proposition \ref{prop:domain-error}. In the region $R_A \subseteq \delta\Omega_h$, $D$ is in $L^2$ and $u$ is in $H^2$, whence, using again \eqref{eq:interp-err}, we get
\[
|(D,u-\IVh u)_{0,R_A}| \lesssim \| D\|_{0,R_A} \| u-\IVh u\|_{0,R_A}
\lesssim h^2 \| D\|_{0,R_A} | u |_{2,\omega(R_A)} 
\lesssim h^2 \|f \|_{0,\wt{\Omega}} \| e \|_{0,\Omega_h}\,,
\]
where $\omega(R_A)$ is a neighborhood of $R_A$ of thickness $O(h)$.
In a region $R_{a,j}$ close to a re-entrant corner $\bz_j$, the $L^2$-norm of $D$ and the $H^2$-norm of $u$ are finite, but they blow up as $h^{\gamma_j-1}$. Thus,
\[
|(D,u-\IVh u)_{0,R_{a,j}}| \lesssim h^{2\gamma_j} \|f \|_{0,\wt{\Omega}} \| e \|_{0,\Omega_h}\,.
\]
Finally, in a region $R_{h,j}$, we write $|(D,u-\IVh u)_{0,R_{h,j}}| \leq |(D,u)_{0,R_{h,j}}|+ |(D,\IVh u)_{0,R_{h,j}}|$ and proceed as in the bound of $|T_4|$ above, getting $
|(D,u-\IVh u)_{0,R_{h,j}}| \lesssim h^{2\gamma_j} \enorm{u} \enorm{z}$. 

Collecting all these results, we arrive again at the bound
\[
|T_5| \lesssim h^{2\gamma} 
\|f \|_{0,\wt{\Omega}} \| e \|_{0,\Omega_h}\,.
\]

\medskip
\noindent \underline{\sl Bound on $|T_6|$.}
At first, note that $T_6=(S,u)_{0,\partial\Omega_h}$ since $u_h$ vanishes on $\partial\Omega_h$. In analogy with the partition of $\delta\Omega_h$ introduced in the proof of Proposition \ref{prop:domain-error} (domain error), we can split $\partial\Omega_h$ into subsets $R_A^\partial$, $R_{a,j}^\partial$ and $R_{h,j}^\partial$, depending upon the distance from a break point with re-entrant corner (the notation is self-explanatory).
Points in $R_A^\partial \cup R_{a,j}^\partial$ have distance $O(h^2)$ from $\partial\Omega$, where $u$ vanishes. Hence, by Poincar\'e inequality we have $\| u\|_{0,R_A^\partial \cup R_{a,j}^\partial} \lesssim h^2 |u|_{1,\Omega_h}$. On the other hand, there exists a neighborhood $N_A$ of $R_A^\partial$ in $\wt{\Omega}$, such that $z \in H^2(N_A)$ and consequently $\|S\|_{0,R_A^\partial} \lesssim \| z \|_{2,N_A}$ with implied constant independent of $h$. Thus, $|(S,u)_{0,R_A^\partial}| \lesssim h^2 |u|_{1,\Omega_h} \enorm{z}$.

Similarly, one can find a neighborhood $N_{a,h}$ of $R_{a,j}^\partial$ in $\wt{\Omega}$, which extends up to a $O(1)$-distance in the direction perpendicular to $R_{a,j}^\partial$, such that $z \in H^2(N_{a,j})$ and $\|S\|_{0,R_{a,j}^\partial} \lesssim \| z \|_{2,N_{a,j}}$, again with implied constant independent of $h$. However in this region the $H^2$-norm of $z$ blows up as $h^{\gamma_j-1}$, which yields $|(S,u)_{0,R_{a,j}^\partial}| \lesssim h^{1+\gamma_j} |u|_{1,\Omega_h} \enorm{z}$. 

At last, consider the set $R_{h,j}^\partial$, which is at distance $O(h)$ from the break point $\bz_j$. Then, for any $\bx\in R_{h,j}^\partial$ we have
\[
u(\bx) =\lambda_j \, \text{dist}(\bx,\bz_j)^{\gamma_j}, \qquad
|S(\bx)| \lesssim |\lambda_j^\star| \, \text{dist}(\bx,\bz_j)^{\gamma_j-1},
\]
whence $|(S,u)_{0,R_{h,j}^\partial}| 
\lesssim h^{2\gamma_j-1} |R_{h,j}^\partial|\, |\lambda_j|\, |\lambda_j^\star| 
\lesssim h^{2\gamma_j} \enorm{u} \enorm{z}$. In conclusion, we have proven the bound
\[
|T_6| \lesssim h^{2\gamma} 
\|f \|_{0,\wt{\Omega}} \| e \|_{0,\Omega_h}\,.
\]

Building on the splitting \eqref{eq:split-norm-e} and the bounds on the addends $|T_i|$ obtained above, we have proven the following optimal error bound for the Galerkin error in $L^2$.

\begin{theorem}[Galerkin error in $L^2$]\label{th:L2-error}
Let all the hypotheses stated in Theorem \ref{th:energy-error} (error in energy) be valid. 
There exists a constant $\bar{C}_G>0$, depending upon $\mu_0$, $\mu_\star$ and the norms of $\mu$, $\bb$, $\sigma$ in these spaces but independent of $u$ and $\mesh$, such that the Galerkin solution $u_h$ defined in \eqref{eq:Galerkin} satisfies
\begin{equation} 
\| u-u_h \|_{0,\Omega_h} \leq \bar{C}_G h^{2\gamma} \| f \|_{0,\wt{\Omega}}\,.
\end{equation}
\end{theorem}

\section{A generalized maximum principle for virtual functions} \label{sec:maximum-principle}
\def\VhE{V_h(E)} 
\def\Cone{\mathsf{C}_1}  
\def\Ctwo{\mathsf{C}_2}  
\def\Cthr{\mathsf{C}_3}  
\def\Cfour{\mathsf{C}_4}  

Each function $w$ in the `plain' virtual space $\WE$ defined in \eqref{vem:basic} is harmonic in $E$, hence by the maximum principle it satisfies the inequality
\[
\| w \|_{L^\infty(E)} \leq \| w \|_{L^\infty(\partial E)}. 
\]
In this section, we prove that functions $v$ in the `enhanced' space $\VE$ defined in \eqref{vem:choice:2} satisfy a similar property, namely, their $L^\infty$-norm in $E$ can be bounded by the $L^\infty$-norm in $\partial E$ up to a multiplicative constant, which may depend on the shape of the polygon $E$. Moreover, we provide examples of classes of polygons for which this constant can be bounded independently of $E$.

We believe that the formulation of the generalized maximum principle that we present hereafter has an interest per se in the theory of Virtual Elements, beyond the specific application given in the present paper.

\begin{lemma}\label{lemma:CC}
Let  $\Cone,\Ctwo,\Cthr, \Cfour$ be the positive constants (possibly depending on the polygon $E$) defined as follows: 
\begin{enumerate}
    \item[(A1)] $\Cone=\|\widehat{v}\|_{L^\infty(E)}$, where $\widehat{v}\in H^1_0(E)$ satisfies $-\Delta \widehat{v}=1$ in $E$; 
    \item[(A2)] $\Ctwo$ is the smallest positive real number such that 
    \begin{equation}
        \|\Pi^\nabla w\|_{L^\infty(E)}\leq \Ctwo \|w\|_{L^\infty(\partial E)} \qquad \forall w \in \WE;
\end{equation}
\item[(A3)] $\Cthr$ is the smallest positive real number such that
\begin{equation}\label{aux_C_3}
    \|\Delta w \|_{L^2(E)}^2 \leq \Cthr \| \nabla w \|_{L^2(E)}^2
\end{equation}
for all $w\in H_0^1(E) ~\text{such~that}~ \Delta w \in \mathbb{P}_1(E)$.
\item[(A4)] $\Cfour$ is the smallest positive real such that
\begin{equation}\label{eq:falcao}
\| q_1 \|_{L^\infty(E)} \le \Cfour |E|^{-1/2} \| q_1 \|_{L^2(E)}
\qquad \forall q_1 \in \Pol_1(E) \, .
\end{equation}
\end{enumerate}
Then, there holds
\begin{equation}
    \| v \|_{L^\infty(E)} \lesssim [1+ \Cone\Cfour(1+\Ctwo)\Cthr]\, \| v \|_{L^\infty(\partial E)} \qquad \forall v \in \VE.
\end{equation}
\end{lemma}
\begin{proof}
Let us split $v=v_H+v_0$ where
\begin{eqnarray}
v_H\vert_{\partial E}=v\vert_{\partial E}, &&\quad \Delta v_H =0 ~\text{in~} E,\nonumber\\
   v_0\vert_{\partial E}=0,&&\quad\Delta v_0 = \Delta v ~\text{in~} E\nonumber. 
\end{eqnarray}
For $v_H$ the maximum principle holds, hence
\begin{equation}\label{aux:cc:0}
    \|v_H\|_{L^\infty(E)} \leq \|v\|_{L^\infty(\partial E)}.
\end{equation}
Therefore we concentrate on $v_0$. Setting $d=\|\Delta v\|_{L^\infty( E)}$, we have $ -d \leq -\Delta v \leq d~\text{in~} E $.
Let $\overline{v}\in H^1_0(E)$ be such that $-\Delta \overline{v}=d$ in $E$. Then, 
$ -\Delta(-\overline{v})\leq - \Delta v_0 \leq -\Delta \overline{v}~\text{in~}E$,
which gives 
$ -\overline{v}\leq v_0\leq \overline{v}~\text{in~}E$,
i.e., $ \|v_0 \|_{L^\infty(E)}\leq 
\| \overline{v}\|_{L^\infty(E)}$.
Write $\overline{v}=d \, \widehat{v}$, with $\widehat{v}$ defined in \emph{(A1)}.
Then, there holds
\begin{equation}\label{aux:cc:1}
    \|v_0\|_{L^\infty(E)}\leq d \|\widehat{v}\|_{L^\infty(E)}=\| \Delta v \|_{L^\infty(E)}\| \widehat{v}\|_{L^\infty(E)}=
    \Cone 
    \| \Delta v \|_{L^\infty(E)}.
\end{equation}

We now estimate $\| \Delta v \|_{L^\infty(E)}$. Let 
$\{p_j\}_{1\leq j \leq 3}$ be a basis in $\mathbb{P}_1(E)$, orthonormal in $L^2(E)$. Then, there 
exist $\alpha_j\in \mathbb{R}$ such that  $\Delta v = \Delta v_0=\sum_{j=1}^3 \alpha_j p_j$. For any $i\in\{1,2,3\}$ let $v_j\in H^1_0(E)$ solve
$$\Delta v_j=p_j~\text{in~}E,$$
hence $v_0=\sum_{j=1}^3 \alpha_j v_j$. The coefficients $\alpha_j$ are such that 
$$ \int_E v p_i = \int_E \Pi^\nabla v p_i\quad i=1,2,3,
$$
that is 
$$ \int_E v_0 p_i = \int_E \Pi^\nabla v p_i - \int_E v_H p_i =: - r_i\quad i=1,2,3,
$$
i.e.
$$ \sum_{j=1}^3 \alpha_j \int_E v_j p_i = -r_i \quad i=1,2,3.
$$
But
$$ \int_E v_j p_i = \int_E v_j \Delta v_i = -\int_E \nabla v_j \cdot \nabla v_i.$$
Introducing 
$$ 
S = \left (\int_E \nabla v_j \cdot \nabla v_i \right )_{1\leq i,j\leq 3}, \qquad
\underline{\alpha}=(\alpha_j)_{1\leq j \leq 3}, \qquad \underline{r}=(r_i)_{1\leq i \leq 3}, $$
we obtain the linear system
$ S \underline{\alpha}= \underline{r}$.
The matrix $S$ is symmetric and positive definite, since
$$ \underline{\beta}^T S \underline{\beta}= \|\nabla w\|^2_{L^2(E)}\quad\text{with~}w=\sum_{j=1}^3 \beta_j v_j.$$
In particular, if $\underline{\beta}$ is an eigenvector with eigenvalue $\lambda>0$
$$ \|\nabla w\|^2_{L^2(E)}= \underline{\beta}^T S \underline{\beta} =  \lambda \|\underline{\beta}\|^2_{\ell^2}= \lambda
\|\Delta w \|^2_{L^2(E)}$$
since $\{p_j\}$ is an orthonormal system in $L^2(E)$. Thus
$$\|S^{-1}\|_{2}=\max_{\lambda} \lambda^{-1} \leq \max_{w\in\text{span}\{v_1,v_2,v_3\}}\frac{\|\Delta w\|^2_{L^2(E)}}{\|\nabla w\|^2_{L^2(E)}}\leq \Cthr,
$$
where we employed  \emph{(A3)}, whence
$$\|\underline{\alpha}\|_{\ell^2} \leq \Cthr \|\underline{r}\|_{\ell^2}.$$
To estimate $\|\underline{r}\|_{\ell^2}$ we recall the definition of $r_i$ and we  observe
that 
\begin{equation}
    \left\vert
    \int_E \Pi^\nabla v p_i\right\vert
    \leq \|\Pi^\nabla v\|_{L^2(E)}\leq |E|^{1/2}\|\Pi^\nabla v\|_{L^\infty(E)}
   \leq |E|^{1/2}\Ctwo \|v\|_{L^\infty(\partial E)}
\end{equation}
where we used \emph{(A2)}.
On the other hand, using \eqref{aux:cc:0} we have
\begin{equation}
    \left\vert
    \int_E v_H p_i\right\vert
    \leq \|v_H\|_{L^2(E)}\leq |E|^{1/2} \|v_H\|_{L^\infty(E)} \leq |E|^{1/2} \|v\|_{L^\infty(\partial E)}.
\end{equation}
Summarizing, we have 
\begin{eqnarray}
    \|\Delta v\|_{L^2(E)}= \|\sum_{j=1}^3 \alpha_j p_j\|_{L^2(E)}=\|\underline{\alpha}\|_{\ell^2}\leq \Cthr \|\underline{r}\|_{\ell^2}\lesssim \Cthr|E|^{1/2}(1+\Ctwo)\|v\|_{L^\infty(\partial E)}.
\end{eqnarray} 
As $\Delta v\in \mathbb{P}_1(E)$, by inverse inequality we get 
$$ 
\| \Delta v\|_{L^\infty(E)} \le \Cfour |E|^{-1/2} \| \Delta v\|_{L^2(E)} \lesssim  \Cfour \Cthr(1+\Ctwo)\|v\|_{L^\infty(\partial E)}.
$$
Employing \eqref{aux:cc:1} we obtain 
$$ 
\|v_0\|_{L^\infty(E)}\lesssim \Cone \Cfour \Cthr(1+\Ctwo)\|v\|_{L^\infty(\partial E)}
$$
which yields the thesis after using the splitting $v=v_H+v_0$ and \eqref{aux:cc:0}.
\end{proof}

\begin{proposition}\label{prop:shapereg}
Let $E$ be a polygon that is star shaped with respect to a ball $B_1$ of radius $\rho_1$ and contained in a ball $B_2$ of radius $\rho_2$.
Then it exists a constant C, only depending on the ratio $\rho_2/\rho_1$, such that
$$
\| v \|_{L^\infty(E)} \le C \| v \|_{L^\infty(\partial E)} \qquad \forall v \in \VE \, .  
$$
\end{proposition}
\begin{proof}
We prove the result by estimating the four constants $\Cone, \Ctwo, \Cthr, \Cfour$ of Lemma \ref{lemma:CC}. In the present proof we will denote by $c_p$ a generic positive constant, possibly changing at different occurrences, which depends only on the ratio $\rho_2/\rho_1$.

Let $\hat{v}$ as in assumption \text{($A1$) of  Lemma \ref{lemma:CC}}. Let the function $\widetilde{v} \in H^1_0(B_2)$ defined by $-\Delta \widetilde{v} = 1$. This nonnegative function can be explicitly computed (it is a quadratic polynomial), and it is trivial to check that its maximum value is equal to $(\rho_2)^2/2$. Since $\widetilde{v} \ge 0$ on the whole $B_2$, it is clearly nonnegative on $\partial E$. Therefore, noting that $\Delta (\widetilde{v}-\hat{v})=0$, standard properties of harmonic functions entail $\widetilde{v} \ge \hat{v} \ge 0$ on $E$. Combining the above bounds immediately yields $\Cone \le (\rho_2)^2/2$.

Let now $v\in H^1(E)$ as in  assumption \text{($A2$)} of Lemma \ref{lemma:CC}. 
We recall that, by definition and some simple manipulations,
\begin{equation}\label{eq:kk:1}
\int_{\partial E} \Pi^\nabla v = \int_{\partial E} v \quad , \qquad
\nabla \Pi^\nabla v = |E|^{-1} \! \int_{E} \nabla v =  |E|^{-1} \! \int_{\partial E} v \, {\bf n_E} \, ,
\end{equation}
with ${\bf n}_E$ denoting the unit outward normal to the element boundary.
It can be checked that the assumptions of the current lemma yield $|\partial E| \le c_p |\partial B_1|$, we here omit the simple but tedious proof. Therefore, from the second identity in \eqref{eq:kk:1} we obtain
\begin{equation}\label{eq:kk:2}
\| \nabla \Pi^\nabla v \|_{L^\infty(E)} \le \frac{|\partial E|}{|E|} \| v \|_{L^\infty(\partial E)}
\le c_p \frac{|\partial B_1|}{|B_1|} \| v \|_{L^\infty(\partial E)} = 2 c_p (\rho_1)^{-1} \| v \|_{L^\infty(\partial E)}  \, .
\end{equation}
From the first identity in \eqref{eq:kk:1} it follows that it exists (at least) a point $\xi \in \partial E$ such that
$|\Pi^\nabla v(\xi)| \le \| v \|_{L^\infty(\partial E)}$. For any point $\eta \in E$ it clearly holds (recall that $\Pi^\nabla$ is an affine functions and, as such, can be extended on the whole $B_2$)
$$
|\Pi^\nabla v(\eta)| \le |\Pi^\nabla v(\xi)| + \| \eta - \xi \|_{\ell^2} \, \| \nabla \Pi^\nabla v \|_{L^\infty(E)}
\le \| v \|_{L^\infty(\partial E)} + 2 \rho_2 \| \nabla \Pi^\nabla v \|_{L^\infty(E)} \, ,
$$
which combined with \eqref{eq:kk:2} immediately yields $\Ctwo \le 1 + 4 c_p \, (\rho_2/\rho_1)$.

Let finally $w \in H^1_0(E)$ as in assumption \text{($A3$)} of Lemma \ref{lemma:CC}. Let $\varphi$ represent the unique quadratic function which vanishes on $\partial B_1$ and satisfies $\| \varphi \|_{L^\infty(B_1)}=1$; such function is clearly nonnegative in $B_1$.
Recalling that $\Delta w \in \Pol_1$ and employing  classical results for polynomials and standard scaling arguments combined with  an integration by parts yield
$$
\| \Delta w \|_{L^2(E)}^2 \le c_p \| \Delta w \|_{L^2(B_1)}^2 \le c_p \int_{B_1} \varphi (\Delta w)^2 
= c_p \int_{B_1} \nabla w \cdot \nabla (\varphi \Delta w) \le c_p \| \nabla w \|_{L^2(B_1)} \| \nabla (\varphi \Delta w) \|_{L^2(B_1)} \, .
$$
We now apply again a standard scaling argument for polynomials on balls and recall $\| \varphi \|_{L^\infty(B_1)}=1$, obtaining
$$
\| \Delta w \|_{L^2(E)}^2 \le c_p \| \nabla w \|_{L^2(B_1)} \rho_1^{-1} \| \Delta w \|_{L^2(B_1)} \, .
$$
Since $B_1 \subset E$, the above bound grants immediately $\| \Delta w \|_{L^2(E)} \le c_p \, \rho_1^{-1} \| \nabla w \|_{L^2(B_1)}$, so that $\Cthr \le c_p \, \rho_1^{-2}$.

Finally, the bound for $\Cfour$ is a standard inverse estimate for polynomials on shape-regular domains.
The proof of the lemma now follows from the four bounds above for $\Cone, \Ctwo, \Cthr, \Cfour$ combined with Lemma \ref{lemma:CC}.
\end{proof}

\begin{proposition}\label{prop:ani:C123}
Let $E$ be a trapezoidal element with height $h$, and bases lengths $\varepsilon_1, \varepsilon_2$, respectively. Let $h \ge \max\{\varepsilon_1,\varepsilon_2\}$. Then it exists a constant $C$ independent of $h, \varepsilon_1, \varepsilon_2$ such that
$$
\| v \|_{L^\infty(E)} \le C \| v \|_{L^\infty(\partial E)} \qquad \forall v \in \VE \, .  
$$
\end{proposition}
\begin{proof}
We again prove the result by estimating the three constants $\Cone, \Ctwo, \Cthr, \Cfour$ of Lemma \ref{lemma:CC}. It is clearly not restrictive to assume $\varepsilon_2 \ge \varepsilon_1$, as in the trapezoid with black edges shown in Figure \ref{fig:quad:X}. We denote by $Q$ the minimal rectangle containing $E$ (the gray rectangle in the same figure), thus of height $h$ and base length $\varepsilon_2$. In the following $(x,y)$ will denote classical cartesian coordinates centered at the bottom left corner of Q.
Given $\widehat{v}$ as in assumption \text{$(A1)$} of Lemma \ref{lemma:CC}, let $\widetilde{v} \in H^1(Q)$ defined as $\widetilde{v}(x,y)=x(\varepsilon_2 - x)/2$ for all $(x,y)$ in $Q$. It is immediate to check that $-\Delta \widetilde{v} =1$ and $\widetilde{v} \ge 0$ on $\partial E$, which in turn implies $\widetilde{v} \ge \widehat{v} \ge 0$ in $E$ by well known properties of harmonic functions. Therefore we immediately obtain $\Cone \le \| \widetilde{v}\|_{L^\infty(E)} = (\varepsilon_2)^2/8$.
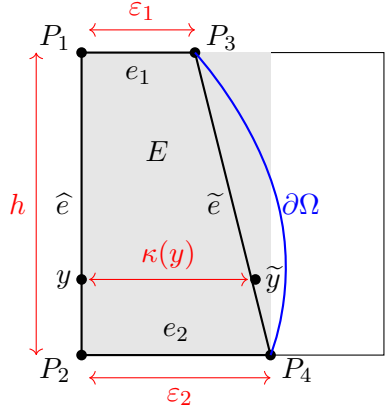
\begin{figure}
\begin{center}
\begin{tikzpicture}[scale=1]

\draw (0,0) rectangle (4,4);
\coordinate (A) at (0,0);   
\node[left]  at (-0.01,-0.2) {$P_2$};

\coordinate (B) at (0,4);   
\node[left]  at (-0.01,4.2) {$P_1$};

\coordinate (C) at (1.5,4);   
\node[right]  at (1.51, 4.2) {$P_3$};

\coordinate (D) at (2.5,0);   
\node[right]  at (2.51, -0.2) {$P_4$};

\fill[black!10] (B) rectangle (D);

\draw[thick] (A) -- (B) -- (C) -- (D) -- cycle;

\node at (1.25,+0.3) {$e_2$};
\node at (0.75,3.7) {$e_1$};
\node at (1.75,2) {$\widetilde{e}$};
\node at (-0.25,2) {$\widehat{e}$};

\node at (1,2.7) {\(E\)};
\draw[red,<->] (-0.6,0) -- (-0.6,4) node[midway,left] {${\color{red}h}$};
\draw[red,<->] (0.1,4.3) -- (1.50,4.3) node[midway,above] {${\color{red}\varepsilon_1}$};

\draw[red,<->] (0.1,-0.3) -- (2.5,-0.3) node[midway,below] {${\color{red}\varepsilon_2}$};

\fill (0,1) circle (2pt) node[left] {$y$};

\draw[red,<->] (0.1,1) -- (2.2,1) node[midway,above] {$\kappa(y)$};
\fill (2.3,1) circle (2pt) node[right] {$\widetilde{y}$};
\foreach \p in {A,B,C,D} \fill[black] (\p) circle (2pt);

\draw[blue, thick]
    plot [smooth, tension=1]
        coordinates {(C) (2.6,2) (D)};
\node at (2.9,2) {${\color{blue} \partial\Omega }$};
        
\end{tikzpicture}
\end{center}

\caption{
Quadrilateral element with $0<\varepsilon_1 \le \varepsilon_2 \leq h$.}\label{fig:quad:X}
\end{figure}

In order to obtain a bound for $\Ctwo$, we follow a similar argument as in the proof of Proposition \ref{prop:shapereg}. 
Let thus $v\in H^1(E)$ as in  assumption \text{$(A2)$} of Lemma \ref{lemma:CC} and let $(n_x,n_y)$ represent the two components of the outward unit normal to $\partial E$. Recalling \eqref{eq:kk:1} we easily obtain
$$
\partial_x \Pi^\nabla v = |E|^{-1} \big( \int_{ \widehat{e}}  v \, n_x + \int_{\widetilde{e}} v \, n_x \big) 
$$
where $\widehat{e}$ denotes the left vertical edge and $\widetilde{e}$ the right oblique edge of $E$
(see again Figure \ref{fig:quad:X}).
Now since $|E|=h(\varepsilon_1 + \varepsilon_2)/2$ some trivial calculation yields (recall that $\Pi^\nabla v$ is a constant vector field)
\begin{equation}\label{x-bound}
| \partial_x \Pi^\nabla v | \le \frac{|\widehat{e}|+|\widetilde{e}|}{|E|} \| v \|_{L^\infty(\partial E)}
\le \frac{3}{2} (\varepsilon_1 + \varepsilon_2 )^{-1} \| v \|_{L^\infty(\partial E)}
\le \frac{3}{2} (\varepsilon_2)^{-1} \| v \|_{L^\infty(\partial E)} \, .
\end{equation}
Let now  $e_1$ represent the top horizontal edge and $e_2$ the bottom horizontal edge of $E$ (see Figure \ref{fig:quad:X}). Furthermore, we denote by $n_{y,2}$ the $n_y$ component of the outward unit normal to $\widetilde{e}$.
Again from \eqref{eq:kk:1} and some simple manipulation we can write 
\begin{equation}\label{eq:trumpisjunk}
\begin{aligned}
| \partial_y \Pi^\nabla v | & \le 
|E|^{-1} \big( \int_{e_1} |v|  + \int_{e_2} |v| + \int_{\widetilde{e}} v \, n_y \big)
\le \frac{\varepsilon_1+\varepsilon_2 + \sqrt{2} h |n_{y,2}|}{h(\varepsilon_1+\varepsilon_2)/2} 
\| v \|_{L^\infty(\partial E)} \\
& \le \Big( h^{-1} + (\varepsilon_1+\varepsilon_2)^{-1} |n_{y,2}| \Big) \| v \|_{L^\infty(\partial E)} \, .
\end{aligned}
\end{equation}
Some simple algebra shows that 
$$
n_{y,2} = (\varepsilon_2-\varepsilon_1)/\sqrt{h^2 + (\varepsilon_2-\varepsilon_1)^2}
\le (\varepsilon_2-\varepsilon_1)/h 
$$ 
so that substitution into \eqref{eq:trumpisjunk} trivially yields
\begin{equation}\label{y-bound}
| \partial_y \Pi^\nabla v | \le 2 h^{-1} \| v \|_{L^\infty(\partial E)} \, . 
\end{equation}
From the first identity in \eqref{eq:kk:1} it follows that it exists (at least) a point $\xi \in \partial E$ such that
$|\Pi^\nabla v(\xi)| \le \| v \|_{L^\infty(\partial E)}$. Any other point $\eta \in E$ can be reached from $\xi$ by following an horizontal path of length equal or less than $\varepsilon_2$ plus a vertical path of length equal or less than $h$. Therefore, also using \eqref{x-bound} and \eqref{y-bound}, we obtain
$$
\| \Pi^\nabla v \|_{L^\infty(E)} \le \| \Pi^\nabla v (\xi) \| + | \partial_x \Pi^\nabla v | \, \varepsilon_2
+ | \partial_y \Pi^\nabla v | \, h 
\le \| v \|_{L^\infty(\partial E)} + 3/2 \| v \|_{L^\infty(\partial E)} + 2 \| v \|_{L^\infty(\partial E)} \, .
$$
Thus we obtain $\Ctwo \le 9/2$.

Let finally $w \in H^1_0(E)$ as in assumption \text{($A3$)} of Lemma \ref{lemma:CC}. We consider the triangle $T$ contained in $E$ and defined by the three vertexes of coordinates $(0,0)$, $(0,\varepsilon_1)$ and $(\varepsilon_2,0)$. Let $b$ denote the cubic bubble on $T$ with $\| b \|_{L^\infty(T)}=1$. Since $\Delta w$ is a first order polynomial (which we now imagine to be defined on the whole $Q$), it is easy to check that it exists a universal constant $c$ such that
$$
\int_E (\Delta w)^2 \le \int_Q (\Delta w)^2 \le c \int_T (\Delta w)^2 \, .
$$
First by standard scaling arguments for polynomials on triangles, then by integration by parts we now obtain
$$
c \int_T (\Delta w)^2 \le c' \int_T b (\Delta w)^2 = c' \int_T \nabla w \cdot \nabla (b \Delta w) 
$$
with $c'$ a positive universal constant. We now combine the two inequalities here above and apply first a Cauchy-Schwarz inequality to the right hand side, then again a standard scaling argument for polynomials on triangles, yielding
$$
\begin{aligned}
\int_E (\Delta w)^2 & \le c' \| \nabla w \|_{L^2(T)}  \| \nabla (b \Delta w) \|_{L^2(T)}
\le c'' \, (\varepsilon_2)^{-1} \| \nabla w \|_{L^2(T)}  \| b \Delta w \|_{L^2(T)} \\
& \le c'' \, (\varepsilon_2)^{-1} \| \nabla w \|_{L^2(E)}  \| \Delta w \|_{L^2(E)} \, ,
\end{aligned}
$$
where we also used $\| b \|_{L^\infty(T)}=1$ in the last step. We thus obtain $\Cthr \le (c'')^2 \, (\varepsilon_2)^{-2}$.

Finally, in order to show that the constant $\Cfour$ is independent of the involved geometric parameters, we consider the mapping that maps each $(x,y)$ into $(x (h/ \varepsilon_2) , y)$. The image of $E$ through such affine mapping is a shape regular polygon $\widetilde{E}$ in the sense of Proposition \ref{prop:shapereg}, with ratio $\rho_2/\rho_1$ independent of $h,\varepsilon_1,\varepsilon_2$. Therefore, on such mapped polygon $\widetilde{E}$ bound \eqref{eq:falcao} holds independently of the aforementioned geometric parameters. By a change of variables, it is therefore easy to check that bound \eqref{eq:falcao} holds also on $E$ with constant $\Cfour$ uniform in $h,\varepsilon_1,\varepsilon_2$. 
\end{proof}

\section{Stability in background Cartesian meshes}
\label{sec:stab:X}

In the present section we verify the validity of the Assumption \ref{ass:stab-form} (stability) for the mesh shapes generated by the proposed trimming algorithm, when the background mesh is the Cartesian mesh defined in Example \ref{exa:background-mesh}.
 Unless otherwise noted, we will assume the classical dofi-dofi choice and the standard $\Pi^\nabla_E$ projection on the space ${\cal S}_1(E)=\mathbb{P}_1(E)$ as the operator $\PiNew_E$ in \eqref{eq:dis:forms}. It is well known that condition \eqref{eq:prop-stabform} implies the stability of the discrete form $a_E(u,v)$ on $V_E$, see for instance \cite{volley}.

\smallskip
Based on the discussion in Example \ref{ex:element-geometries}, one can identify four types of polygons than can be generated by the proposed trimming procedure (see Fig. \ref{fig:geometries-1}). Furthermore, each polygon can exhibit different behaviors, depending on its aspect ratio and relative size of edges, as detailed below.

{\bf Triangular elements}. On triangles condition \eqref{eq:prop-stabform} is automatically satisfied as the involved space is void. 

{\bf Well behaved elements}. These are the majority of elements $E$ (all those not intersecting the boundary, for example) and are characterized by {\sl i)} being shape-regular (uniformly in the mesh family) and {\sl ii)} having all edges of length uniformly comparable to $h_E$. Under those conditions, it is well known in the VEM literature that \eqref{eq:prop-stabform} is satisfied e.g. by the so-called dofi-dofi stabilization, see \cite{BLRstab,brenner2018,chen2018}.

{\bf Almost well behaved elements} (see Fig. \ref{fig:geometries-1}, pentagons G or H). There are boundary elements which are shape-regular but exhibit (one or more) small edges, each having one or both extrema on $\partial \Omega_h$. By ``small edge'' we mean an edge with length $\varepsilon$ that may be arbitrarily smaller than $h_E$. In such case property \eqref{eq:prop-stabform} is more involved. 
If the small edge has both extrema on $\partial \Omega_h$ (meaning that all functions in $\VE$ vanish on the edge) then the small edge can be essentially ignored, leading again to \eqref{eq:prop-stabform}. But if only one extrema lays on $\partial \Omega_h$, then the best stability result one expects using the dofi-dofi stabilization is 
\begin{equation}\label{X:stab-log}
C_1 \, s_E(v_h,v_h) \le |v_h|_{H^1(E)}^2 \le C_2 \, \log{(1+h_E/\varepsilon)}  s_E(v_h,v_h) \quad \forall v_h \in \VE
\textrm{ with } \Pi^\nabla_E v_h = 0 \, ,
\end{equation}
see again \cite{BLRstab,brenner2018,chen2018}). On the other hand, such logarithmic factor is in practice negligible and equivalent to having a bigger constant $C_2$ in \eqref{eq:prop-stabform}, as verified by our numerical tests in Sect. \ref{sec:numerics}.    

However, there is an alternative definition of the stabilization form that avoids the logarithmic factor for elements of this kind. Let us briefly hint to it, assuming (only in order to simplify the exposition, the generalization being trivial) that there is only one ``small'' edge. In \cite{BLRstab,brenner2018,chen2018} it is shown that, on any shape-regular polygonal element $E$, the $H^1(E)$ seminorm of any $v_h \in V_E$ with $\Pi^\nabla_E v_h = 0$ is uniformly equivalent to the $H^{1/2}$ seminorm of $v|_{\partial E}$. 
We now recall that {\sl i)} such traces are piecewise linear on the boundary, {\sl ii)} the considered element has all edges but one with length comparable to $h_E$, and {\sl iii)} one edge $\overline{e}$ has length $\varepsilon$ and the functions in $V_E$ must vanish at one of its extrema. By a direct calculation (denoting by ${\cal V}_E \subset \partial E$ the set of vertexes of $E$), it can therefore be shown that 
\begin{equation}\label{eq:log-equiv}
| v_h |_{H^1(E)}^2 \, \simeq \, | v_h |_{H^{1/2}(\partial E)}^2 \, \simeq \, \log{(1 + h_E/\varepsilon)} |v_h(\widehat{\nu})|^2
\, + \!\!\! \sum_{\nu \in {\cal V}_E,  \, \nu \not=\widehat{\nu}} |v_h(\nu)|^2 \quad \forall v_h \in \VE \textrm{ with } \Pi^\nabla_E v_h = 0 \, ,
\end{equation}
where $\widehat{\nu}$ denotes the other extrema of $\overline{e}$, where the functions do not vanish. 
The above argument suggests that a more precise choice for $s_E(\cdot,\cdot)$ for this class of elements is
$$
s_E(v_h,w_h) := \log{(1 + h_E/\varepsilon)} v_h(\widehat{\nu})w_h(\widehat{\nu}) \, +
\!\!\! \sum_{\nu \in {\cal V}_E,  \,  \nu \not=\widehat{\nu}} v_h(\nu)w_h(\nu) \, ,
$$
that allows to get rid of the logarithmic factor in \eqref{X:stab-log}.

\begin{remark}[hexagonal elements]{\rm These elements may occur when the cell $K$ has exactly two opposite vertices in $\Omega$ (see Example \ref{ex:element-geometries} and Fig. \ref{fig:geometries-1}, plots E or F). If at least one edge contained in $\partial\Omega$ has length $c\, h$, where the constant $c$ may be small but only depends on the local geometry of $\Omega$, the element is well-behaved or almost well-behaved. Otherwise, $K \cap \Omega$ is not connected, each connected component being contained in a neighborhood of radius arbitrarily small of one of the vertices of $K$ contained in $\Omega$; in such a case, the hexagon can be replaced by two triangles, as mentioned in Example \ref{ex:element-geometries}. 
We conclude that, in all cases, these elements can be treated as discussed above.
}    
\end{remark}

{\bf Anisotropic quadrilaterals} (see Fig. \ref{fig:quad:X}). These are the most complex elements, since they lack shape regularity. Therefore in  Section \ref{S:stab_anis_quad} we will focus on the novel case of (boundary) anisotropic quadrilateral elements, where the lack of shape-regularity rules out the results mentioned above. More specifically, in the following part of this section $E$ will represent a (possibly anisotropic) quadrilateral defined by the lengths $\varepsilon_1 \le  \varepsilon_2 \le h$, as depicted in Figure \ref{fig:quad:X}. Most importantly, the edge $\widetilde{e}$ lays on $\partial \Omega_h$.  
Note that the interesting results in \cite{Chen_anisotr:2018}, also dealing with anisotropic quads, are not useful in our present context.

This is the only kind of elements for which we use a projection $\PiNew_E$ different from $\Pi^\nabla_E$, see \eqref{eq:dis:forms}. For each $v_h\in \VE$, we denote by $\PiNew_E\,v_h$ the unique element of $\Pol_1(E)$ such that 
\begin{equation}\label{def:Pitilde}
(\PiNew_E \, v_h)_{|\widetilde{e}}=0 \qquad \text{and} \qquad \int_{E} \nabla \PiNew_E\, v_h\cdot n_{\tilde e}=\int_{E }\nabla v_h \cdot n_{\tilde e},
\end{equation}
where $n_{\tilde e}$ is the normal to $\tilde e$. Correspondingly, we set ${\cal S}_1(E):=\{q\in {\mathbb P}_1(E) : q_{|E}=0 \text{ on } \tilde e\}$.
It is immediate to verify the computability of $\tilde\Pi_E v_h$ based on the sole knowledge of the vertex values (DOFs) of $v_h$. Indeed, 
integration by parts immediately yields 
$\int_{E }\nabla v_h \cdot n_{\tilde e}= \int_{\partial E} n_{\tilde{e}}\cdot n\, v_h$ with $n$ normal to $\partial E$ and the latter integral is computable.
For what concerns the $H^1$-continuity of $\tilde\Pi_E$, it is easy to verify that as  $\nabla \tilde \Pi_E v_h$ is constant on $E$ and $v_h\vert_{ \tilde{e}}=0$ there holds $\nabla \tilde\Pi_E \, v_h\cdot n_{\tilde{e}}=\frac{1}{\vert E\vert}\int_E \nabla v_h \cdot  n_{\tilde{e}}$ and  $\nabla \tilde\Pi_E \, v_h = (\nabla \tilde\Pi \, v_h\cdot n_{\tilde{e}}) n_{\tilde{e}}$. 
These results yield  $\|\nabla \tilde\Pi_E\, v_h \|^2_{L^2(E)}\leq \|\frac{1}{\vert E\vert} \int_E \nabla v_h\cdot n_{\tilde{e}} \|^2_{L^2(E)} \leq \|\nabla v_h \|^2_{L^2(E)}$.

\subsection{Stability for anisotropic quads}\label{S:stab_anis_quad}

With reference to Figure \ref{fig:quad:X}, let $\kappa=\kappa(y)$ be the length of the horizontal segment in $E$ whose left and right endpoints sit on $\widehat{e}$ and $\widetilde{e}$, respectively. It holds 
\begin{equation}\label{eq:def-gammay}
\kappa(y)=(h-y)\frac{\varepsilon_2-\varepsilon_1}{h}+\varepsilon_1.
\end{equation}
We also recall that the virtual functions are zero on the edge $\widetilde{e}$.
In the following the symbol $\lesssim$ will denote a bound up to a constant independent of $\varepsilon_1, \varepsilon_2, h$.
On the polygon $E$, instead of the classical dofi-dofi, in order to attain uniform stability we must make use of the following novel stabilization form on the enhanced virtual space $\VE$.
\begin{definition}[stabilization]\label{def:gamma-stabilization}
Let us set, for any $v_h, w_h \in \VE,$
\begin{equation}\label{def:stab_quad}
s_E(v_h,w_h):=\int_0^h \frac{v_h(0,y) w_h(0,y)}{\kappa(y)}dy.
\end{equation}
\end{definition}
The following results show that on $\VE$ the quantity $s_E(v_h,v_h)$ is equivalent to the square of the $H^1$-seminorm of $v_h$.
\begin{lemma}[continuity] \label{lemma:continuity_new_stab}
For any $v_h\in \VE$ there holds
\begin{equation}\label{stab_quad_cont}
    s_E(v_h,v_h)\leq \vert v_h \vert^2_{H^1(E)}.
\end{equation}
\end{lemma}
\begin{proof}
    Use that $v_h\vert_{\widetilde{e}}=0$ to write
    \begin{equation}
        v_h(0,y)=-\int_0^{\kappa(y)}\partial_x v_h(x,y)dx.
    \end{equation}
    Then there holds
    \begin{equation}
    \vert v_h(0,y)\vert^2
    \leq \kappa(y)
\int_0^{\kappa(y)} \vert \partial_x v_h(x,y)\vert^2 dx
    \end{equation}
    which implies
    \begin{equation}
        \int_0^h \frac{\vert v_h(0,y) \vert^2}{\kappa(y)}dy \leq \int_0^h \int_0^{\kappa(h)}\vert \partial_x v_h(x,y)\vert^2 dx dy = \vert v_h\vert^2_{H^1(E)} \, .
    \end{equation}
\end{proof}
To prove coercivity, we need a technical result. Let
        \begin{equation}\label{eq:def:S}
            W_{E,0}:=\{
            \psi\in \WE \ : \ \psi\vert_{\widetilde{e}}=0 \},
        \end{equation}
        where $\WE$ is the space, defined in \eqref{vem:basic}, of classical VEM functions on $E$.
\begin{lemma}\label{Lemma:aux_norma_gamma}
    For any $\psi \in W_{E,0}$ there holds 
    \begin{equation}
        \vert \psi \vert _{H^1(E)}\lesssim \| \psi \|_{\kappa(\widehat{e})}.
    \end{equation}
\end{lemma}
\begin{proof}
    We start observing that the definition of $W_{E,0}$ immediately implies that 
    \begin{equation}
        \psi(x,y)=\psi(0,y)(1-\frac{x}{\kappa(y)}).
    \end{equation}
    First, we have $ \partial_x \psi(x,y)=-\psi(0,y)\kappa(y)^{-1}$ which implies 
    \begin{eqnarray}
        \int_0^h \int_0^{\kappa(y)}
        \vert \partial_x \psi (x,y)\vert^2 dx dy = \int_0^h \frac{\vert \psi(0,y) \vert^2}{\kappa(y)} dy = \| \psi \|^2_{\kappa(\widehat{e})}.
    \end{eqnarray}
    On the other hand, we have 
    $$ \partial_y \psi(x,y)= \partial_y \psi(0,y)(1-\frac{x}{\kappa(y)})+\psi(0,y)\frac{x}{\kappa(y)^2}\kappa^\prime(y)=:A(x,y)+B(x,y)
    $$
    which yields
    \begin{equation}
        \int_E \vert \partial_y \psi \vert^2 \lesssim \int_E A^2(x,y) + \int_E B^2(x,y).
    \end{equation}
    Let us consider the first term
    \begin{eqnarray}
      \int_E A^2(x,y) &\leq&
\int_0^h \int_0^{\kappa(y)}\vert \partial_y \psi(0,y) \vert^2  dx dy = \int_0^h \kappa(y)\vert \partial_y \psi(0,y) \vert^2 dx dy \nonumber\\
&\leq& \varepsilon_2 \int_0^h \vert \partial_y\psi(0,y)\vert^2 dx dy \leq \frac{\varepsilon_2}{h} \max_{\nu \text{~vertices~of~}\widehat{e}}\vert \psi(\nu)\vert^2\nonumber\\
&\leq& \int_0^h \frac{\vert \psi(0,y)\vert^2}{\kappa(y)} dy = \| \psi \|^2_{\kappa(\widehat{e})} \, ,
    \end{eqnarray}
    where in the last inequality we applied \eqref{eq:X:piddu} recalling $\varepsilon_2 \le h$.

    We now consider the second term

    \begin{eqnarray}
        \int_E B(x,y)^2 = 
\int_0^h \int_0^{\kappa(y)}\left( \psi(0,y)\frac{x}{\kappa(y)^2}\kappa^\prime(y)\right)^2 dx dy.
    \end{eqnarray}
    Let us note that 
    $ \kappa(y)=\varepsilon_2- \frac{\Delta \varepsilon}{h}y$ with $\Delta \varepsilon:= \varepsilon_2 -\varepsilon_1$, which implies $\kappa^\prime(y)= -\frac{\Delta \varepsilon}{h}$.

    Moreover, we observe that as $x\in [0,\kappa(y)]$ there holds
    \begin{eqnarray}
        \left\vert 
        \frac{x}{\kappa^2(y)} \kappa^\prime(y)\right\vert \leq 
        \left\vert 
        \frac{x}{\kappa(y)} \right\vert
        \left\vert \frac{\kappa^\prime(y)}{\kappa(y)}\right\vert \leq 1 \left\vert \frac{\kappa^\prime(y)}{\kappa(y)}\right\vert. 
    \end{eqnarray}
    Thus, we have
    \begin{eqnarray}
        \int_E B(x,y)^2 &\leq& 
        \int_0^h \int_0^{\kappa(y)} \vert \psi(0,y)\vert^2 
\left \vert \frac{\kappa^\prime(y)}{\kappa(y)} \right\vert^2 dx dy \nonumber\\
&\leq&
\int_0^h  \vert \psi(0,y)\vert^2 
 \frac{\vert \kappa^\prime(y)\vert^2}{\kappa(y)}  dy \nonumber\\
 &\leq& \left(\frac{\Delta \varepsilon}{h}\right)^2
 \int_0^h \frac{\vert \psi(0,y) \vert^2}{\kappa(y)} dy \leq 1 \| \psi\|_{\kappa(\widehat{e})}^2.
    \end{eqnarray}
    Putting all together we obtain the thesis.
\end{proof}        
        
\begin{lemma}{(coercivity)}
\label{lemma:coercivity_new_stab}
For any $v_h\in \VE$ it holds
\begin{equation}\label{stb_quad_coerc}
\vert v_h \vert^2_{H^1(E)}\lesssim s_E(v_h,v_h).    
\end{equation}
    \begin{proof}
        We first observe that integration by parts yields 
        \begin{equation}
            \vert v_h\vert^2_{H^1(E)}=\int_E \nabla v_h\cdot\nabla v_h = -\int_E v_h \Delta v_h + \int_{\partial E} v_h \partial_n v_h=: T_1+T_2.
        \end{equation}
        By employing Proposition \ref{prop:shapereg} together with \eqref{aux_C_3} and the estimate of $\Cthr $, we have 
        \begin{eqnarray}
            T_1&\leq& \|v_h \|_{L^\infty(E)}
            \vert E\vert^{1/2}\| \Delta v_h\|_L^2(E)\nonumber\\
            &\lesssim& 
            \|v_h \|_{L^\infty(E)}
            (\varepsilon_2 h)^{1/2} \varepsilon_2^{-1} \vert v_h \vert^2_{H^1(E)}\nonumber\\
            &\lesssim& 
            (h/\varepsilon_2)^{1/2} \vert v_h\vert_{H^1(E)} \max_{\nu\in \partial E} \vert v_h(\nu)\vert.
        \end{eqnarray}
        Since it holds $\kappa(y)^{-1}\geq \varepsilon_2^{-1}$ we obtain 
        \begin{equation}\label{eq:X:piddu}
            \frac{h}{\varepsilon_2} 
            \max_{\nu\in \partial E} \vert v_h(\nu)\vert^2 \lesssim \int_0^h \frac{\vert v_h(0,y)\vert^2}{\kappa(y)} dy
        \end{equation}
        where we employed a standard inverse estimate along the edge $\widehat{e}$. 
        Thus we get
        \begin{equation}
            T_1\lesssim \vert v_h\vert_{H^1(E)} \left( s_E (v_h,v_h)\right)^{1/2}.
        \end{equation}
        Let us now focus on the term $T_2$. First, we define the following quantity for all sufficiently regular functions
        \begin{equation}
    \label{def_norm_gamma}
    \| \varphi\|^2_{\kappa(\widehat{e})}:=\int_0^h \frac{\vert \varphi(0,y)\vert^2}{\kappa(y)} dy.        
        \end{equation}
        We have
        \begin{eqnarray}
       T_2=\int_{\partial E} v_h \partial_n v_h
            \leq \| v_h\|_{\kappa(\widehat{e})} \sup_{\psi\in W_{E,0}} \frac{\int_{\partial E} \partial_n v_h \psi}{\|\psi\|_{\kappa(\widehat{e})}}.
        \end{eqnarray}
        Thus, we have
        \begin{equation}
            T_2\leq \|v_h\|_{\kappa(\widehat{e})} \sup_{\psi\in W_{E,0}} \frac{\int_E \nabla v_h\cdot \nabla \psi + \int_E \Delta v_h \psi}{\|\psi\|_{\kappa(\widehat{e})}}.
        \end{equation}
        Obviously, there hold
        \begin{equation}
            \int_E \nabla v_h \nabla \psi \leq \vert v_h\vert_{H^1(E)} \vert \psi \vert_{H^1(E)}
        \end{equation}
        and 
        \begin{equation}
            \int_E \Delta v_h \psi \leq \| \Delta v_h\|_{L^2(E)}
            \|\psi \|_{L^2(E)}\lesssim \varepsilon_2^{-1} \vert v_h\vert_{H^1(E)} \varepsilon_2 \vert \psi\vert_{H^1(E)}
        \end{equation}
        where we employed \eqref{aux_C_3} and an anisotropic Poincar\'e inequality. Applying Lemma \ref{Lemma:aux_norma_gamma}, 
        we have 
        \begin{equation}
            T_2\lesssim \vert v_h\vert_{H^1(E)} \| v_h \|_{\kappa(\widehat{e})}.
        \end{equation}
        Hence, we conclude 
        $$ \vert v_h \vert^2_{H^1(E)} \lesssim \| v_h \|_{\kappa(\widehat{e})}^2=s_E(v_h,v_h).$$
    \end{proof}
\end{lemma}

\begin{remark}
{\rm
We note that Lemmas \ref{lemma:continuity_new_stab} and \ref{lemma:coercivity_new_stab} clearly imply \eqref{eq:prop-stabform}. 
The reason why on these elements we adopted $\PiNew_E$ instead of $\Pi^\nabla_E$ is that $\PiNew_E$ respects the boundary condition on $\widetilde{e}$, which was a critical ingredient in the above proofs.
}
\end{remark}

\section{Interpolation in background Cartesian meshes}
\label{sec:checking}

In the present section we verify the validity of Assumption \ref{ass:interp} (interpolation) for the mesh types generated by the proposed trimming algorithm in two cases: regular and non-regular, according to the smoothness of the exact solution (and its extension). In the regular case, the boundary $\partial\Omega$ does not contain break points with reentrant corners, hence, the solution is in $H^2(\Omega)$; on the contrary, in the non-regular case, the solution contains a singular part, as indicated in \eqref{eq:split-u}-\eqref{eq:def-zeta}.

\subsection{Regular case} \label{sec:regular-case}

Hereafter, we will prove the following local interpolation error estimate, which clearly implies the global estimate \eqref{eq:interp-err} whenever $u \in H^2(\wt{\Omega})$.
\begin{theorem}[interpolation error bound in the regular case]\label{theo:interp-error-H2}
Let $E \in \mesh$ and let $u$ be any function in  $H^2(E)$ vanishing at each vertex of $E$ sitting on $\partial\Omega$. Then, there exists a constant $C_r>0$ independent of $u$ and $E$ such that
\begin{equation}\label{eq:interp-err-H2}
h^{-1}\|u-\IVh u \|_{0,E} +  | u-\IVh u |_{1,E} \leq C_r h |u|_{2,E}\,. 
\end{equation}  
\end{theorem}

To prove the theorem, let us recall that polygons generated by the proposed trimming procedure may exhibit different behaviors, as discussed in Sect. \ref{sec:stab:X} to which we refer for the classification. The following results hold.

{\bf Triangular elements}. These are triangles, possibly not shape-regular but always with a $\pi/2$ internal angle. Due to this angle condition, \eqref{eq:interp-err-H2} holds (see for instance Theorem 2.1 in \cite{apel-book}).

{\bf Well behaved elements}. For such elements, it is well known in the VEM literature that \eqref{eq:interp-err-H2} holds, see \cite{BLRstab,brenner2018,chen2018}.

{\bf Almost well behaved elements}. Property \eqref{eq:interp-err-H2} is valid for these elements, too (see again the above references or also \cite{smalledges}).

{\bf Anisotropic quadrilaterals}. These are the most complex elements, since they lack shape regularity, which rules out the application of the results mentioned above. Also note that the interesting results in \cite{Chen_anisotr:2018}, dealing with anisotropic quads as well, are not useful in our present context. Therefore, we devote the whole Section \ref{S:interp_anis_quad} to deal with these boundary anisotropic quadrilateral elements.

\smallskip\noindent
We briefly discuss the validity of Assumption \ref{ass:polapprox} (simpler polynomial approximation).

\begin{remark}\label{rem:polapprox:H2}
Let a generic element $E \in \mesh$ with $u \in H^2(E)$; we distinguish two cases. \\
\noindent
(1) Whenever $E$ is not an anisotropic quadrilateral (in accordance with our classification) we define ${\cal P}^1_E u := p_1$ to be the unique first order polynomial such that $\int_E (p_1-u)=0$ and $\int_E \nabla (p_1-u)={\bf 0}$. Then, noting that all the elements generated by our procedure are convex, and applying the Poincar\'e-Wirtinger inequality \cite{payne1960optimal}, we obtain
$$
\| u - {\cal P}^1_E u \|_{L^2(E)} 
\le \pi^{-1} h_E |u - {\cal P}^1_E u |_{H^1(E)}
\le \pi^{-2} h_E^2 |u|_{H^2(E)} \, ,
$$
yielding \eqref{eq:polapprox-err} for all elements on which $u$ is regular. \\

\noindent
(2) Whenever $E$ is an anisotropic quadrilateral, the definition of $\PiNew_E$ given in \eqref{def:Pitilde} implies that the affine polynomial ${\cal P}_E^1 u$ cannot be chosen freely as it must vanish on $\widetilde{e}$. In this case it is convenient to introduce $\widetilde{E}$, the patch given by the union of $E$ and its neighbor square element of the background mesh on the opposite side with respect to $\widetilde{e}$. Since $u$ vanishes at the endpoints of $\widetilde{e}$ and $\widetilde{E}$ is shape regular, it is easy to check that (choosing for instance ${\cal P}_E^1 u=\PiNew_{\widetilde{E}} u$, see \eqref{def:Pitilde})
$$
\| u - {\cal P}^1_{E} u \|_{L^2(E)}  \le
\| \widetilde{u} - {\cal P}^1_{E} u \|_{L^2(\widetilde{E})}
\le h_E |\widetilde{u} - {\cal P}^1_E u |_{H^1(\widetilde{E})}
\le h_E^2 |\widetilde{u}|_{H^2(\widetilde{E})} \, ,
$$
where $\widetilde{u}$ is the extension introduced in Section \ref{sec:extension}.
This result is satisfactory recalling \eqref{eq:reg-u-tilde}. 
\end{remark}

\subsubsection{Anisotropic quads}\label{S:interp_anis_quad}

Hereafter, $E$ will represent a (possibly anisotropic) quadrilateral defined by the lengths $\varepsilon_1 \le  \varepsilon_2 \le h$, as depicted in Figure \ref{fig:quad:X}. It is important to notice that the endpoints of the edge $\widetilde{e}$ lay on $\partial \Omega_h$; in particular, virtual functions vanish on this edge. The symbol $\lesssim$ here denotes a bound up to a constant independent of $\varepsilon_1, \varepsilon_2, h$.

We start by the following lemma, concerning interpolation in the classical virtual element space $\WE$ defined in \eqref{vem:basic}.
\begin{lemma}\label{Q:lem:1}
Let $u \in H^2(E)$ and let $w_I$ denote its nodal (vertex) interpolant in $\WE$. Then it holds
$$
| u - w_I |_{H^1(E)} \lesssim h |u|_{H^2(E)} \, .
$$
\end{lemma}
\begin{proof}
We consider a partition obtained subdividing $E$ into the two triangles obtained by drawing a diagonal to $E$. We denote by $q_1$ the piecewise (first order) polynomial function on such sub-triangulation obtained by nodal interpolation of $u$ at the vertexes.
It holds (see for instance \cite{apel-book}): 
\begin{equation}\label{Q:pol-approx}
| u - q_1 |_{H^1(E)} \lesssim h |u|_{H^2(E)} \, .
\end{equation}

Let us now write $u = u_\partial + u_b$, defined by
$$
\Delta u_\partial = 0 \, , \ u_\partial |_{\partial E} = u |_{\partial E} \ \textrm{ and } \
\Delta u_b = \Delta u \, , \ u_b |_{\partial E} = 0 \, .
$$
Then 
\begin{equation}\label{Q:ing1}
| u - w_I |_{H^1(E)} \le | u_\partial - w_I |_{H^1(E)} + | u_b |_{H^1(E)} \, .
\end{equation}
Since $q_1|_{\partial E} = w_I |_{\partial E}$ and $\Delta (u_\partial-w_I)=0$ in $E$, properties of harmonic functions immediately yield (also recalling \eqref{Q:pol-approx})
\begin{equation}\label{Q:ing2}
| u_\partial - w_I |_{H^1(E)} \le | u_\partial - q_1 |_{H^1(E)} \le | u - q_1 |_{H^1(E)} + | u_b |_{H^1(E)}
\lesssim h  |u|_{H^2(E)} + | u_b |_{H^1(E)} \, .
\end{equation}
It is immediate to check, integrating by parts, that $\nabla u_\partial$ and $\nabla u_b$ are orthogonal with respect to the $L^2(E)$ scalar product. Furthermore, let $\overline{\nabla u}$ denote the integral average of $\nabla u$ on $E$. Then, by an obvious orthogonality and Poincar\'e estimates for zero average functions
$$
\begin{aligned}
| u_b |_{H^1(E)}^2 & = \int_E \nabla u_b \cdot \nabla u = \int_E \nabla u_b \cdot (\nabla u
- \overline{\nabla u}) \le | u_b |_{H^1(E)} \| \nabla u - \overline{\nabla u} \|_{L^2(E)} \\
& \lesssim | u_b |_{H^1(E)}  h |\nabla u|_{H^{1}(E)} = 
| u_b |_{H^1(E)} h |u|_{H^2(E)} \, .
\end{aligned}
$$
The proof is concluded combining the above bound with \eqref{Q:ing1} and \eqref{Q:ing2}.
\end{proof}

We now need to tackle the more complex interpolation into the {\it enhanced} VEM space $\VE$, defined in \eqref{vem:choice:2}. 
We start by the following result. 

\begin{lemma}\label{Q:lem:2}
Let $u$ and $w_I$ as in Lemma \ref{Q:lem:1}. Let $u_I = \IVE u$ denote the nodal (vertex) interpolant of $u$ in $\VE$. 
Then it holds
$$
| u_I - w_I |_{H^1(E)} \lesssim \varepsilon_2^{-1} \| w_I - \Pi^\nabla w_I \|_{L^2(E)} \, .
$$
\end{lemma}
\begin{proof}
First noting that $(u_I - w_I)$ vanishes on $\partial E$ and integrating by parts, then observing  that $\Delta (u_I - w_I) \in \Pol_1(E)$ 
and by definition of $\VE$, we get
\begin{equation}\label{Q:ing3}
| u_I - w_I |_{H^1(E)}^2 = \int_E \nabla (u_I - w_I) \cdot \nabla (u_I - w_I)
= - \int_E (u_I - w_I) \Delta (u_I - w_I) = - \int_E (\Pi^\nabla u_I - w_I) \Delta (u_I - w_I) \, .
\end{equation}
We now recall that the $\Pi^\nabla$ operator only depends on the boundary values of the function, and that $u_I = w_I$ on $\partial E$.
Therefore, starting from identity \eqref{Q:ing3} and recalling the bound for $\Cthr$ obtained in Proposition \ref{prop:ani:C123}, we derive
$$ 
\begin{aligned}
| u_I - w_I |_{H^1(E)}^2 & = - \int_E (\Pi^\nabla w_I - w_I) \Delta (u_I - w_I)
\le \| w_I - \Pi^\nabla w_I \|_{L^2(E)} \| \Delta (u_I - w_I) \|_{L^2(E)}  \\
& \lesssim \varepsilon_2^{-1} \| w_I - \Pi^\nabla w_I \|_{L^2(E)}  | u_I - w_I |_{H^1(E)} \, .
\end{aligned}
$$
The proof is concluded.
\end{proof}

We now need a result stating the boundedness of the $\Pi^\nabla$ operator on anisotropic quads.
\begin{lemma}\label{Q:lem:3}
It holds
$$
\| \Pi^\nabla v \|_{L^2(E)} \lesssim 
\varepsilon_2^{1/2} \| v \|_{L^2(\widehat{e})} +
h^{1/2} \varepsilon_2 \| \partial_x v \|_{L^2(e_1 \cup e_2)} 
$$
for all $v \in H^2(E)$  with $v|_{\widetilde{e}} = 0$.
\end{lemma}
\begin{proof}
We note that $\nabla \Pi^\nabla v$ is constant; therefore, by definition and integration by parts, 
\begin{equation}\label{Q:ing4}
\nabla \Pi^\nabla v = |E|^{-1} \int_E \nabla \Pi^\nabla v = 
|E|^{-1} \int_E \nabla v =  |E|^{-1} \int_{\partial E} v \, {\bf n} \, ,
\end{equation}
with ${\bf n}$ denoting the unit outward normal to the boundary. \\
By a H\"older inequality, from \eqref{Q:ing4} we easily obtain (also since $(1/2) h \varepsilon_2 \le |E| \le h \varepsilon_2$)
\begin{equation}\label{Q:ing5}
| \partial_x \Pi^\nabla v | \leq |E|^{-1} |\! \int_{\widehat{e}} v \, | \lesssim h^{-1/2} \varepsilon_2^{-1} \| v \|_{L^2(\widehat{e})} \, .
\end{equation}
By an analogous argument, also using a standard Poincar\'e in one dimension along each edge (recall that $v|_{\widetilde{e}} = 0$),
\begin{equation}\label{Q:ing6}
| \partial_y \Pi^\nabla v |  \lesssim |E|^{-1} \varepsilon_2^{1/2} \| v \|_{L^2(e_1 \cup e_2)}
\lesssim h^{-1} \varepsilon_2^{1/2} \| \partial_x v \|_{L^2(e_1 \cup e_2)} \, .
\end{equation}
From the first identity in \eqref{eq:kk:1} it easily follows that it exists (at least) a point $\xi \in \partial E$ such that
$\Pi^\nabla v (\xi) =  |\partial E|^{-1} \int_{\partial E} v$. Any other point in $E$ can be reached from $\xi$ by following an horizontal path of length equal or less than $\varepsilon_2$ plus a vertical path of length equal or less than $h$. Therefore, also using \eqref{Q:ing5} and \eqref{Q:ing6}, by some standard calculation we obtain
$$
\begin{aligned}
\| \Pi^\nabla v \|_{L^\infty(E)} & \le |\partial E|^{-1} | \int_{\partial E} v | + \varepsilon_2 | \partial_x \Pi^\nabla v | + h | \partial_y \Pi^\nabla v | \\
& \lesssim h^{-1/2} \| v \|_{L^2(\partial E)} + h^{-1/2} \| v \|_{L^2(\widehat{e})} + \varepsilon_2^{1/2} \| \partial_x v \|_{L^2(e_1 \cup e_2)} \\
& \lesssim h^{-1/2} \| v \|_{L^2(\widehat{e})} + \varepsilon_2^{1/2} \| \partial_x v \|_{L^2(e_1 \cup e_2)} \, ,
\end{aligned}
$$
where the last step follows recalling $v|_{\widetilde{e}} = 0$ and again a one dimensional Poincar\'e inequality on the horizontal edges.
The result now follows immediately by a H\"older inequality.
\end{proof}

We can finally state the following approximation result, which is precisely the error bound in energy contained in \eqref{eq:interp-err-H2} for the considered quadrilaterals.
\begin{theorem}[$H^1$ interpolation error bound in the regular case]\label{Q:theo:int}
Let $u \in H^2(E)$ vanish at each vertex of $E$ sitting on $\partial\Omega$.  
Then it holds
$$
| u - \IVE u |_{H^1(E)} \lesssim h |u|_{H^2(E)} \, .
$$
\end{theorem}
\begin{proof}
Hereafter, we refer again to Figure \ref{fig:quad:X} for the notation, and we set $u_I := \IVE u$ for short. As a preliminary observation we note that, for any function $\psi \in H^1(E)$ vanishing on $\widetilde{e}$, a standard integration and density argument, in the spirit of the Poincar\'e inequality, shows that 
\begin{equation}\label{ani-poinc}
\| \psi \|_{L^2(E)} \lesssim \varepsilon_2 \| \partial_x \psi \|_{L^2(E)} 
\ \textrm{ and } \
\| \psi \|_{L^2(\widehat{e})} \lesssim \varepsilon_2^{1/2} \| \partial_x \psi \|_{L^2(E)}
\, .
\end{equation}
Furthermore, let us define $p_1 \in \Pol_1(E)$ by
$$
p_1|_{\widetilde{e}} = 0 \ , \quad \int_E \partial_x (u-p_1) = 0 \, .
$$
A classical Poincar\'e inequality for functions with zero integral (which can be proved by a standard mapping argument) yields
\begin{equation}\label{p1-approx}
\| \partial_x (u-p_1) \|_{L^2(E)} \lesssim h |u|_{H^2(E)} \, .
\end{equation}
We now start by combining a triangle inequality with Lemmas \ref{Q:lem:1} and \ref{Q:lem:2}, leading to
\begin{equation}\label{pileup-1}
| u - u_I |_{H^1(E)} \le | u - w_I |_{H^1(E)} + | u_I - w_I |_{H^1(E)} 
\lesssim h |u|_{H^2(E)} + \varepsilon_2^{-1} \| w_I - \Pi^\nabla w_I \|_{L^2(E)} \, .
\end{equation}
The second term in the right hand side here above is bounded by the triangle inequality and recalling that $\Pi^\nabla$ preserves first order polynomials:
\begin{equation}\label{pileup-2}
\varepsilon_2^{-1} \| w_I - \Pi^\nabla w_I \|_{L^2(E)} \le
\varepsilon_2^{-1} \| w_I - p_1 \|_{L^2(E)} + \varepsilon_2^{-1} \| \Pi^\nabla (w_I - p_1) \|_{L^2(E)}
=: T_1 + T_2 \, .
\end{equation}
We now note that, by definition, both $w_I$ and $p_1$ vanish on $\widetilde{e}$. Therefore we can apply \eqref{ani-poinc}, followed by a triangle inequality and \eqref{p1-approx}, yielding
\begin{equation}\label{pileup-3}
\begin{aligned}
T_1 & \lesssim \| \partial_x (w_I - p_1) \|_{L^2(E)}
\le \| \partial_x (w_I - u) \|_{L^2(E)} + 
\| \partial_x (u - p_1) \|_{L^2(E)} \\
& \le |u - w_I|_{H^1(E)} + h |u|_{H^2(E)} 
\lesssim h |u|_{H^2(E)} \, ,
\end{aligned}
\end{equation}
where in the last step we used Lemma \ref{Q:lem:1}. 
Regarding term $T_2$, we make use of Lemma \ref{Q:lem:3} (note again that both $w_I$ and $p_1$ vanish on $\widetilde{e}$)
\begin{equation}\label{pileup-4}
T_2
\lesssim 
\varepsilon_2^{-1/2} \| (w_I - p_1) \|_{L^2(\widehat{e})} +
h^{1/2} \| \partial_x (w_I - p_1) \|_{L^2(e_1 \cup e_2)} 
=: T_3 + T_4 \, .
\end{equation}
Term $T_3$ is bounded immediately first applying \eqref{ani-poinc} and then the same steps as in \eqref{pileup-3}, obtaining
\begin{equation}\label{pileup-5}
T_3 \lesssim \| \partial_x (w_I - p_1) \|_{L^2(E)} \lesssim h |u|_{H^2(E)} \, .
\end{equation}
In order to handle term $T_4$, we start by a triangle inequality 
\begin{equation}\label{noteworthy}
T_4 \le h^{1/2} \| \partial_x (w_I - u) \|_{L^2(e_1 \cup e_2)} 
+ h^{1/2} \| \partial_x (u - p_1) \|_{L^2(e_1 \cup e_2)} 
\, .
\end{equation}
We observe that, since $w_I$ is affine on each edge $e_i$, $i=1,2$, and interpolatory at the endpoints, the function $\partial_x w_I|_{e_i}$ corresponds to the $L^2(e_i)$ projection of $\partial_x u|_{e_i}$ on constant functions. As a consequence it is trivial to check that
$$
\| \partial_x (w_I - u) \|_{L^2(e_1 \cup e_2)} \le \| \partial_x (u - p_1) \|_{L^2(e_1 \cup e_2)} \, .
$$
First using the above observation in \eqref{noteworthy}, then by a scaled trace inequality 
\begin{equation}\label{pileup-6}
T_4 \le h^{1/2} \| \partial_x (u - p_1) \|_{L^2(e_1 \cup e_2)}
\lesssim \| \partial_x (u - p_1) \|_{L^2(E)}
+ h | (u - p_1) |_{H^1(E)} 
\lesssim h |u|_{H^2(E)} \, ,
\end{equation}
where we also used \eqref{p1-approx} in the last inequality.
The proof is concluded combining the bounds \eqref{pileup-1}, \eqref{pileup-2}, \eqref{pileup-3}, \eqref{pileup-4}, \eqref{pileup-5}, and \eqref{pileup-6}.
\end{proof}

We now extend the above result to estimating the interpolation error in the $L^2$-norm.

\begin{corollary}[$L^2$ interpolation error bound in the regular case]\label{Q:theo:int:L2}
Let $u \in H^2(E)$ vanish at each vertex of $E$ sitting on $\partial\Omega$. 
Then it holds
$$
\| u - \IVE u \|_{L^2(E)} \lesssim h^2 |u|_{H^2(E)} \, .
$$
\end{corollary}
\begin{proof}
Also in this proof we follow the notation depicted in Figure \ref{fig:quad:X}, and we use the notation $u_I:= \IVE u$.
By exploiting the density $C^1(E) \subseteq H^1(E)$ and by a standard integration argument in the $x-$direction, it is easy to check that, for any $w \in H^1(E)$, it holds
\begin{equation}
\| w \|_{L^2(E)} \lesssim \varepsilon_2 \| \partial_x w \|_{L^2(E)} + \varepsilon_2^{1/2} \| w \|_{L^2(\widetilde{e})} \, . 
\end{equation}
Since interpolation preserves $\Pol_1(E)$, we can write 
$$
u-u_I = \widetilde{u} - \widetilde{u}_I \quad  \textrm{ where } \quad \widetilde{u} = u - p_1 \, ,
$$
with $p_1 \in \Pol_1(E)$ any polynomial such that $\int_E \nabla \widetilde{u} = {\bf 0}$.
We now first combine the two observations above (with the choice $w=u-u_I=\widetilde{u} - \widetilde{u}_I$), then apply Theorem \ref{Q:theo:int}, yielding
$$
\| \widetilde{u} - \widetilde{u}_I \|_{L^2(E)}
\lesssim \varepsilon_2 \| \partial_x (u - u_I) \|_{L^2(E)}
+ \varepsilon_2^{1/2} \| \widetilde{u} - \widetilde{u}_I \|_{L^2(\widetilde{e})} 
\lesssim h  \varepsilon_2 |u|_{H^2(E)} + \varepsilon_2^{1/2} \| \widetilde{u} - \widetilde{u}_I \|_{L^2(\widetilde{e})} \, .
$$
Note that, by definition, the trace of $\widetilde{u}_I$ on $\widetilde{e}$ corresponds to the first order polynomial interpolant of the trace of $\widetilde{u}$ on $\widetilde{e}$. Therefore we can bound the second term in the right-hand side using classical interpolation results for $\Pol_1$ polynomials in one dimension, followed by a standard anisotropic scaled trace inequality applied to $\nabla \widetilde{u}$:
$$
\varepsilon_2^{1/2} \| \widetilde{u} - \widetilde{u}_I \|_{L^2(\widetilde{e})} \lesssim
\varepsilon_2^{1/2} h | \widetilde{u} |_{H^1(\widetilde{e})}
\lesssim \varepsilon_2^{1/2} h \Big( \varepsilon_2^{-1/2} | \widetilde{u} |_{H^1(E)} + 
\varepsilon_2^{1/2} | \widetilde{u} |_{H^2(E)} \Big) \, .
$$
The first term on the right-hand side is bounded by a classical Poincar\'e inequality for functions with vanishing integral on convex domains (recall the definition of $\widetilde{u}$ above). We finally obtain
$$
\varepsilon_2^{1/2} \| \widetilde{u} - \widetilde{u}_I \|_{L^2(\widetilde{e})} 
\lesssim
h^2 | \widetilde{u} |_{H^2(E)} = h^2 | u |_{H^2(E)} \, .
$$
The proof now follows trivially combining all the above bounds and recalling $\varepsilon_2 \le h$.
\end{proof}

\subsection{Non-regular case}\label{sec:non-regular-case}

We now discuss the validity of the error bound \eqref{eq:interp-err} when the domain exhibits reentrant corners, so that the solution (and its extension in $\widetilde{\Omega}$) is the sum of a regular part and a singular part, according to \eqref{eq:split-u}-\eqref{eq:def-zeta} (and its extension \eqref{eq:split-u-ext}). For simplicity, let us ignore the notation for the extensions, so let us write $u=\psi+\sum_{j=1}^J \lambda_j \zeta_j$, with $\psi \in H^2(\wt{\Omega})$ and each $\zeta_j$ having a singularity of index $\gamma_j \in (\frac12,1)$ at a break point $\bz_j \in \partial\Omega$. Let us set again $\bar{\gamma} := \min_{1 \leq j \leq J} \gamma_j$. At first, we are going to establish the following result.
\begin{theorem}[$H^1$ interpolation error bound in the non-regular case]\label{thm:error-nrc} There exists a constant $C>0$ independent of $u$ and ${\cal T}_h$ such that
\begin{equation}\label{eq:error-nrc}
| u-\IVh u |_{1,\Omega_h} \leq C\, h^{\bar{\gamma}}  \enorm{u} \,.    
\end{equation}
\end{theorem}

To prove the estimate, we first observe that the linearity of the interpolation operator implies
\[
| u-\IVh u |_{1,\Omega_h} \leq | \psi-\IVh \psi |_{1,\Omega_h} + \sum_{j=1}^J |\lambda_j| | \zeta_j-\IVh \zeta_j |_{1,\Omega_h}.
\]
By applying Theorem \ref{theo:interp-error-H2} to the function $\psi$ (note that each $\zeta_j$ vanishes on $\partial\Omega$, hence $\psi$ vanishes therein as well), we have
\[
| \psi-\IVh \psi |_{1,\Omega_h} \lesssim h \, |\psi |_{2, \Omega_h},
\]
hence, if we prove that
\begin{equation}\label{eq:interp-singular}
| \zeta_j-\IVh \zeta_j |_{1,\Omega_h} \lesssim h^{\gamma_j}, \qquad 1 \leq j \leq J,
\end{equation}
we immediately arrive at \eqref{eq:error-nrc} after recalling the bound \eqref{eq:reg-u-tilde}.

\medskip
From now on, we focus on bound \eqref{eq:interp-singular} for some $j \in [1,J]$. Since $\zeta_j$ is locally supported around $\bz_j$, there exists a neighborhood ${\cal N}_j(\bz_j,R_j)$ of $\bz_j$ of radius $R_j=O(1)$ such that $\zeta_j-\IVh \zeta_j$ vanishes outside this neighborhood. Let ${\cal N}_j(\bz,r_j)$ be a neighborhood of $\bz_j$ of radius $r_j=ch$, $c \simeq 1$. In the annulus $A_j = {\cal N}_j(\bz_j,R_j) \setminus {\cal N}_j(\bz_j,r_j)$, the function $\zeta_j$ is smooth; hence, we can apply Theorem \ref{theo:interp-error-H2} again, this time to the function $\zeta_j$, and get the error bound
\[
| \zeta_j-\IVh \zeta_j |_{1,\Omega_h \cap A_j} \lesssim h \, |\zeta_j|_{2,A_j}.
\]
An explicit computation yields
\[
|\zeta_j|_{2,A_j}^2 \lesssim \int_{r_j}^{R_j} r^{2(\gamma_j -2)+1} {\rm d}r \lesssim h^{2(\gamma_j -1)},
\]
whence
\begin{equation}
| \zeta_j-\IVh \zeta_j |_{1,\Omega_h \cap A_j} \lesssim h^{\gamma_j}.   
\end{equation}

It remains to bound the interpolation error in ${\cal N}_j(\bz_j,r_j)$, or - more precisely - in each element $E \in \mesh$ that intersects ${\cal N}_j(\bz_j,r_j)$. The number of such elements is 
bounded by the number of elements $K \in {\cal T}_h^B$ of the background mesh that intersect ${\cal N}_j(\bz_j,r_j)$ (recall that by construction each $K \in {\cal T}_h^B$ contains at most one $E \in \mesh$), and the latter number is upper bounded independently of $h$, since each $K$ is shape-regular with diameter $O(h)$. Considering any such $E$, we use the triangle inequality
\begin{equation}
| \zeta_j-\IVh \zeta_j |_{1,E} \leq | \zeta_j|_{1,E} + |\IVh \zeta_j |_{1,E}
\end{equation}
and estimate independently the two addends on the right-hand side. The first one is easily bounded:
\begin{equation}\label{eq:Falcao}
| \zeta_j|_{1,E}^2 \lesssim \int_0^{ch} r^{2(\gamma_j -1)+1} {\rm d}r \lesssim h^{2\gamma_j}.
\end{equation}
Thus, we are left with the problem of estimating the quantity $|\IVh \zeta_j |_{1,E}$ when $E \in \mesh$ has distance $O(h)$ from $\bz_j$. This will be the object of the forthcoming analysis.

We first establish a technical result that will be useful in the sequel.

\begin{lemma}[values of the singular functions]\label{lemma:value-sing}
Let $Z=\bz_j \in \partial\Omega$ be a break point with singularity index $\gamma_j \in (\frac12,1)$, and let $\zeta_j$ be the associated singular function defined in \eqref{eq:def-zeta-j}. Let ${\cal N}_j(\bz_j)$ be a neighborhood of $\bz_j$ of radius $O(h)$. Let $P,V \in {\cal N}_j(\bz_j) \cap \overline{\Omega}_h$, with $\zeta_j(P)=0$. Then, there exists an absolute positive constant $\Lambda=O(1)$ such that
\begin{equation}\label{eq:bound-value-sing}
| \zeta_j(V)| \lesssim 
\begin{cases}
|V-Z|^{\gamma_j} & \text{ if } |V-Z| < \Lambda\,  |V-P|, \\[3pt]
|V-Z|^{\gamma_j-1}|V-P| & \text{ if } |V-Z| \geq \Lambda \,  |V-P|.
\end{cases}
\end{equation}
\end{lemma}
\proof By applying the ${\cal C}^2$-diffeomorphism $\widetilde{T}_j: {\cal N}_j \to \widetilde{S}_j$, one is lead to consider the geometry of Figure \ref{fig:Singularity region}, 
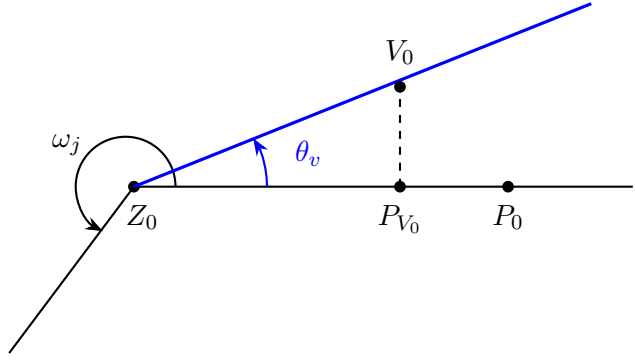
\begin{figure}[t!]
\begin{center}

\begin{tikzpicture}[scale=1.1,>=Stealth]

\draw[thick] (1,0) -- (7,0);

\draw[thick] (-0.5,-2) -- (1,0);

\fill (1,0) circle (2pt);
\node[below] at (1.1,-0.1) {$Z_0$};

\draw[thick,->] (1.5,0) arc[start angle=0,end angle=243,radius=0.6];
\node[left] at (0.5,0.5) {$\omega_{j}$};

\draw[very thick,blue] (1,0) -- (6.5,2.2);

\draw[blue,thick,->] (2.6,0) arc[start angle=0,end angle=32,radius=1.1];
\node[blue] at (3.1,0.4) {$\theta_v$};

\fill (4.2,0) circle (2pt);
\node[below] at (4.2,-0.1) {$P_{V_0}$};

\fill (4.2,1.2) circle (2pt);
\node[above] at (4.2,1.35) {$V_0$};

\draw[dashed,thick] (4.2,0) -- (4.2,1.2);

\fill (5.5,0) circle (2pt);
\node[below] at (5.5,-0.1) {$P_0$};

\end{tikzpicture}
\end{center}
\caption{Mapped region around the break point $Z=\bz_j$ }\label{fig:Singularity region}
\end{figure}
where $Z_0=\widetilde{T}_j(Z)$, $P_0=\widetilde{T}_j(P)$, $V_0=\widetilde{T}_j(V)$, and $\theta_V$ is the angle formed by the segments $Z_0P_0$ and $Z_0V_0$. In this reference system, we have
\[
\zeta_j(V)=\zeta_j^0(V_0) = |V_0-Z_0|^{\gamma_j} \sin(\gamma_j \theta_V).
\]
This shows that, in any case,
\[
|\zeta_j(V)| \leq |V_0-Z_0|^{\gamma_j} \lesssim |V-Z|^{\gamma_j}.
\]
Thus, from now on let us assume $|V-Z| \geq \Lambda \,  |V-P|$, which is equivalent to
\[
|V_0-Z_0| \geq \Lambda_0 \,  |V_0-P_0|
\]
for a constant $\Lambda_0 \simeq \Lambda$. If we choose $K_0=1$, then necessarily $0 < \theta_V < \frac\pi2$, hence
\[
\sin (\gamma_j \theta_V) < \sin \theta_V = \frac{|V_0-P_{V_0}|}{|V_0-Z_0|} \leq \frac{|V_0-P_0|}{|V_0-Z_0|}, 
\]
where $P_{V_0}$ is the orthogonal projection of $V_0$ upon the segment $Z_0P_0$. This implies
\[
|\zeta_j(V)| \leq |V_0-Z_0|^{\gamma_j-1}|V_0-P_0| \lesssim |V-Z|^{\gamma_j-1}|V-P|
\]
as desired.
\endproof

\begin{corollary}\label{cor:value-sing}
Under the assumptions of the previous Lemma, one has
\begin{equation}\label{eq:cor:value-sing}
|\zeta_j(V)| \lesssim |V-Z|^{\gamma_j-1}|V-P| \lesssim h^{\gamma_j-1} |V-P|.
\end{equation}
    \end{corollary}

As already declared, we can now proceed in bounding the quantity $|\IVh \zeta_j |_{1,E}$ for a generic element $E \in \mesh$ within distance $O(h)$ from $\bz_j$. We follow the same element classification introduced in Section \ref{sec:stab:X}.

\medskip
{\bf Triangular elements.} It is not restrictive to focus on the triangle $E$ of Figure \ref{fig:aniso-tria}. Since $\zeta_j$ vanishes at $P_1$ and $P_3$, we can write $\IVh \zeta_j = \zeta_j(P_2) \phi_2$, where $\phi_2 \in \mathbb{P}_1(E)$ is the Lagrange basis function satisfying $\phi_2(P_i)=\delta_{i,2}$ for $1 \leq i \leq 3$. Corollary \ref{eq:cor:value-sing} applied with $V=P_2$ and $P=P_1, P_3$ yields
\[
|\zeta_j(P_2)| \lesssim h^{\gamma_j-1}\min(\varepsilon_1, \varepsilon_2),
\]
whereas 
\[
|\phi_2|_{1,E}^2 = \left(\frac1{\varepsilon_1^2}+ \frac1{\varepsilon_2^2}\right) \varepsilon_1\epsilon_2 \lesssim \frac{\max(\varepsilon_1, \varepsilon_2)}{\min(\varepsilon_1, \varepsilon_2)}, \qquad \|\phi_2\|_{0,E}^2 \simeq \varepsilon_1 \varepsilon_2. 
\]
Since $\varepsilon_1, \varepsilon_2 \leq h$, we arrive at the desired bounds
\[
|\IVh \zeta_j|_{1,E}^2 \lesssim h^{2\gamma_j -2} \max(\varepsilon_1, \varepsilon_2) \min(\varepsilon_1, \varepsilon_2) \leq h^{2\gamma_j},
\quad 
\|\IVh \zeta_j\|_{0,E}^2 \lesssim h^{2\gamma_j -2}  \min(\varepsilon_1, \varepsilon_2) \varepsilon_1 \varepsilon_2 \leq 
h^{2\gamma_j+2}.
\]


\begin{figure}[t!]
\begin{center}

\begin{tikzpicture}[scale=1.2]

\coordinate (P2) at (0,0);
\coordinate (P1) at (0,4);
\coordinate (P3) at (2.5,0);
\coordinate (E)  at (0.6,1.2);

\draw (0,0) rectangle (4,4);

\draw[thick] (P1) -- (P3) -- (P2) -- cycle;

\draw[blue, thick] (P1) -- (E) -- (P3);

\node[left]  at (P1) {$P_1$};
\node[below] at (-0.3,0)  {$P_2$};
\node[below] at (P3) {$P_3$};
\node[left]  at (1.25,0.4) {$E$};

\draw[fill=blue] (E) circle (2pt);

\fill (P1) circle (2pt);
\fill (P2) circle (2pt);
\fill (P3) circle (2pt);
\node[left]  at ($(P1)!0.5!(P2)$) {$e_1$};
\node[below] at ($(P2)!0.5!(P3)$) {$e_2$};

\draw[red, <->, thick] (-0.8,0) -- (-0.8,4);
\node[red, left] at (-0.8,2) {$\varepsilon_1$};

\draw[red, <->, thick] (0,-0.6) -- (2.5,-0.6);
\node[red, below] at (1.25,-0.6) {$\varepsilon_2$};

\node[blue] at (1.2,1.5) {$\partial\Omega$};

\end{tikzpicture}

\end{center}
\caption{ }\label{fig:aniso-tria}
\end{figure}

{\bf Well behaved elements.} Under the current assumptions, all vertexes of the element $E$ are within $c h$ from the singularity corner point, $c \in {\mathbb R}_{+}$. 
Known results from the literature of VEM stabilization \cite{BLRstab,brenner2018,chen2018} yield
$$
|v|_{1,E}^2 \lesssim C s_E(v,v) \qquad \forall v \in \VE \, .
$$
Since the number of edges are uniformly bounded, using the above result applied to $u_I$ yields (the symbol ${\cal V}_E$ denoting again the set of vertices of $E$)
\begin{equation}\label{eq:serv:1}
|\IVh \zeta_j|_{1,E}^2 \lesssim  \max_{\nu \in {\cal V}_E} |\IVh \zeta_j(\nu)|^2
= \max_{\nu \in {\cal V}_E} |\zeta_j(\nu)|^2 \lesssim h^{2\gamma_j} \, ,
\end{equation}
where the last bound follows from \eqref{eq:cor:value-sing} with $V=\nu$ and the simple choice $P=Z$. \\

{\bf Almost well behaved elements.} 
 We recall, from the analysis of elements given in Sect. \ref{sec:stab:X}, that this kind of element can have edges $\overline{e}$ with length not comparable to $h_E$, and that (at least) one extrema $\overline{\nu}$ of each of such edges $\overline{e}$ lays on $\partial\Omega_h$. As done before, uniquely in order to simplify the exposition, we here assume the presence of only one of such ``small'' edges.
We denote by $\widehat{\nu}$ the other extrema of the ``small'' edge $\overline{e}$ and by $\varepsilon$ the length of $\overline{e}$. 
Applying Lemma 3.9 in \cite{brenner2018} to $\IVh \zeta_j$ yields, after some trivial manipulation and recalling that in the present analysis the polynomial degree $k=1$, leads to
$$
| \IVh \zeta_j |_{H^{1}(E)} \lesssim \| \IVh \zeta_j \|_{L^\infty(E)} + 
| \IVh \zeta_j |_{H^{1/2}(E)} \, . 
$$ 
Such bound, recalling that on $\partial E$ the function $\IVh \zeta_j$ is a continuous piecewise linear polynomial, yields after a direct calculation (as for the second equivalence in \eqref{eq:log-equiv})
$$
\begin{aligned}
| \IVh \zeta_j |_{H^{1}(E)}^2 & \lesssim \log{(1 + h/\varepsilon)} |\IVh \zeta_j(\widehat{\nu})|^2
\, + \!\!\! \sum_{\nu \in {\cal V}_E,  \,  \nu \not=\widehat{\nu}} |\IVh \zeta_j(\nu)|^2  \\
& = \log{(1 + h/\varepsilon)} |\zeta_j(\widehat{\nu})|^2
\, + \!\!\! \sum_{\nu \in {\cal V}_E,  \,  \nu \not=\widehat{\nu}} |\zeta_j(\nu)|^2 
\, .
\end{aligned}
$$
While the second term in the right hand side is bounded identically as in \eqref{eq:serv:1}, for the first term we apply \eqref{eq:cor:value-sing} with $V=\widehat{\nu}$ and $P=\overline{\nu} \in \partial \Omega_h$. We obtain 
$$
| \IVh \zeta_j |_{H^{1}(E)}^2 \lesssim h^{2\gamma_j} + \log{(1 + h/\varepsilon)} h^{2\gamma_j-2}\varepsilon^2
\lesssim h^{2\gamma_j} \, .
$$

{\bf Anisotropic quadrilaterals.} 
In the present proof we follow the geometric notation for anisotropic quads previously introduced, see also Figure \ref{fig:quad:X}. Furthermore we denote by $P_1$ and $P_2$ the extrema of edge $\widehat{e}$, as shown in Figure \ref{fig:quad:X} (note that $P_i$ is adjacent to the horizontal edge of length $\varepsilon_i$, $i=1,2$).
We recall that both $\varepsilon_1$ and $\varepsilon_2$ are smaller than $h$.
By applying \eqref{eq:cor:value-sing} with $V=P_i$, $i\in \{1,2\}$, and $P \in \partial \Omega_h$ as the other extrema of the associated horizontal edge, we easily obtain
\begin{equation}\label{nnw:init}
u(P_1) \le h^{\gamma_j-1} \varepsilon_1 \ , \qquad u(P_2) \le h^{\gamma_j-1} \varepsilon_2 \, .
\end{equation} 
By Lemma \ref{lemma:coercivity_new_stab} and Definition \ref{def:gamma-stabilization} we have (with the usual coordinate notation)
$$
| \IVh \zeta_j |_{H^1(E)} \lesssim s_E(\IVh \zeta_j,\IVh \zeta_j) = \int_0^h \frac{\vert \IVh \zeta_j(0,y) \vert^2}{\kappa(y)} {\rm d}y \, .
$$
Recalling that $\IVh \zeta_j$ is affine on the edge $\widehat{e}$ and interpolates $\zeta_j$ at the extrema, together with the definition of $\kappa$ (see \eqref{eq:def-gammay}), we can write
$$
| \IVh \zeta_j |_{H^1(E)}^2 \lesssim \int_0^h \frac{ \big(\zeta_j(P_1) y/h + \zeta_j(P_2)(h-y)/h \big)^2}{(h-y)(\varepsilon_2-\varepsilon_1)/h + \varepsilon_1} {\rm d}y
= h \int_0^1 \frac{\big( \hat{y} \, \zeta_j(P_1) + (1-\hat{y}) \zeta_j(P_2) \big)^2}{(1-\hat{y})(\varepsilon_2-\varepsilon_1) + \varepsilon_1} {\rm d}\hat{y} \, ,
$$
where we applied a trivial change of variables in the last identity.
The above bound yields immediately
\begin{equation}\label{nnw:1}
| \IVh \zeta_j |_{H^1(E)}^2 \lesssim h \int_0^1 \frac{( \hat{y} \, \zeta_j(P_1)  )^2}{(1-\hat{y})(\varepsilon_2-\varepsilon_1) + \varepsilon_1} {\rm d}\hat{y} 
\, + \, 
h \int_0^1 \frac{( (1-\hat{y}) \zeta_j(P_2) )^2}{(1-\hat{y})(\varepsilon_2-\varepsilon_1) + \varepsilon_1} {\rm d}\hat{y}  =: T_1 + T_2 \, .
\end{equation}
Term $T_1$ is easy to bound noting that $(1-\hat{y})(\varepsilon_2-\varepsilon_1) + \varepsilon_1 \ge \hat{y}\,\varepsilon_1 $ and recalling \eqref{nnw:init} 
$$
T_1 \le \zeta_j(P_1)^2 h \int_0^1 \frac{\hat{y}^2}{\hat{y}\varepsilon_1} {\rm d}\hat{y}  \le (h^{\gamma-1} \varepsilon_1)^2 h \, (\varepsilon_1)^{-1} 
\le h^{2 \gamma} \, \left(\frac{\varepsilon_1}{h} \right) \le h^{2 \gamma}  \, .
$$
Regarding term $T_2$ we note that $(1-\hat{y})(\varepsilon_2-\varepsilon_1) + \varepsilon_1 \ge (1-\hat{y}) \varepsilon_2$ and again recall \eqref{nnw:init}, 
leading to
$$
T_2 \le \zeta_j(P_2)^2  h \int_0^1 \frac{(1-\hat{y})^2}{(1-\hat{y}) \varepsilon_2}  {\rm d}\hat{y}  \le (h^{\gamma-1} \varepsilon_2)^2 h (\varepsilon_2)^{-1}
\le h^{2 \gamma} \, \left(\frac{\varepsilon_2}{h} \right) \le h^{2 \gamma}  \, .
$$
The estimate on $|\IVh \zeta_j |_{1,E}$ follows by combining \eqref{nnw:1} with the above bounds for $T_1$ and $T_2$.

Since all possible geometries have been taken into consideration, the proof of Theorem \ref{thm:error-nrc} is complete.

Next, we state the $L^2$-interpolation error bound in the non-regular case, and sketch its proof.
\begin{theorem}[$L^2$-interpolation error bound in the non-regular case]\label{thm:L2-error-nrc} There exists a constant $C>0$ independent of $u$ and ${\cal T}_h$ such that
\begin{equation}\label{eq:L2-error-nrc}
\| u-\IVh u \|_{0,\Omega_h} \leq C \, h^{\bar{\gamma}+1} \enorm{u}\,.    
\end{equation}
\end{theorem}
\begin{proof}

Following the same path of the proof of Theorem \ref{thm:error-nrc}, one is left with the problem of bounding the norms $\| \zeta_j \|_{0,E}$ and $\|\IVh \zeta_j \|_{0,E}$ in each element $E \in \mesh$ that intersects the neighborhood ${\cal N}_j(\bz_j,r_j)$ of $\bz_j$.
The first norm is easily bounded since
\[
\| \zeta_j\|_{0,E}^2 \lesssim \int_0^{ch} r^{2\gamma_j+1} {\rm d}r \lesssim h^{2\gamma_j+2}.
\]
We already proved that the second norm satisfies 
\[
\|\IVh \zeta_j\|_{0,E} \lesssim  
h^{\gamma_j+1}
\]
when $E$ ia a triangular element. On the other hand,   
when $E$ is an (almost) well-behaved element or an anisotropic quadrilateral element, we employ Propositions \ref{prop:shapereg} and \ref{prop:ani:C123} to obtain
\[    
\|\IVh \zeta_j\|^2_{0,E}\leq \vert E \vert \| \IVh \zeta_j \|^2_{L^{\infty}(E)} \lesssim \vert E \vert \| \IVh \zeta_j \|^2_{L^{\infty}(\partial E)}=\vert E \vert \max_{\nu \in {\cal V}_E} \vert\zeta_j(\nu)\vert^2.
\]
As for any point $P\in {\cal N}_j(\bz_j,r_j)$  we have 
$\zeta_j(P)\lesssim h^{\gamma_j}$ we get 
\[
\|\IVh \zeta_j\|^2_{0,E}\leq h^{2+2\gamma_j},
\]
and the thesis is proven.
\end{proof}

\begin{remark}\label{rem:polapprox:sing} {\rm
Checking Assumption \ref{ass:polapprox} (polynomial approximation) for the parts of the domain near the corner singularities  follows the same identical steps as for the interpolation operator $\IVh$. More specifically, in the annulus $A_j$ we make use of the polynomial approximation introduced in Remark \ref{rem:polapprox:H2}, yielding by the same identical argument
\begin{equation}
| \zeta_j -{\cal P}^1_h \zeta_j |_{1,\Omega_h \cap A_j} \lesssim h^{\gamma_j}.   
\end{equation}
For the elements that intersect ${\cal N}_j(\bz_j,r_j)$, we simply choose ${\cal P}^1_h \zeta_j |_E = 0$ 
so that we can immediately apply \eqref{eq:Falcao} and conclude.
}
\end{remark}

\section{Numerical results}\label{sec:numerics}

In this section, we present two numerical tests to validate the theoretical results of Theorems~\ref{th:energy-error} and \ref{th:L2-error}.
To compute the VEM error between the exact solution $u$ and the VEM solution $u_h$, we consider the computable error quantities 
\begin{equation}
    \label{eq:err_quant}
\begin{aligned}
\texttt{err}(u_h, H^1) &=
\frac{\sum_{E \in \mathcal{T}_h} \left( | u -  \Pi^{\nabla}  u_h|_{1, E}^2 \right)^{1/2}}{|u |_{1, \Omega_h}}  \,, 
\\
\texttt{err}(u_h, L^2) &=
\frac{\| u -  \Pi^{\nabla}  u_h\|_{0, \Omega_h}}{ \|u\|_{0, \Omega_h}}  \,.
\end{aligned}
\end{equation}
Given a sequence of $N+1$ meshes with mesh diameters $h_0 > \dots > h_N$, and denoting by $E_h$ any of the error quantities listed in \eqref{eq:err_quant}, we define the experimental order of convergence \texttt{EOC}
\begin{equation}
\label{eq:err_eoc}
\texttt{EOC}(n) = 
\frac{\log(E_{h_{n-1}}/E_{h_{n}})}{\log(h_{n-1}/h_{n})} 
\qquad \text{for $n=1, \dots, N$.}
\end{equation}
The the average experimental order of convergence \texttt{AEOC} is defined as
\begin{equation}
\label{eq:err_aeoc}
\texttt{AEOC}= \frac{1}{N} \sum_{n=1}^N 
\texttt{EOC}(n) \,.
\end{equation}
For simplicity, in the forthcoming tests we consider the Poisson problem with general Dirichlet boundary conditions. Specifically, we consider the following elliptic boundary value problem:
\begin{equation}
\label{eq:poisson_test}
\left\{
\begin{aligned}
-\Delta u &= f &\quad &\text{in $\Omega$,} \\
u &= g         &\quad &\text{on $\partial \Omega$,}
\end{aligned}
\right .
\end{equation}
where the domain $\Omega$, the source term $f$ and the boundary value $g$ will be specified in each test. Recall that the numerical analysis for the non-homogeneous boundary-value problem can be reduced to the one given in the previous sections for the vanishing boundary condition, by applying the change of variable $u=u_0+u_g$ where $u_g$ is a suitable lifting of $g$ to the whole domain $\Omega$. 

In order to assess the robustness of the VEM with respect to the elements generated by the proposed cutting procedure, in the forthcoming tests we systematically adopt the so-called dofi–dofi stabilization (see \cite{BLRstab,brenner2018,chen2018}) in all cases, including elements with small edges or highly anisotropic polygons (unless explicitly stated otherwise in Test \ref{sec:test2}).
%

\subsection{Test 1: VEM approximation on a curved domain}
\label{sec:test1}

The purpose of this numerical experiment is to assess the performance of the  VEM when the discrete computational domain $\Omega_h$ differs from the exact domain $\Omega$.

To this end, we consider the Poisson problem \eqref{eq:poisson_test}, where $\Omega$ is a quarter of the unit disk, that is,
$0 \leq r \le 1$, $0 \leq \theta \leq \frac{\pi}{2}$.
Our goal is to construct an explicit solution $u$ that exhibits a boundary layer near the curved boundary $r=1$. 
To do so, we look for a solution in separated form in polar coordinates:
\[
u(r,\theta) = \varphi(r)\psi(\theta).
\]
We choose
$
\psi(\theta) = \sin(4\theta)
$,
since this choice automatically satisfies the homogeneous Dirichlet boundary conditions
$
\psi(0)=\psi\left(\frac{\pi}{2}\right)=0
$.
%
%
%
We now turn to the radial part. We look for $\varphi$ in the form
\[
\varphi(r) = r^2 \eta(r)\,,
\]
where $\eta$ is chosen to model a boundary layer near $r=1$.
More precisely, we assume that $\eta$ solves the one-dimensional problem
\begin{equation}
\label{eq:nu_test1}
-\frac{1}{\nu}\eta''(r) + \eta'(r) = 0, \quad r \in (0,1),
\end{equation}
with boundary conditions
$
\eta(0)=1$, $\eta(1)=0
$,
where $\nu>1$ is the Péclet number. This choice ensures a sharp variation near $r=1$.
The explicit solution of this problem is
\[
\eta(r) = \frac{e^{\nu} - e^{\nu r}}{e^{\nu} - 1}\,.
\]
Note that, by construction,
$
\varphi(0)=0$ and $\varphi(1)=0
$,
so that the full solution $u$ satisfies the boundary conditions on $\partial \Omega$.
%
%
%
%
%
%
Computing the Laplacian in polar coordinates,
%
%
we obtain that the term $f$ in Problem \eqref{eq:poisson_test} is
\[
f = - \Delta u=
-\varphi''(r)\psi(\theta)
- \frac{1}{r}\varphi'(r)\psi(\theta)
- \frac{1}{r^2}\varphi(r)\psi''(\theta)
=
\left[60\eta(r)
- r(5 + \nu r)\eta'(r)
\right]\sin(4\theta).
\]

In order to assess the proposed  VEM scheme, we consider a Cartesian partition of the square domain $[0,1]^2$ cut by the  curved portion of $\partial\Omega$.
The configuration is shown in Fig. \ref{fig:meshes_c}. The red line  represents the curved portion of $\partial \Omega$, whereas black lines define the mesh elements and the boundary of the discrete domain $\Omega_h$.  We  remark that the trimming procedure automatically generates polygons with up to six edges, and/or elements with small edges. 
We also note that, in the present case, the inclusion $\Omega_h \subset \Omega$ is strict. Further, $\Omega_h$ and $\Omega$ differ by at most a region of width $O(h^2)$.

\begin{figure}[!h]
	\centering
\includegraphics[width=0.30\textwidth]{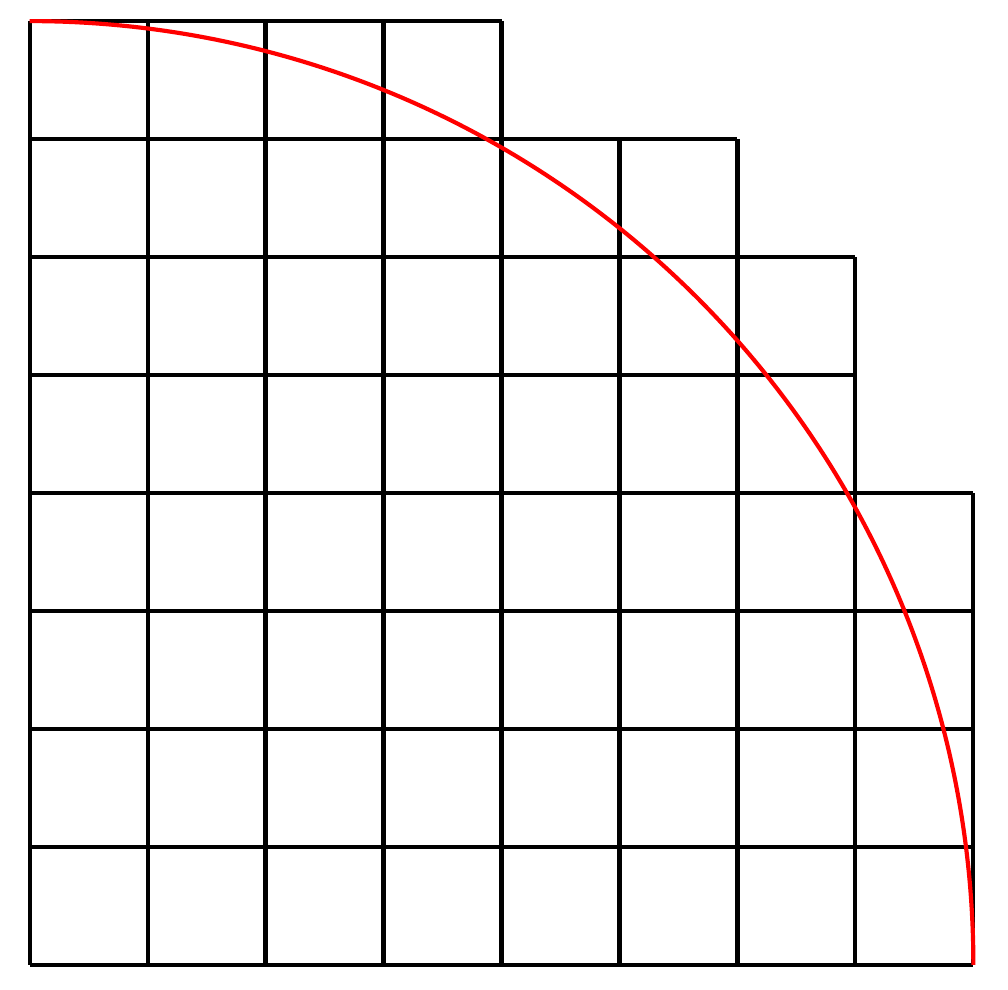}
\quad  
\includegraphics[width=0.30\textwidth]{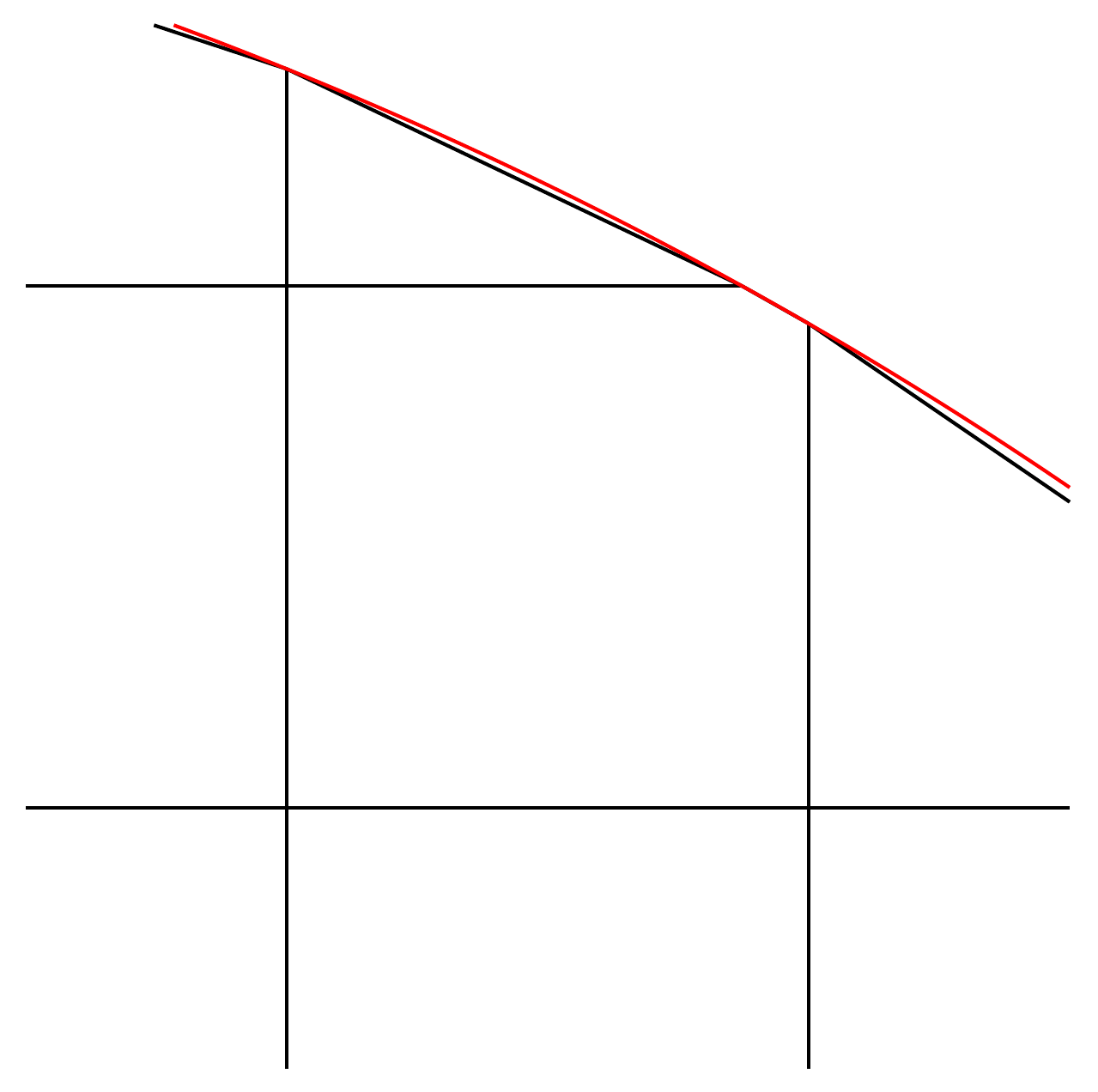}
\quad  
\includegraphics[width=0.30\textwidth]{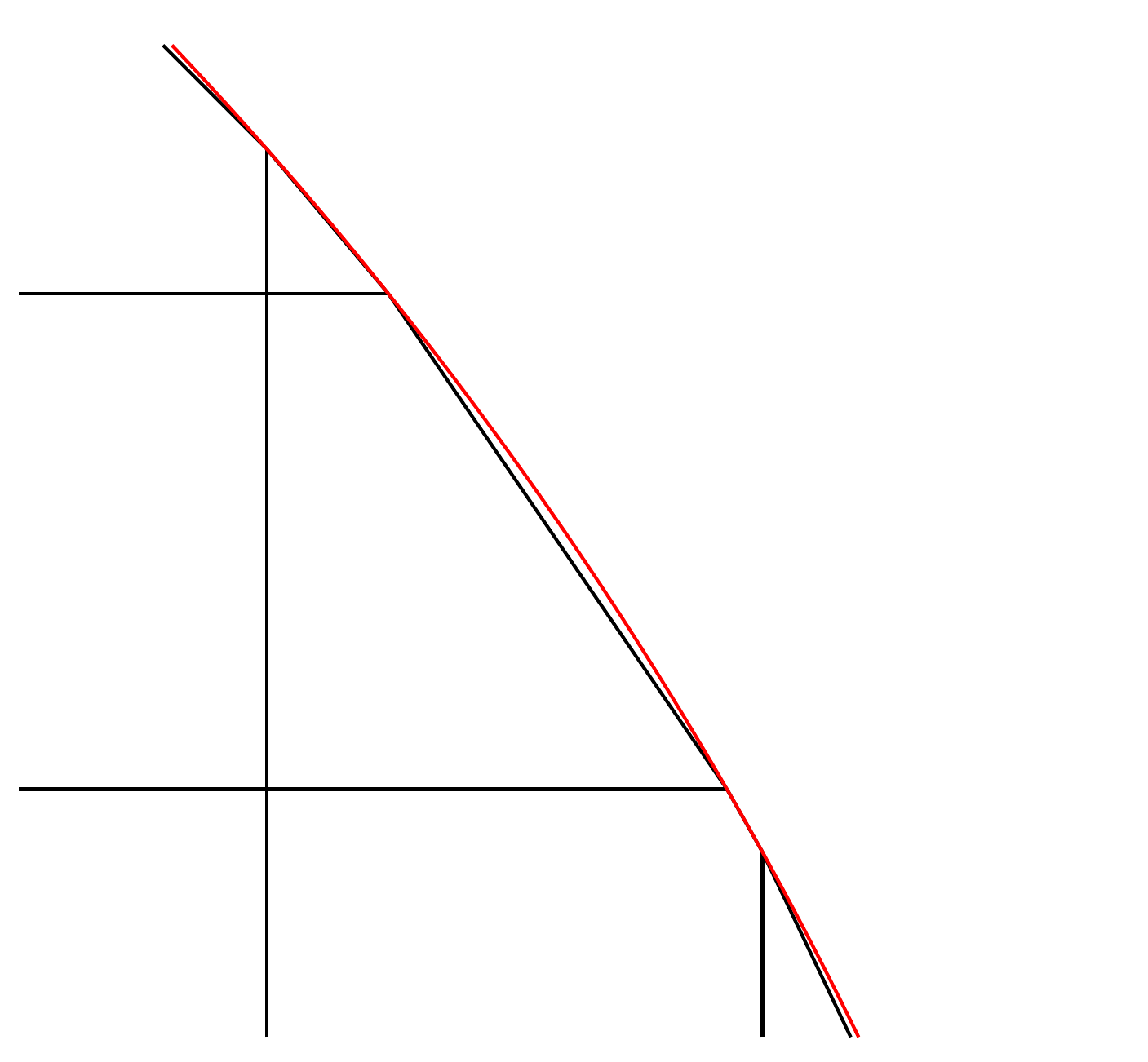}
	\caption{Test 1. Left: configuration of the domain $\Omega$ and the tessellation $\mathcal{T}_h^B(\Omega)$, the red curve represents the curved portion of $\partial \Omega$; 
Middle: detail of $\mesh$ with a polygon with five edges;
Right: detail of $\mesh$, to highlight that $\Omega_h \subset \Omega$ is a strict inclusion.}
\label{fig:meshes_c}
\end{figure}
In Fig.~\ref{fig:Test1}, we show the discrete solutions $u_h$ obtained on a Cartesian tessellation of $[0,1]^2$ with $64 \times 64$ elements, for $\nu=1, 32, 64, 128$ (cf.~\eqref{eq:nu_test1}).
As expected, the solutions $u_h$ exhibit a boundary layer near the curved boundary when $\nu$ becomes larger.
\begin{figure}[h!]
	\centering
\subfloat[$u_h$ with \texttt{$\nu=1$}]{\includegraphics[width=0.475\linewidth]{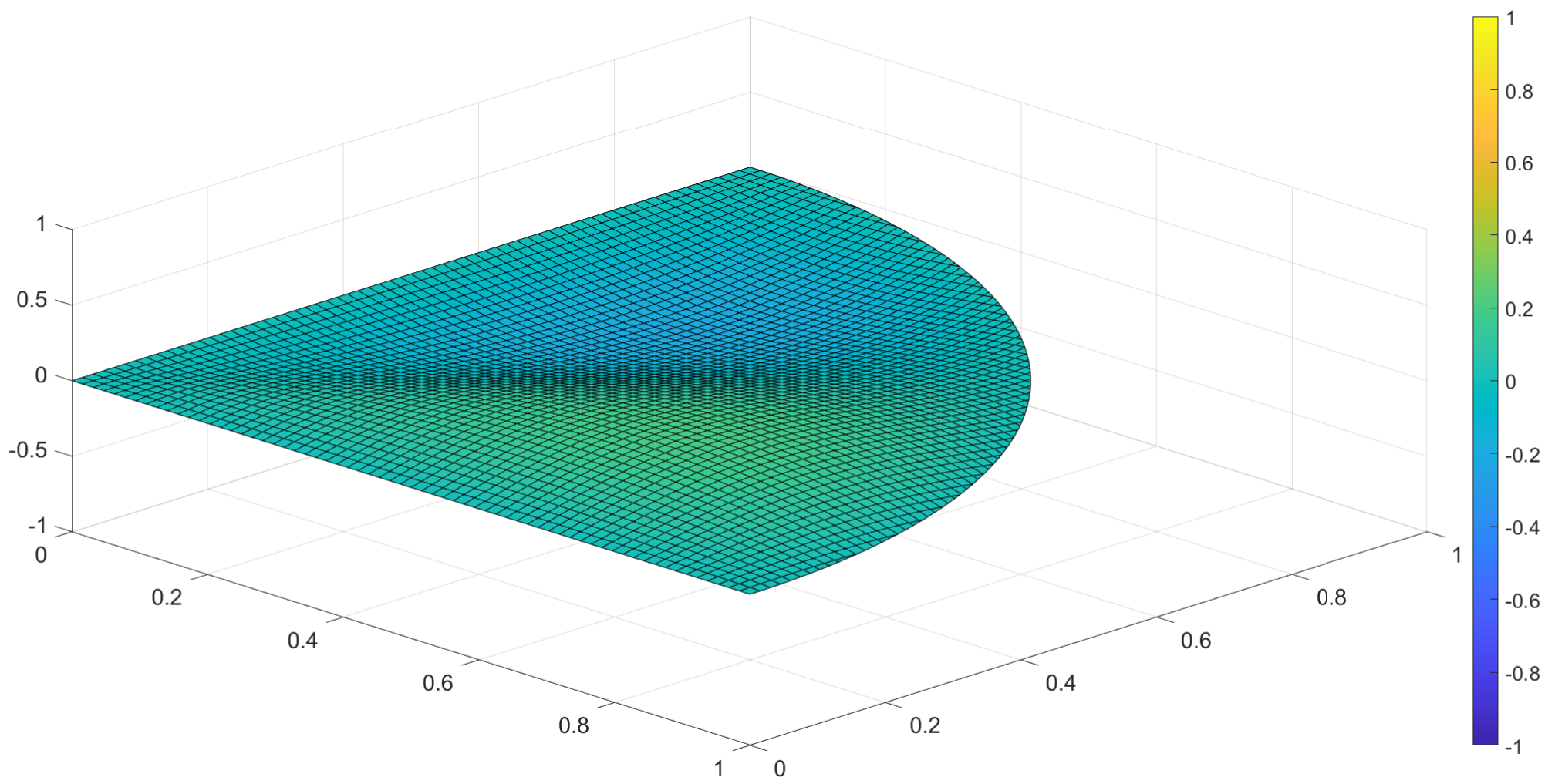}}\hfill
\subfloat[$u_h$ with \texttt{$\nu=32$}]{\includegraphics[width=0.475\linewidth]{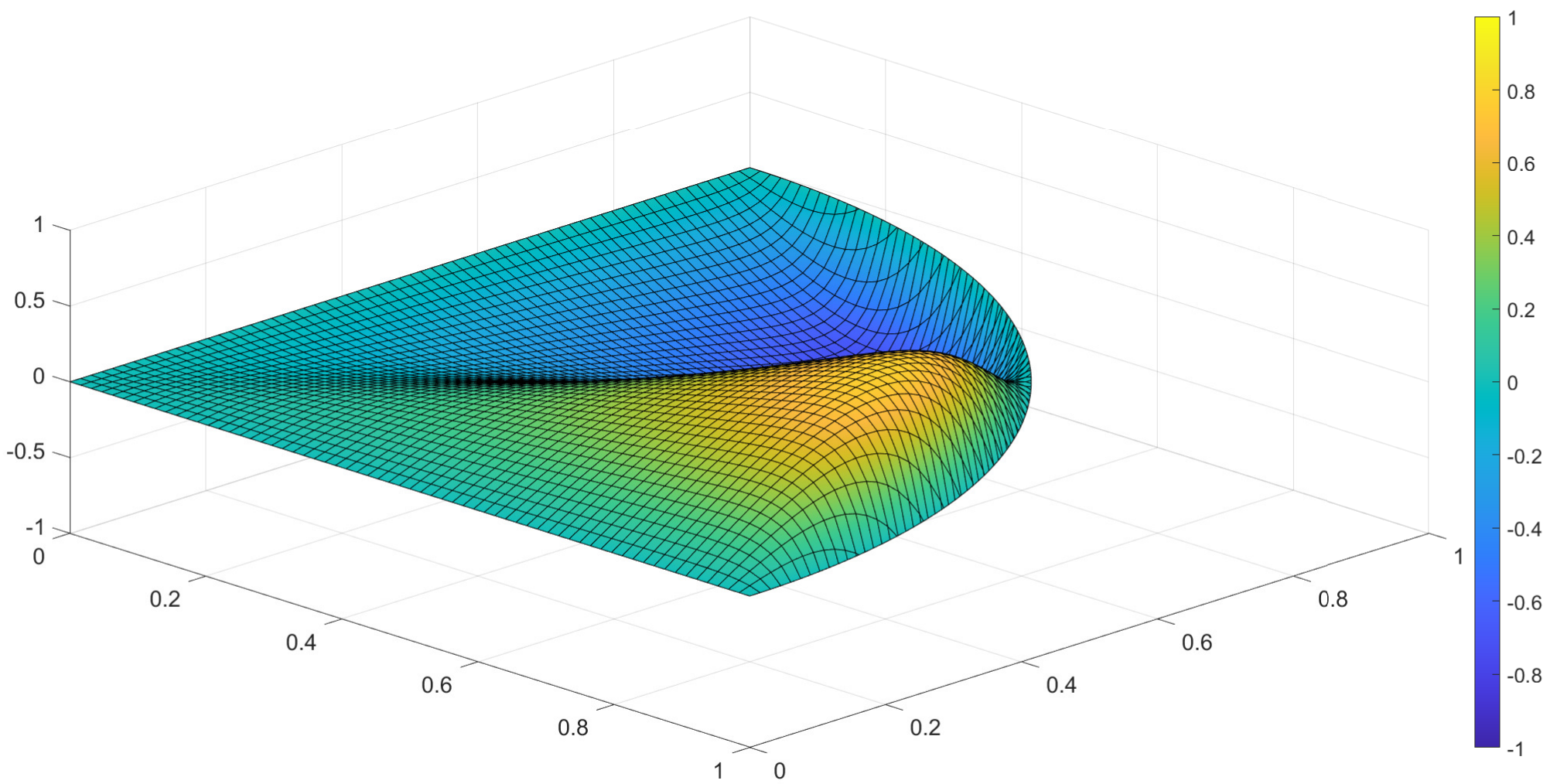}}
\\
\subfloat[$u_h$ with \texttt{$\nu=64$}]{\includegraphics[width=0.475\linewidth]{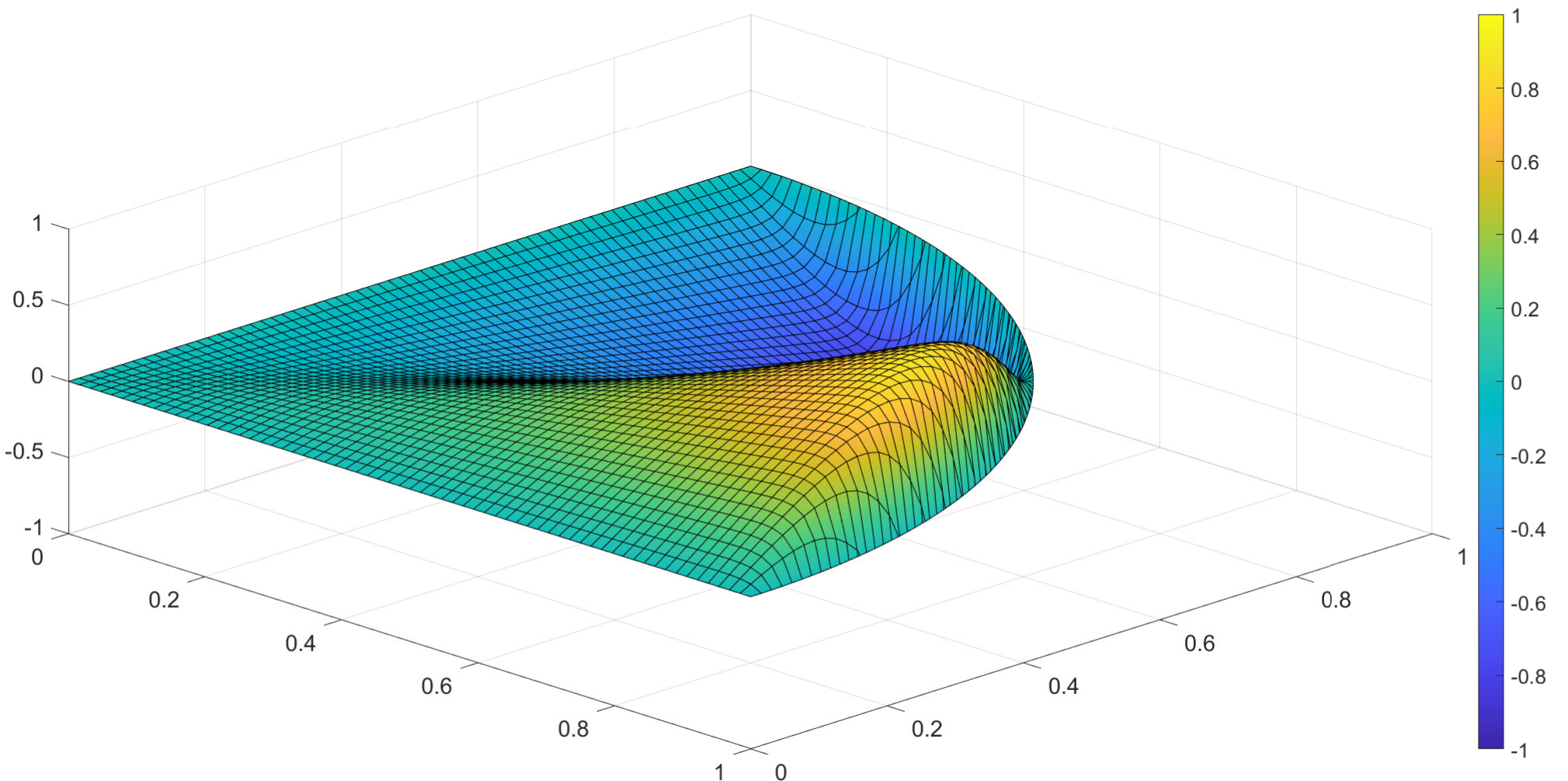}}
\hfill
\subfloat[$u_h$ with \texttt{$\nu=128$}]{\includegraphics[width=0.475\linewidth]{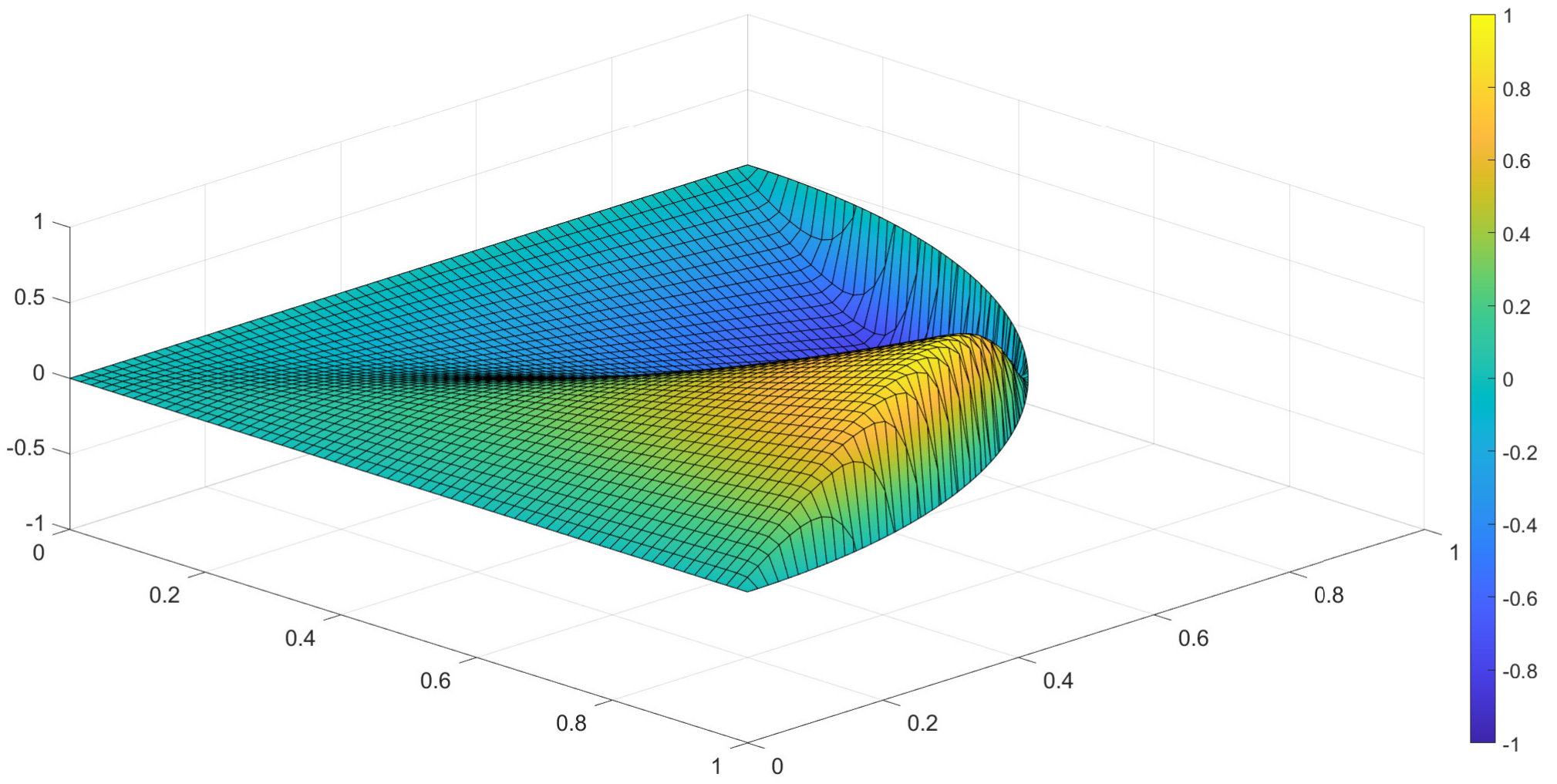}}
	\caption{Test 1. Discrete solutions $u_h$ obtained on a Cartesian tessellation of $[0,1]^2$ with $64 \times 64$ elements.}
	\label{fig:Test1}
\end{figure}

Table~\ref{tab:test1-conv} reports the computed error quantities in \eqref{eq:err_quant} and the experimental order of convergence \texttt{EOC} in \eqref{eq:err_eoc} for a sequence of Cartesian meshes on $[0,1]^2$ with $\texttt{N}\times \texttt{N}$ elements for different values of  $\nu$ (the adopted values of $\texttt{N}$ and $\nu$ are reported in the table). 
%



\begin{table}[!ht]
\centering
\begin{small}
\begin{tabular}{l|r|cr|cr|cr|cr}
\toprule
&
& \multicolumn{8}{c}{\texttt{$\nu$}} \\
&
& \multicolumn{2}{c}{\texttt{1}}
& \multicolumn{2}{c}{\texttt{32}}
& \multicolumn{2}{c}{\texttt{64}}
& \multicolumn{2}{c}{\texttt{128}} \\
 \cmidrule(lr){3-10}
   \texttt{ERROR}
&  \texttt{N}
& \texttt{err} & \texttt{EOC}
& \texttt{err} & \texttt{EOC} 
& \texttt{err} & \texttt{EOC} 
& \texttt{err} & \texttt{EOC} 
\\
\midrule
\multirow{6}{*}{\texttt{$H^1$-err}} 
&  \texttt{8}
& \texttt{2.904e-1} & 
& \texttt{5.897e-1} & 
& \texttt{7.439e-1} & 
& \texttt{8.596e-1} & 
\\
&  \texttt{16}
& \texttt{1.480e-1} & \texttt{0.97}
& \texttt{3.952e-1} & \texttt{0.57}
& \texttt{5.642e-1} & \texttt{0.39}
& \texttt{7.250e-1} & \texttt{0.24}
\\
&  \texttt{32}
& \texttt{7.490e-2} & \texttt{0.98}
& \texttt{2.414e-1} & \texttt{0.71}
& \texttt{3.855e-1} & \texttt{0.54}
& \texttt{5.600e-1} & \texttt{0.37}
\\
&  \texttt{64}
& \texttt{3.762e-2} & \texttt{0.99}
& \texttt{1.353e-1} & \texttt{0.83}
& \texttt{2.352e-1} & \texttt{0.71}
& \texttt{3.844e-1} & \texttt{0.54}
\\
&  \texttt{128}
& \texttt{1.885e-2} & \texttt{0.99} 
& \texttt{7.173e-2} & \texttt{0.91}
& \texttt{1.311e-1} & \texttt{0.84}
& \texttt{2.328e-1} & \texttt{0.72}
\\
&  \texttt{256}
& \texttt{9.434e-3} & \texttt{0.99}
& \texttt{3.696e-2} & \texttt{0.95}
& \texttt{6.946e-2} & \texttt{0.91}
& \texttt{1.296e-1} & \texttt{0.84}
\\
\midrule
\multirow{6}{*}{\texttt{$L^2$-err}} 
&  \texttt{8}
& \texttt{8.342e-2} & 
& \texttt{1.902e-1} & 
& \texttt{2.549e-1} & 
& \texttt{3.078e-1} & 
\\
&  \texttt{16}
& \texttt{2.133e-2} & \texttt{1.96}
& \texttt{6.967e-2} & \texttt{1.44}
& \texttt{1.110e-1} & \texttt{1.19}
& \texttt{1.563e-1} & \texttt{0.97}
\\
&  \texttt{32}
& \texttt{5.461e-3} & \texttt{1.96}
& \texttt{2.365e-2} & \texttt{1.55}
& \texttt{4.408e-2} & \texttt{1.33}
& \texttt{7.369e-2} & \texttt{1.08}
\\
&  \texttt{64}
& \texttt{1.377e-3} & \texttt{1.98}
& \texttt{7.042e-3} & \texttt{1.74}
& \texttt{1.493e-2} & \texttt{1.56}
& \texttt{2.969e-2} & \texttt{1.31}
\\
&  \texttt{128}
& \texttt{3.454e-4} & \texttt{1.99} 
& \texttt{1.915e-3} & \texttt{1.87}
& \texttt{4.378e-3} & \texttt{1.77}
& \texttt{9.864e-3} & \texttt{1.59}
\\
&  \texttt{256}
& \texttt{8.647e-5} & \texttt{1.99}
& \texttt{5.006e-4} & \texttt{1.93}
& \texttt{1.193e-3} & \texttt{1.87}
& \texttt{2.900e-3} & \texttt{1.76}
\\
\bottomrule
\end{tabular}
\end{small}
\caption{Test 1. Computed errors $\texttt{err}(u_h, H^1)$ (top) and $\texttt{err}(u_h, L^2)$ (bottom)
as in \eqref{eq:err_quant} for the a Cartesian tessellation of $[0,1]^2$ into $\texttt{N}\times \texttt{N}$ elements and different values of $\nu$ (cf. \eqref{eq:nu_test1}).}
\label{tab:test1-conv}
\end{table}


We observe that for $\nu=\texttt{1}$ we recover the rate of convergence predicted in Theorems~\ref{th:energy-error} and \ref{th:L2-error}. For larger values of $\nu$, as expected, the correct rate is observed only when $\texttt{N}$ is sufficiently large to resolve the boundary layer of the exact solution $u$.


\subsection{Test 2: VEM approximation on anisotropic elements}
\label{sec:test2}

The aim of the second test is to assess the performance of the proposed method and to validate the theoretical results in~Theorems~\ref{th:energy-error} and \ref{th:L2-error} in the case of a parametrized family of domains with a reentrant corner and in the presence of elements with poor aspect ratio.
To this end, we consider the Poisson problem \eqref{eq:poisson_test} posed on the L-shaped domain
\begin{equation}
\label{eq:omega_t}
\Omega = \Omega(t) = (-1+t, 1+t)^2 \setminus [t, 1+t]^2 \qquad
\text{for $t \in [0, 1]$.} 
\end{equation}
The loading term is $f=0$, and therefore the exact solution $u$ of Problem~\eqref{eq:poisson_test}, expressed in polar coordinates, is given by
\[
u(r, \theta) = r^{2/3} \, \sin\left(\frac{2}{3} \theta \right) \,, \qquad 0 \leq \theta < \frac{3\pi}2.
\]
Here, $r$ and $\theta$ denote the polar coordinates with respect to the vertex of the reentrant corner $(t, t)$ and the vertical positive semi-axis.

Let us consider the extended domain $\widetilde{\Omega} = (-1, 2)^2$ and note that $\Omega(t) \subset \widetilde{\Omega}$ for all values of the parameter $t \in [0,1]$.
We further observe that the load $f$ trivially extends to 0 on $\widetilde{\Omega}$,
whereas we extend with continuity the exact solution  $u$, namely $u|_{\widetilde{\Omega}\setminus \Omega}=0$.

To discretize the problem, we consider a sequence of Cartesian meshes centered at the origin $(0,0)$, which are in general not aligned with the boundary of $\Omega$, and in particular with the vertex of the reentrant corner $(t,t)$.

In Fig.~\ref{fig:options_l}, we show three possible configurations of the discrete domain $\Omega_h$ and the corresponding meshes.
The first configuration, labelled as \texttt{OPT A}, coincides with the one proposed in Sect~\ref{sec:meshes}, Remark \ref{exa:background-mesh}, i.e. the reentrant corner in $\Omega$ is approximated in $\Omega_h$ by the convex hull containing it. 
In the second and third configurations, labelled as \texttt{OPT B} and \texttt{OPT C}, respectively, the computational domain coincides with the exact domain. However, the adopted meshes differ: in \texttt{OPT B}, the element containing the reentrant corner is split into two elements by connecting the vertex $(t,t)$ with the lower-left vertex of the background mesh; in \texttt{OPT C}, the nonconvex element containing the reentrant corner is retained.
We observe that the three approaches coincide whenever $(t,t)$ is a vertex of the background mesh.

\begin{figure}[t!]
	\centering
\includegraphics[width=0.30\textwidth]{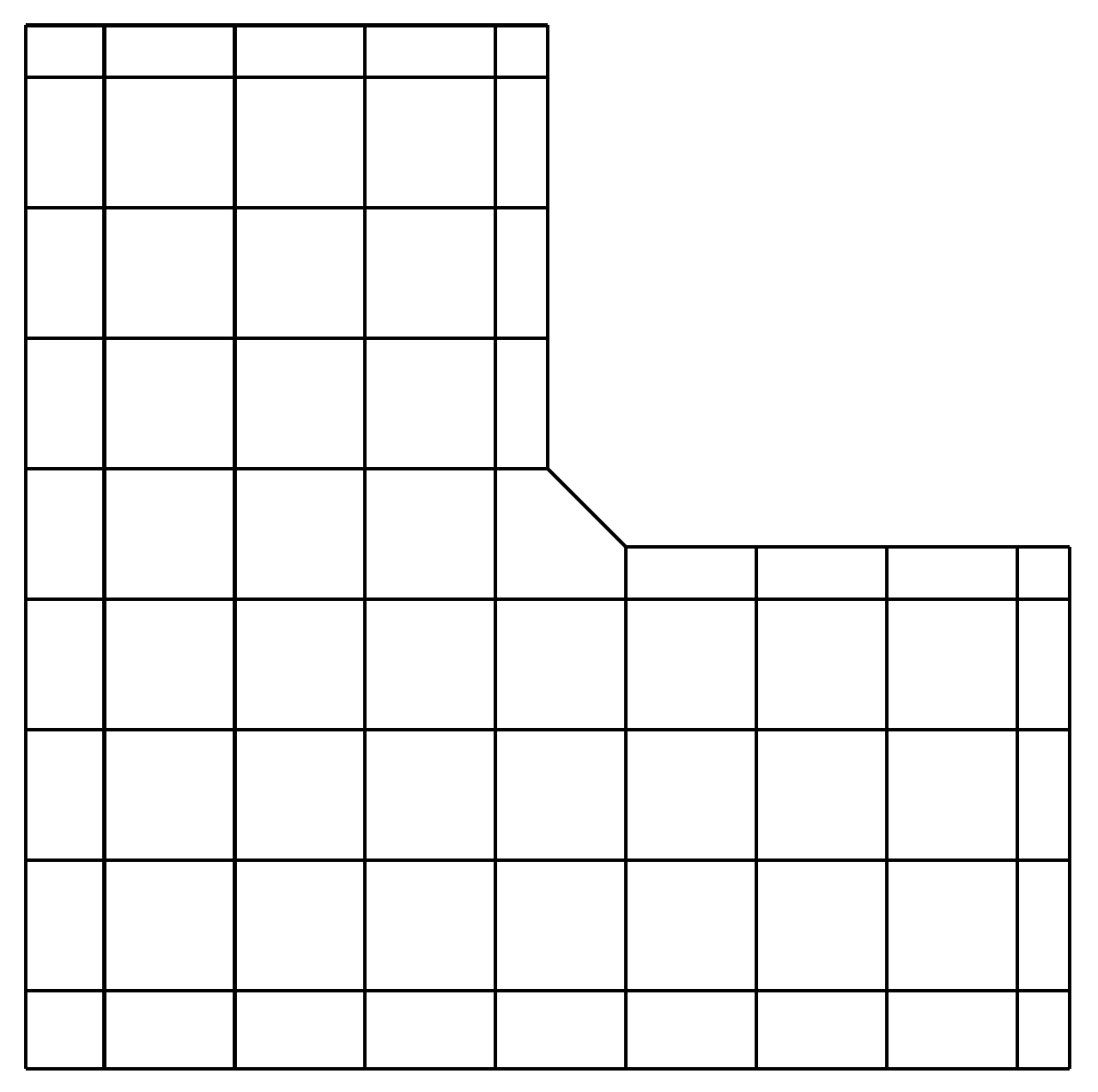}
\quad 
\includegraphics[width=0.30\textwidth]{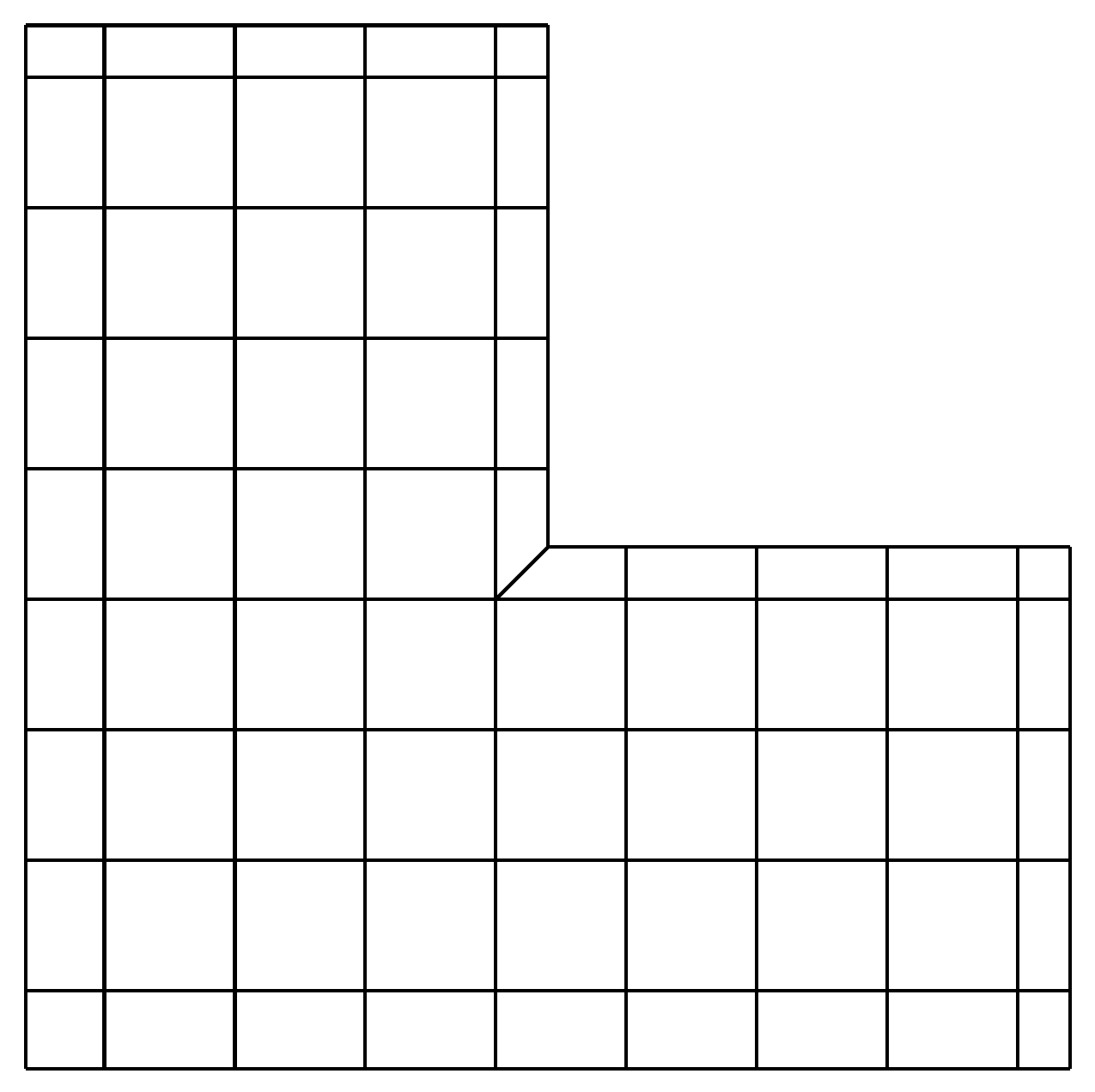}
\quad 
\includegraphics[width=0.30\textwidth]{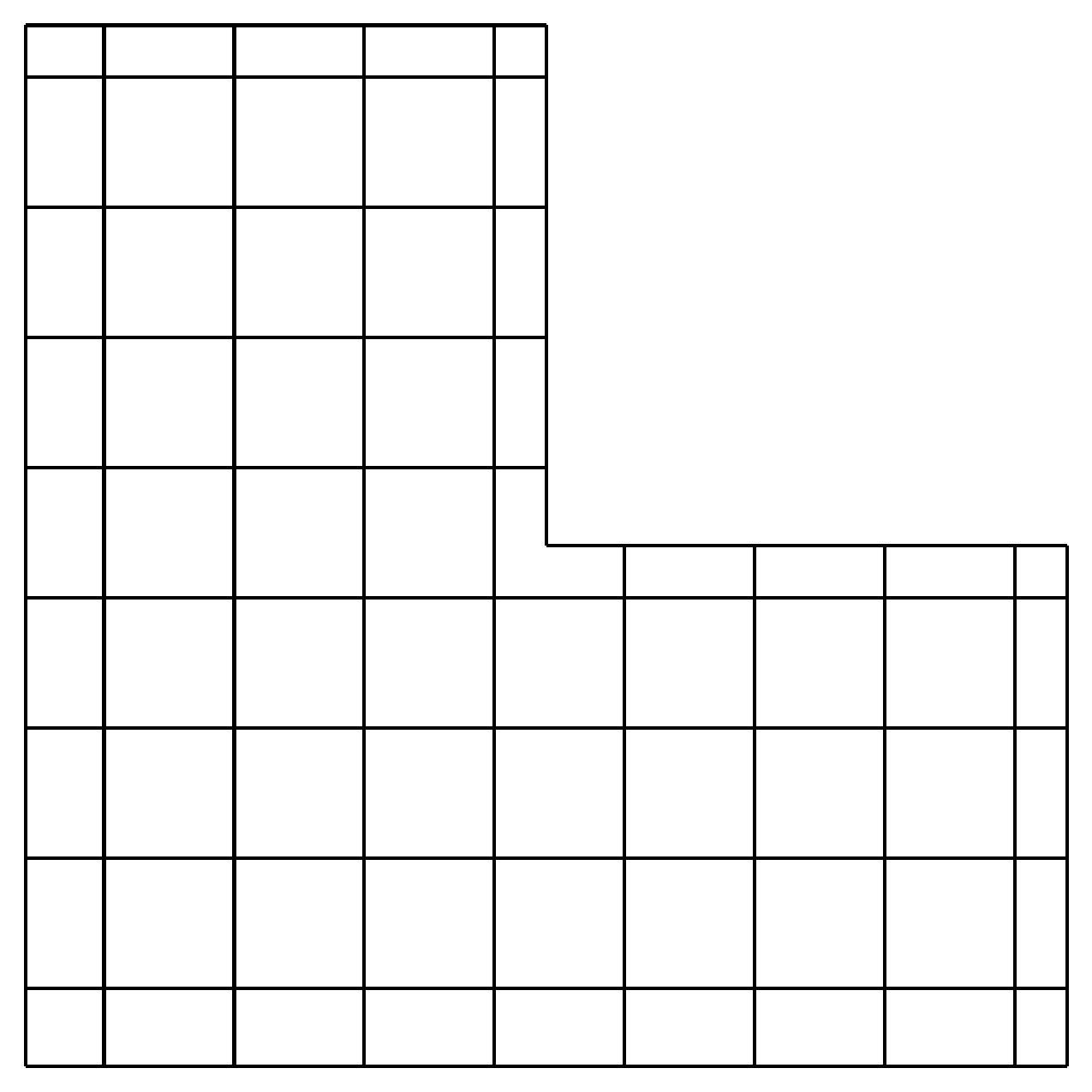}
	\caption{Test 2. Three configurations of the discrete domain $\Omega_h$ and the corresponding meshes. Left: \texttt{OPT A}, where the reentrant corner is approximated by its convex hull; Middle: \texttt{OPT B}, where the element containing the reentrant corner is split into two elements; Right: \texttt{OPT C}, where the nonconvex element containing the reentrant corner is retained.}
\label{fig:options_l}
\end{figure}
We first consider the domain $\Omega = \Omega(\tfrac1{10})$ (cf. \eqref{eq:omega_t}).
In Fig. \ref{fig:solution_l} we show the plot of the discrete solution $u_h$ obtained on a Cartesian background tessellation having mesh size $\texttt{1/16}$ with \texttt{OPT A}. 

\begin{figure}[!h]
	\centering
\includegraphics[scale=0.30]{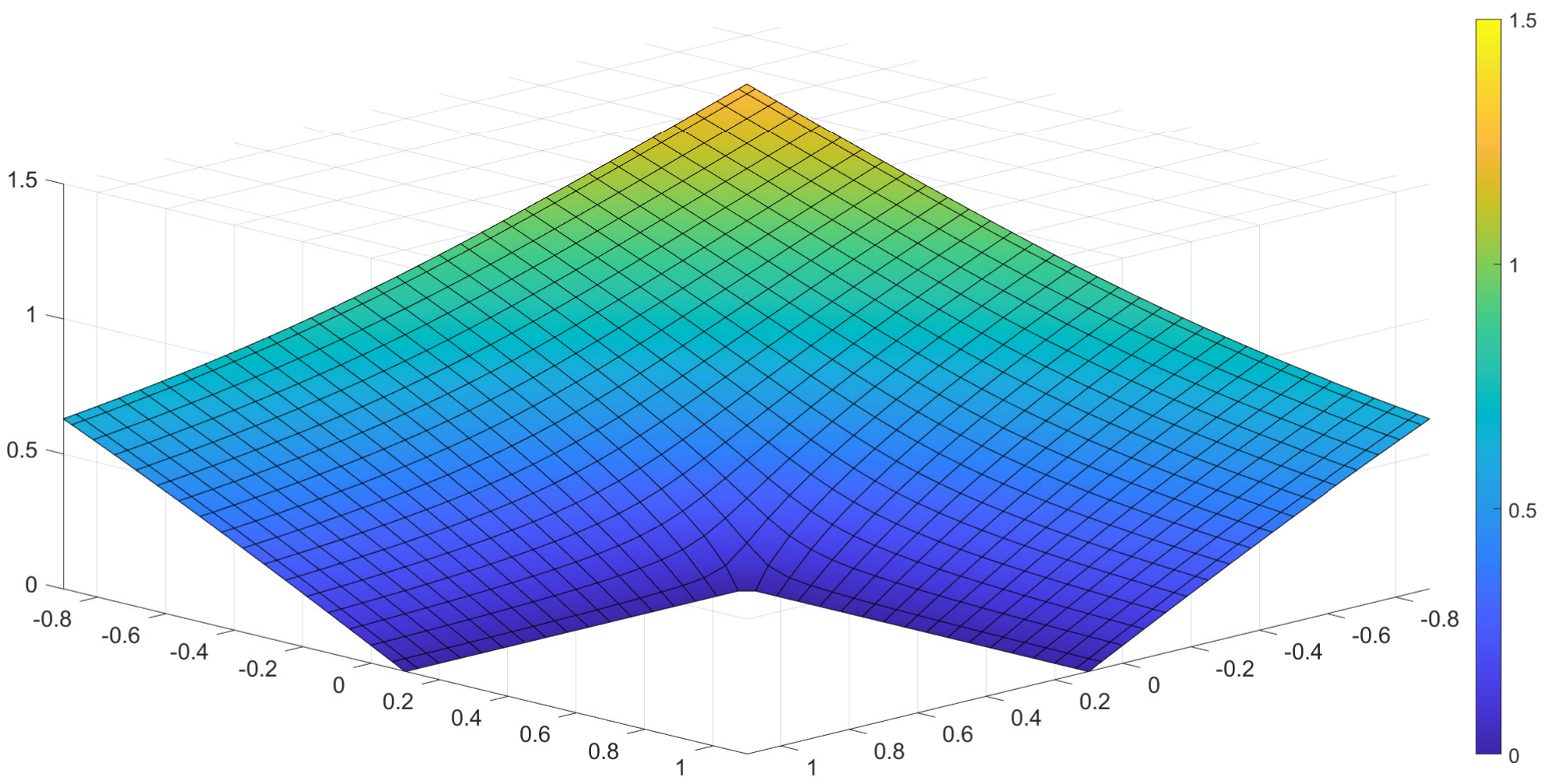}
	\caption{Test 2. Discrete solutions $u_h$ obtained on a Cartesian tessellation having mesh size \texttt{1/16} with \texttt{OPT A}.}
\label{fig:solution_l}
\end{figure}

In Table \ref{tab:test2-conv} we report the computed error quantities in \eqref{eq:err_quant}, the experimental order of convergence \texttt{EOC}, and the average experimental order of convergence \texttt{AEOC} for a sequence of Cartesian meshes with mesh size $\texttt{1/h}$ (the corresponding values of $\texttt{h}$ are listed in the table), for the three options described above.
To ensure a fair comparison among the three options,  in \texttt{OPT A-int} we compute the error quantities in \eqref{eq:err_quant}, obtained with \texttt{OPT A} on the physical domain $\Omega$ (instead of $\Omega_h$), recalling that for \texttt{OPT B} and \texttt{OPT C} it holds $\Omega_h= \Omega$.

\begin{table}[b!]
\centering
\begin{small}
\begin{tabular}{l|r|cr|cr|cr|cr}
\toprule
&
& \multicolumn{8}{c}{\texttt{OPTIONS}} \\
&
& \multicolumn{2}{c}{$\texttt{OPT A-int}$}
& \multicolumn{2}{c}{$\texttt{OPT A}$}
& \multicolumn{2}{c}{$\texttt{OPT B}$}
& \multicolumn{2}{c}{$\texttt{OPT C}$} \\
 \cmidrule(lr){3-10}
   \texttt{ERROR}
&  \texttt{1/h}
& \texttt{err} & \texttt{EOC}
& \texttt{err} & \texttt{EOC} 
& \texttt{err} & \texttt{EOC} 
& \texttt{err} & \texttt{EOC} 
\\
\midrule
\multirow{7}{*}{\texttt{$H^1$-err}}     
&  \texttt{8}
& \texttt{7.519e-2} & 
& \texttt{7.690e-2} &
& \texttt{7.537e-2} &
& \texttt{7.610e-2} &
\\
&  \texttt{16}
& \texttt{5.338e-2} & \texttt{0.49}
& \texttt{5.856e-2} & \texttt{0.39}
& \texttt{4.578e-2} & \texttt{0.72}
& \texttt{5.296e-2} & \texttt{0.52}
\\
&  \texttt{32}
& \texttt{4.171e-2} & \texttt{0.35}
& \texttt{5.031e-2} & \texttt{0.21}
& \texttt{2.944e-2} & \texttt{0.63}
& \texttt{3.947e-2} & \texttt{0.42}
\\
&  \texttt{64}
& \texttt{2.413e-2} & \texttt{0.78}
& \texttt{2.778e-2} & \texttt{0.85}
& \texttt{1.821e-2} & \texttt{0.69}
& \texttt{2.266e-2} & \texttt{0.80}
\\
&  \texttt{128}
& \texttt{1.224e-2} & \texttt{0.97}
& \texttt{1.250e-2} & \texttt{1.15}
& \texttt{1.226e-2} & \texttt{0.57}
& \texttt{1.238e-2} & \texttt{0.87}
\\
&  \texttt{256}
& \texttt{8.552e-3} & \texttt{0.51}
& \texttt{9.358e-3} & \texttt{0.42}
& \texttt{7.375e-3} & \texttt{0.73}
& \texttt{8.485e-3} & \texttt{0.54}
\\
\cmidrule(r){2-10}
& \texttt{AEOC}
& & \texttt{0.62}
& & \texttt{0.60}
& & \texttt{0.67}
& & \texttt{0.63}
\\
\midrule
\multirow{7}{*}{\texttt{$L^2$-err}}     
&  \texttt{8}
& \texttt{1.997e-3} & 
& \texttt{2.054e-3} &
& \texttt{4.617e-3} &
& \texttt{2.626e-3} &
\\
&  \texttt{16}
& \texttt{1.448e-3} & \texttt{0.46}
& \texttt{1.499e-3} & \texttt{0.45}
& \texttt{1.577e-3} & \texttt{1.54}
& \texttt{7.046e-4} & \texttt{1.90}
\\
&  \texttt{32}
& \texttt{1.891e-3} & \texttt{-0.38}
& \texttt{1.924e-3} & \texttt{-0.36}
& \texttt{5.656e-4} & \texttt{1.47}
& \texttt{8.257e-4} & \texttt{-0.23}
\\
&  \texttt{64}
& \texttt{4.712e-4} & \texttt{2.00}
& \texttt{4.761e-4} & \texttt{2.01}
& \texttt{2.416e-4} & \texttt{1.22}
& \texttt{1.165e-4} & \texttt{2.82}
\\
&  \texttt{128}
& \texttt{2.919e-5} & \texttt{4.01}
& \texttt{2.957e-5} & \texttt{4.01}
& \texttt{1.111e-4} & \texttt{1.11}
& \texttt{5.622e-5} & \texttt{1.05}
\\
&  \texttt{256}
& \texttt{3.172e-5} & \texttt{-0.12}
& \texttt{3.195e-5} & \texttt{-0.11}
& \texttt{4.027e-5} & \texttt{1.46}
& \texttt{9.137e-6} & \texttt{2.62}
\\
\cmidrule(r){2-10}
& \texttt{AEOC}
& & \texttt{1.19}
& & \texttt{1.20}
& & \texttt{1.36}
& & \texttt{1.63}
\\
\bottomrule
\end{tabular}
\end{small}
\caption{Test 2. Computed errors $\texttt{err}(u_h, H^1)$ (top) and $\texttt{err}(u_h, L^2)$ (bottom)
as defined in \eqref{eq:err_quant}, together with the corresponding \texttt{EOC} and \texttt{AEOC} (see \eqref{eq:err_eoc} and \eqref{eq:err_aeoc}), for a sequence of Cartesian tessellations and for the three options described above.}
\label{tab:test2-conv}
\end{table}
We observe that the three options yield very similar results overall.
Furthermore, the exact solution $u$ has a singularity of order $2/3$ at the reentrant corner $(t,t)$; hence, the theoretical convergence rates predicted by Theorems~\ref{th:energy-error} and \ref{th:L2-error} are \texttt{2/3} and \texttt{4/3} for the errors in the $H^1$ and $L^2$ norms, respectively.
While the average convergence rates are in good agreement with these predictions, the individual rates do not consistently reflect this behavior. In particular, a kink can be observed for the third mesh, corresponding to the mesh size $\texttt{h=1/32}$.
To further investigate this behavior, in Fig.~\ref{meshes_l} we show a zoom of the three meshes corresponding to $\texttt{h=1/16}$, $\texttt{h=1/32}$, and $\texttt{h=1/64}$ for the configuration \texttt{OPT A}.
We note that the proposed trimming procedure does not generate a nested sequence of meshes as the mesh size decreases; therefore, some irregularities in the convergence behaviour are to be expected.
\begin{figure}[!h]
	\centering
\includegraphics[width=0.30\textwidth]{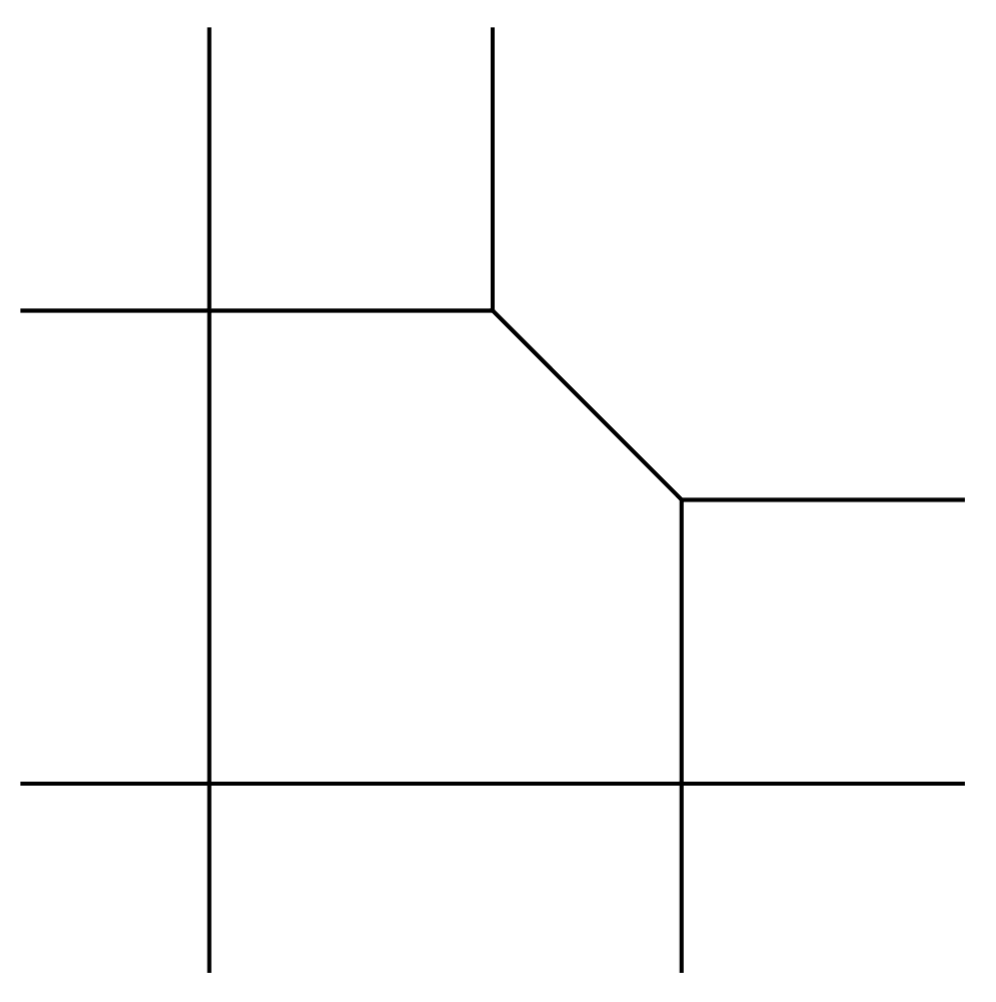}
\quad 
\includegraphics[width=0.30\textwidth]{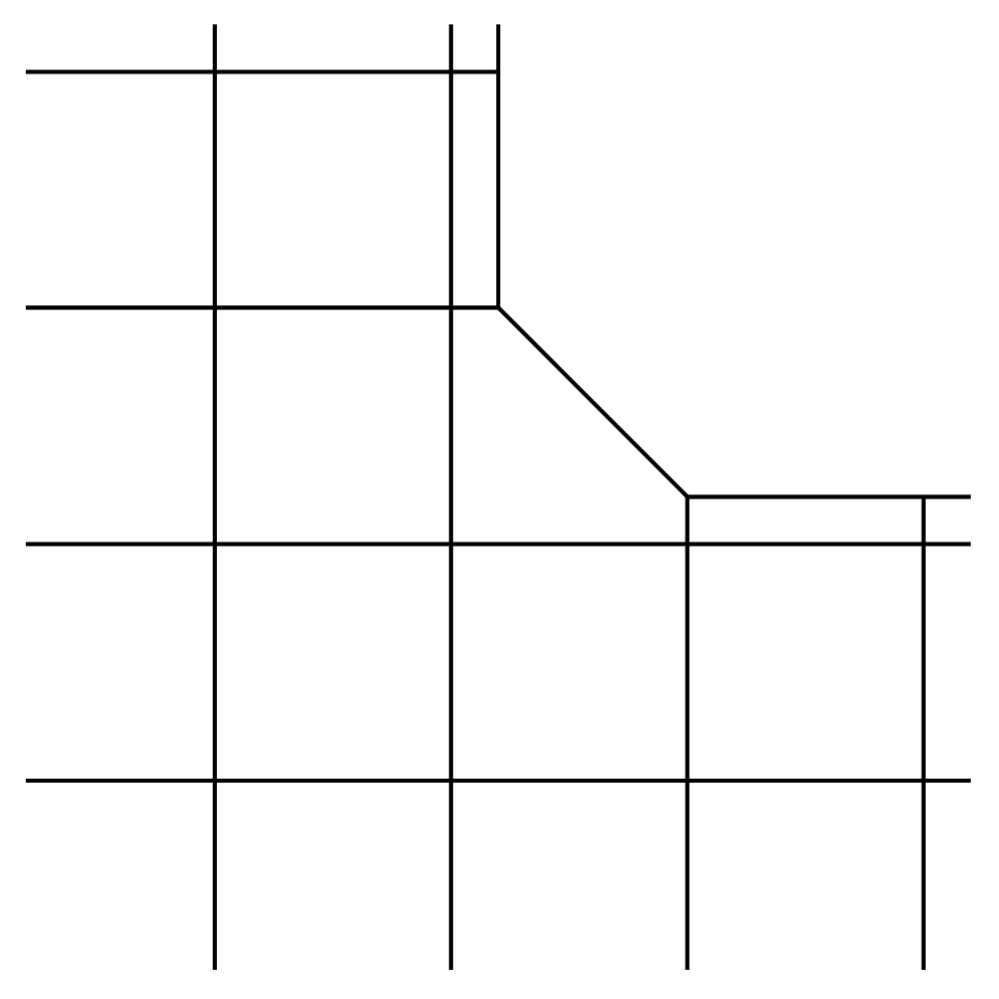}
\quad 
\includegraphics[width=0.30\textwidth]{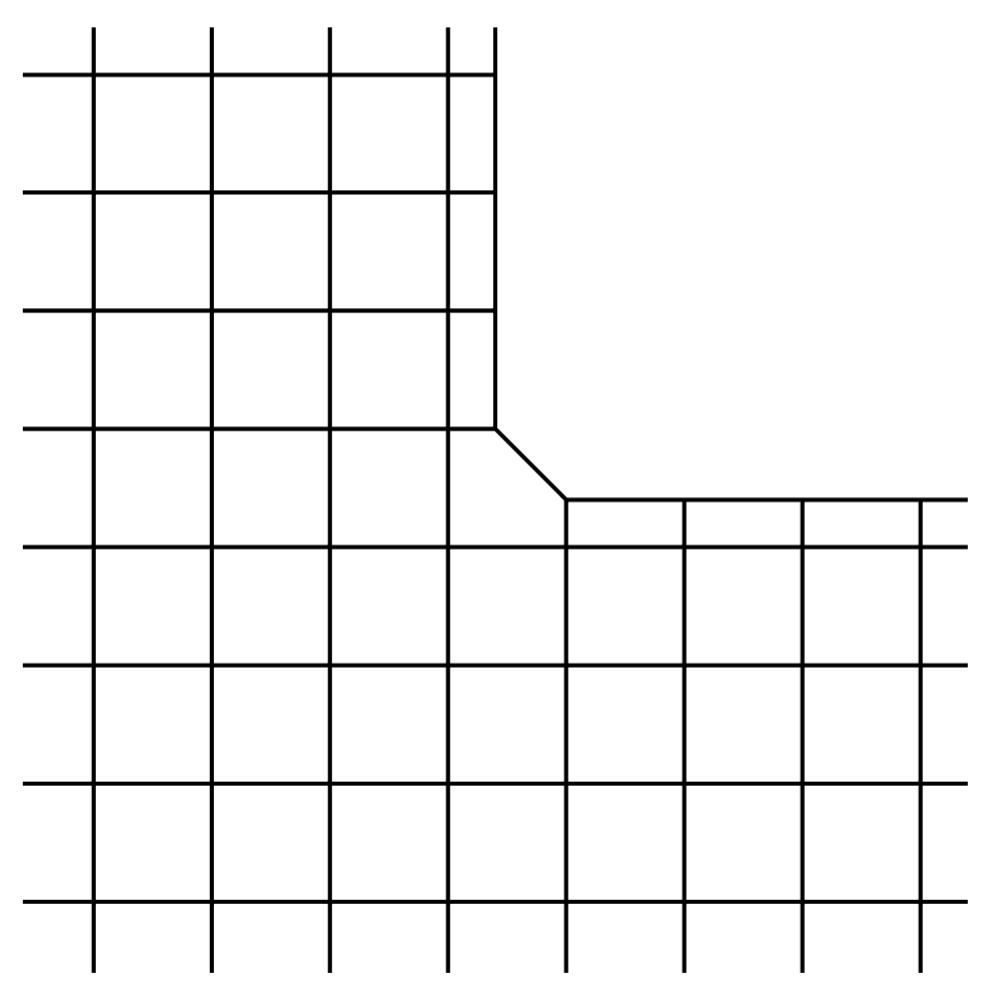}
	\caption{Test 2. Zoom to $(-1/40, 9/40)^2$ of the mesh $\mathcal{T}_h$ with \texttt{h=1/16} (left), \texttt{h=1/32} (middle), \texttt{h=1/64} (right), for \texttt{OPT A}.}
\label{meshes_l}
\end{figure}
We finally note that very similar results are obtained when using, instead of the dofi–dofi stabilization, the stabilization forms described in Section~\ref{sec:stab:X} (results not reported here).

We further assess the performance of the proposed VEM scheme in the presence of elements with poor aspect ratio.
We thus consider the Poisson problem \eqref{eq:poisson_test} posed on the L-shaped domain $\Omega(t)$ in \eqref{eq:omega_t} with smaller and smaller \texttt{t} and set the mesh size of the background tessellation to $\texttt{h=1/16}$.
In Fig. \ref{meshes_l_c} we exhibit the details of the meshes $\mathcal{T}_h$ obtained with $\texttt{t=1/64}$ and $\texttt{t=1/512}$. 

\begin{figure}[!h]
	\centering
\includegraphics[width=0.3\textwidth]{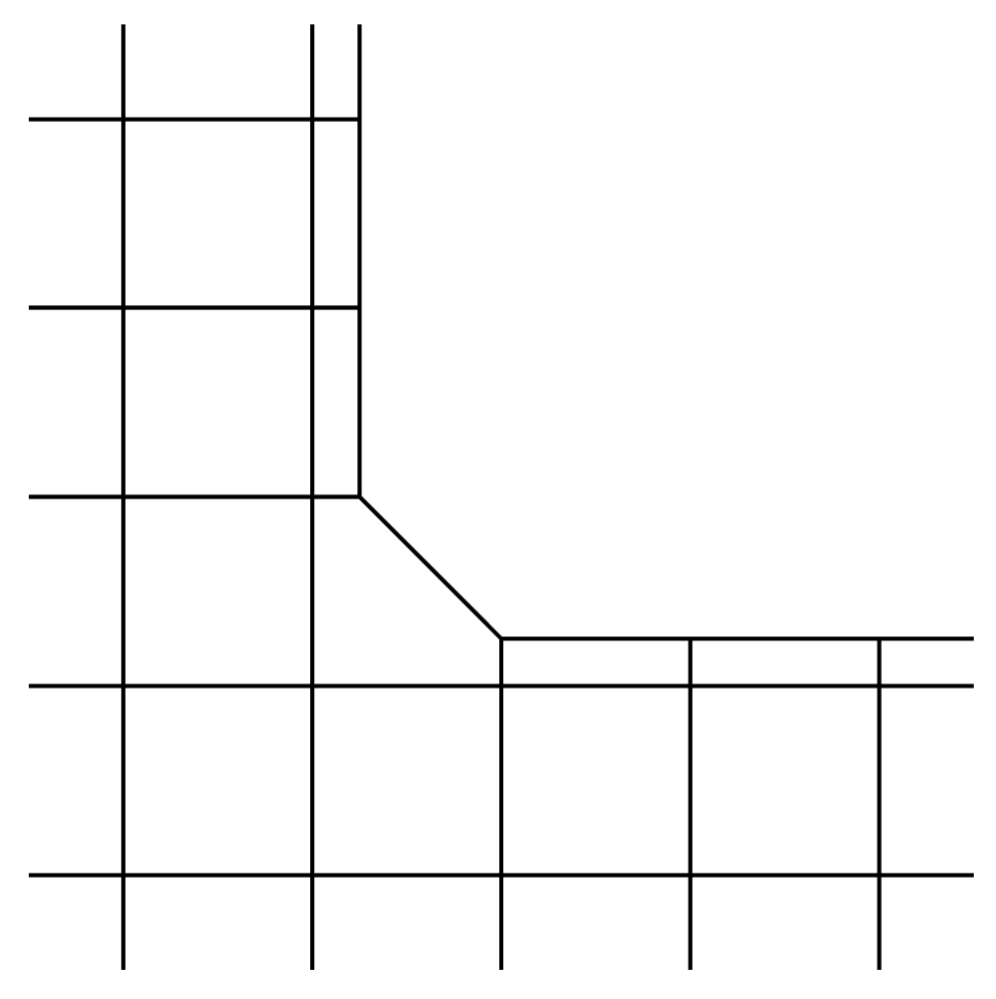}
\qquad 
\includegraphics[width=0.3\textwidth]{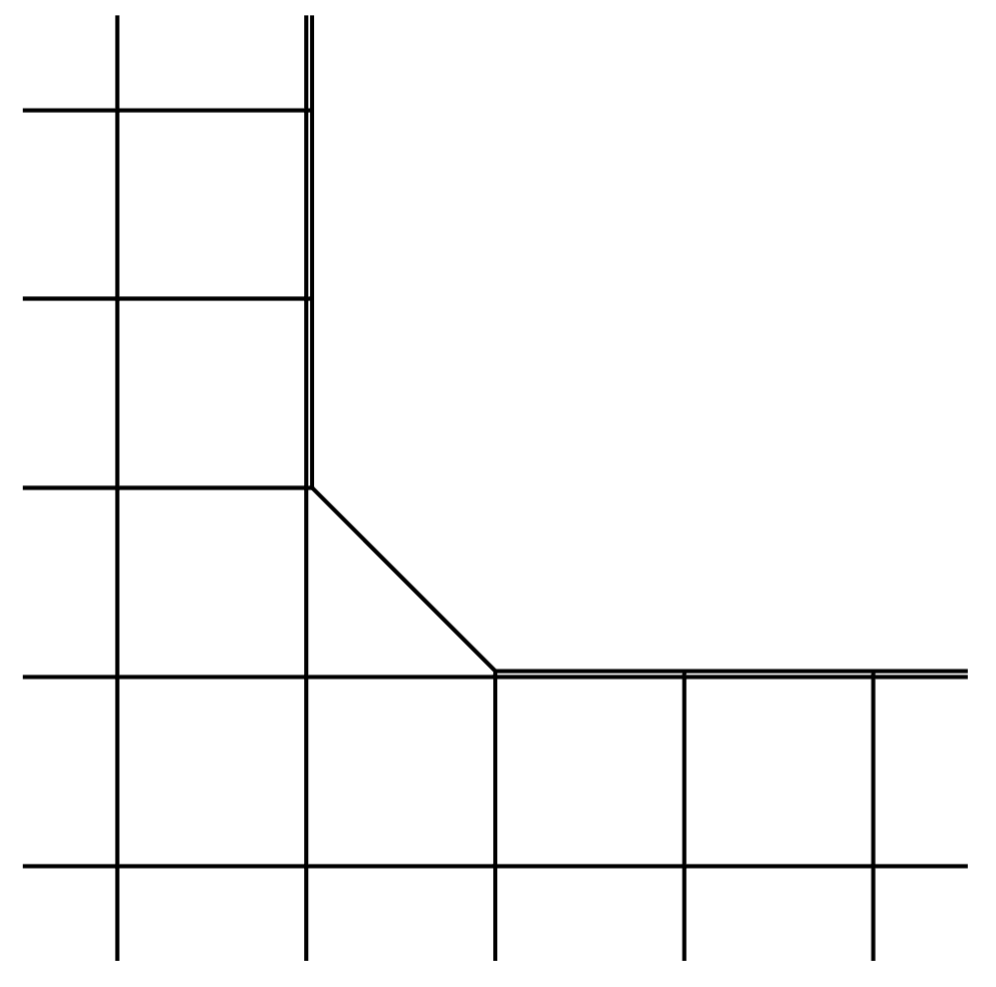}
	\caption{
Test 2. Zoom to $(-3/32, 7/32)^2$ of the mesh $\mathcal{T}_h$ with \texttt{h=1/16} for \texttt{t=1/64} (left) and \texttt{t=1/512} (right), for \texttt{OPT A}.}
\label{meshes_l_c}
\end{figure}

Note that, according to the trimming procedure given in Definition \ref{def:E-K} (see also Example \ref{ex:element-geometries}), for small values of $\texttt{t}$ the procedure generates rectangular elements with long side $\texttt{1/16}$ and short side $\texttt{t}$; thus, the aspect ratio is $\texttt{1/16t}$.

Table~\ref{tab:test2-t} reports the computed error quantities in \eqref{eq:err_quant} and the condition numbers of the resulting linear systems for a sequence of domains $\Omega(t)$ for different values of $\texttt{t}$ and a fixed background mesh with size $\texttt{h = 1/16}$,  (the corresponding values of $\texttt{t}$ are listed in the table).
In order to assess the stability of the method with respect to the choice of the stabilization, we consider both the standard dofi-dofi stabilization and the stabilization ad hoc proposed in Section \ref{sec:stab:X} for anisotropic elements labeled as $s_E^{\texttt dofi}$ and $s_E^{\texttt new}$ respectively. We also consider the three possible configurations detailed above.

\begin{table}[!ht]
\centering
\begin{small}
\begin{tabular}{l|r|cc|cc|cc}
\toprule
&
& \multicolumn{6}{c}{\texttt{OPTIONS}} \\
&
& \multicolumn{2}{c}{$\texttt{OPT A-int}$}
& \multicolumn{2}{c}{$\texttt{OPT B}$}
& \multicolumn{2}{c}{$\texttt{OPT C}$} \\
 \cmidrule(lr){3-8}
   \texttt{QUANTITY}
&  \texttt{1/t}
& $s_E^{\texttt{dofi}}$ & $s_E^{\texttt{new}}$
& $s_E^{\texttt{dofi}}$ & $s_E^{\texttt{new}}$
& $s_E^{\texttt{dofi}}$ & $s_E^{\texttt{new}}$
\\
\midrule
\multirow{6}{*}{\texttt{$H^1$-err}}     
&  \texttt{64}
& \texttt{6.469e-2} & \texttt{6.469e-2}
& \texttt{4.587e-2} & \texttt{4.587e-2}
& \texttt{5.467e-2} & \texttt{5.467e-2}
\\
&  \texttt{128}
& \texttt{6.725e-2} & \texttt{6.716e-2}
& \texttt{4.704e-2} & \texttt{4.704e-2}
& \texttt{5.747e-2} & \texttt{5.710e-2}
\\
&  \texttt{256}
& \texttt{6.822e-2} & \texttt{6.799e-2}
& \texttt{4.813e-2} & \texttt{4.805e-2}
& \texttt{5.749e-2} & \texttt{5.850e-2}
\\
&  \texttt{512}
& \texttt{6.878e-2} & \texttt{6.848e-2}
& \texttt{4.883e-2} & \texttt{4.869e-2}
& \texttt{5.578e-2} & \texttt{5.666e-2}
\\
&  \texttt{1024}
& \texttt{6.888e-2} & \texttt{6.856e-2}
& \texttt{4.897e-2} & \texttt{4.883e-2}
& \texttt{5.348e-2} & \texttt{5.388e-2}
\\
&  \texttt{2048}
& \texttt{6.881e-2} & \texttt{6.849e-2}
& \texttt{4.896e-2} & \texttt{4.885e-2}
& \texttt{5.029e-2} & \texttt{5.038e-2}
\\
\midrule
\multirow{6}{*}{\texttt{$L^2$-err}}     
&  \texttt{64}
& \texttt{4.321e-3} & \texttt{4.321e-3}
& \texttt{1.441e-3} & \texttt{1.441e-3}
& \texttt{1.660e-3} & \texttt{1.660e-3}
\\
&  \texttt{128}
& \texttt{5.405e-3} & \texttt{5.386e-3}
& \texttt{1.403e-3} & \texttt{1.432e-3}
& \texttt{2.594e-3} & \texttt{2.561e-3}
\\
&  \texttt{256}
& \texttt{5.936e-3} & \texttt{5.885e-3}
& \texttt{1.431e-3} & \texttt{1.519e-3}
& \texttt{2.595e-3} & \texttt{2.753e-3}
\\
&  \texttt{512}
& \texttt{6.181e-3} & \texttt{6.104e-3}
& \texttt{1.494e-3} & \texttt{1.624e-3}
& \texttt{2.275e-3} & \texttt{2.444e-3}
\\
&  \texttt{1024}
& \texttt{6.281e-3} & \texttt{6.189e-3}
& \texttt{1.565e-3} & \texttt{1.713e-3}
& \texttt{2.021e-3} & \texttt{2.173e-3}
\\
&  \texttt{2048}
& \texttt{6.313e-3} & \texttt{6.220e-3}
& \texttt{1.631e-3} & \texttt{1.774e-3}
& \texttt{1.881e-3} & \texttt{2.016e-3}
\\
\midrule  
\multirow{7}{*}{\texttt{cond}}     
&  \texttt{64}
& \texttt{2.323e+2} & \texttt{2.323e+2}
& \texttt{2.484e+2} & \texttt{2.484e+2}
& \texttt{2.647e+2} & \texttt{2.647e+2}
\\
&  \texttt{128}
& \texttt{3.877e+2} & \texttt{3.877e+2}
& \texttt{4.007e+2} & \texttt{4.007e+2}
& \texttt{6.440e+2} & \texttt{7.632e+2}
\\
&  \texttt{256}
& \texttt{6.982e+2} & \texttt{6.982e+2}
& \texttt{7.051e+2} & \texttt{7.050e+2}
& \texttt{2.158e+3} & \texttt{4.138e+3}
\\
&  \texttt{512}
& \texttt{1.319e+3} & \texttt{1.319e+3}
& \texttt{1.314e+3} & \texttt{1.313e+3}
& \texttt{7.956e+3} & \texttt{3.001e+4}
\\
&  \texttt{1024}
& \texttt{2.560e+3} & \texttt{2.560e+3}
& \texttt{2.532e+3} & \texttt{2.530e+3}
& \texttt{3.051e+4} & \texttt{2.347e+5}
\\
&  \texttt{2048}
& \texttt{5.043e+3} & \texttt{5.041e+3}
& \texttt{4.967e+3} & \texttt{4.964e+3}
& \texttt{1.194e+5} & \texttt{1.869e+6}
\\
\cmidrule(r){2-8}
& \texttt{RATE}
& \texttt{0.89} & \texttt{0.88} 
& \texttt{0.86} & \texttt{0.86}
& \texttt{1.76} & \texttt{2.56}
\\
\bottomrule
\end{tabular}
\end{small}
\caption{Test 2. Computed error quantities in \eqref{eq:err_quant} and condition numbers of the resulting linear systems for a sequence of domains $\Omega(t)$ with varying $\texttt{t}$, using a fixed background mesh with $\texttt{h=1/16}$ and the three options described above. The values of $\texttt{t}$ are listed in the table.
Standard dofi-dofi stabilization and  stabilization proposed in Section \ref{sec:stab:X} are labeled as $s_E^{\texttt dofi}$ and $s_E^{\texttt new}$ respectively}
\label{tab:test2-t}
\end{table}

The results show that the method is robust with respect to both the aspect ratio of the elements and the choice of stabilization.
Furthermore they indicate that the condition number of the resulting linear systems grows approximately linearly with the aspect ratio for configurations \texttt{OPT A} and \texttt{OPT B}. In contrast, for \texttt{OPT C} the growth appears to be quadratic when using the dofi-dofi stabilization, and cubic when employing the stabilization proposed in Section~\ref{sec:stab:X}. We stress, however, that the latter stabilization was specifically designed for quadrilateral elements with poor aspect ratio, and its performance in more general settings (e.g., for concave polygons) is not guaranteed.



\section{Conclusions and perspectives}\label{sec:conclusions}

 The contributions of this paper are twofold. First, we classify and study the geometric configurations that arise when a two-dimensional stationary piecewise smooth domain cuts through a quasi-uniform fixed polygonal background mesh made of rectangles. This study is rather technical, but is instrumental for robust stability of the ensuing VEM and for its optimal accuracy. Second, we provide novel and crucial tools for investigating VEMs for time-dependent PDEs posed on deformable domains, which is our next objective \cite{NostroMovingDomain}. In such a case, VEMs must deal with dynamic mesh changes along the domain boundary which may give rise to extreme geometric configurations even under the assumption of background mesh resolution.

This motivates our systematic classification of all possible polygonal shapes of boundary elements, along with their stability and approximation properties. Results of this type exist in the current VEM literature, but they are insufficient for our overall purposes including the analysis of evolving domains. Therefore, we embark on the design of stabilization mechanisms that are robust with respect to large aspect ratios and small cuts. We also develop a weak maximum principle for enhanced virtual elements that controls the max-norm of a virtual function on star-shaped polygons in terms of its max-norm on the element boundary; we examine the dependence of the stability constant on anisotropy. Moreover, we derive interpolation error estimates for virtual functions that, within the classes of elements here investigated, are insensitive to element shape and apply both to regular solutions and corner singularities. 
These results are crucial for this paper but are also essential building blocks for \cite{NostroMovingDomain}. Moreover, we believe that they possess an intrinsic interest in that they extend the existing basic VEM theory about geometric flexibility.

Our stability and convergence analyses hinge on abstract Assumptions \ref{ass:stab-form} (stabilization form), \ref{ass:interp} (interpolation error) and \ref{ass:polapprox} (polynomial approximation). We verify these assumptions in Sections \ref{sec:maximum-principle},
\ref{sec:stab:X}, and \ref{sec:checking}, which is where we deal directly with geometry, as explained above. In contrast, Section \ref{stability-convergence-abstract} presents a geometry-free approach to robust $H^1$-stability of the VEM and optimal order-regularity error estimates in $H^1$ and $L^2$. We believe this style of presentation improves readability because we postpone the technical discussion of geometry in favor of a direct discussion of stability and convergence. Yet, the latter must account for the discrepancy between computational and actual domains (geometric error), especially near the break points which are not generally nodes of the computational mesh - the so-called trimmed background mesh.

This paper provides the basic ingredients for developing a VEM-ALE approach in \cite{NostroMovingDomain}. Our goal is to circumvent the usual bottleneck of FEM-ALE regarding mesh distortion and remeshing. Our approach would avoid extending the velocity of the domain boundary inside the domain, a delicate procedure without theoretical guarantees. It would also eliminate the need of remeshing to maintain mesh quality. However, the ensuing VEM-ALE must be robust with respect to extreme geometric situations such as anisotropy, small elements and small edges to be competitive. This justifies our current in-depth study of robust stabilization and approximation properties relative to element shape in a simpler stationary setting.

The abstract analysis of Section \ref{stability-convergence-abstract} as well as the construction of boundary elements via the convex hull procedure of Section \ref{sec:discrete-abstract} extend to 3D. However, the classification of element shapes and their stability and approximation properties require further study. It is possible to resort to glueing of a small boundary element with an adjacent larger one, an interesting idea we have not pursued in this paper.

\bigskip
\begin{center}
{\bf Aknowledgements} 
\end{center}
\smallskip
LBDV and MV has been partially funded by the European Union (ERC Synergy, NEMESIS, project number 101115663).
Views and opinions expressed are however those of the authors only and do not necessarily reflect those of the European Union or the ERC Executive Agency. 
The authors LBDV, CC, GV, MV have been partially supported by the INdAM Research group GNCS,
whereas RHN has been partially supported by NSF grants DMS--1908267 and DMS-2512392.


\bibliographystyle{plain}
\bibliography{biblio}

@article{NostroMovingDomain,
author = {Beir\~ao da Veiga,L. and Canuto, C. and Nochetto, R.H. and Vacca, G. and Verani, M.},
title = {An {ALE-VEM} approach to parabolic problems in moving domains},
}

@article{payne1960optimal,
  title={An optimal {P}oincar{\'e} inequality for convex domains},
  author={Payne, L. E. and Weinberger, H. F.},
  journal={Archive for Rational Mechanics and Analysis},
  volume={5},
  pages={286--292},
  year={1960}
}

@article {volley,
    AUTHOR = {Beir{\~a}o da Veiga, L. and Brezzi, F. and Cangiani, A. and
              Manzini, G. and Marini, L. D. and Russo, A.},
     TITLE = {Basic principles of virtual element methods},
   JOURNAL = {Math. Models Methods Appl. Sci.},
  FJOURNAL = {Mathematical Models and Methods in Applied Sciences},
    VOLUME = {23},
      YEAR = {2013},
    NUMBER = {1},
     PAGES = {199--214},
}

@article {hitchhiker,
    AUTHOR = {Beir\~ao da Veiga, L. and Brezzi, F. and Marini, L. D. and
              Russo, A.},
     TITLE = {The hitchhiker's guide to the virtual element method},
   JOURNAL = {Math. Models Methods Appl. Sci.},
  FJOURNAL = {Mathematical Models and Methods in Applied Sciences},
    VOLUME = {24},
      YEAR = {2014},
    NUMBER = {8},
     PAGES = {1541--1573},
       DOI = {10.1142/S021820251440003X},
       URL = {https://doi.org/10.1142/S021820251440003X},
}

@article {vem-acta,
    AUTHOR = {Beir\~ao da Veiga, Louren\c co and Brezzi, Franco and Marini,
              L. Donatella and Russo, Alessandro},
     TITLE = {The virtual element method},
   JOURNAL = {Acta Numer.},
  FJOURNAL = {Acta Numerica},
    VOLUME = {32},
      YEAR = {2023},
     PAGES = {123--202},
       DOI = {10.1017/S0962492922000095},
       URL = {https://doi.org/10.1017/S0962492922000095},
}

@book {vem-special-issue,
     TITLE = {The virtual element method and its applications},
    SERIES = {SEMA SIMAI Springer Series},
    VOLUME = {31},
    EDITOR = {Antonietti, Paola F. and Beir\~ao da Veiga, Louren\c co and
              Manzini, Gianmarco},
 PUBLISHER = {Springer, Cham},
      YEAR = {[2022] \copyright 2022},
     PAGES = {xxiv+605},
       DOI = {10.1007/978-3-030-95319-5},
       URL = {https://doi.org/10.1007/978-3-030-95319-5},
}

@article {vem-esher,
    AUTHOR = {Paulino, G. H. and Gain, A. L.},
     TITLE = {Bridging art and engineering using {E}scher-based virtual
              elements},
   JOURNAL = {Struct. Multidiscip. Optim.},
  FJOURNAL = {Structural and Multidisciplinary Optimization},
    VOLUME = {51},
      YEAR = {2015},
    NUMBER = {4},
     PAGES = {867--883},
       DOI = {10.1007/s00158-014-1179-7},
       URL = {https://doi.org/10.1007/s00158-014-1179-7},
}

@article {vem-noncov,
    AUTHOR = {Park, K. and Chi, H.and Paulino, G. H.},
     TITLE = {On nonconvex meshes for elastodynamics using virtual element
              methods with explicit time integration},
   JOURNAL = {Comput. Methods Appl. Mech. Engrg.},
  FJOURNAL = {Computer Methods in Applied Mechanics and Engineering},
    VOLUME = {356},
      YEAR = {2019},
     PAGES = {669--684},
       DOI = {10.1016/j.cma.2019.06.031},
       URL = {https://doi.org/10.1016/j.cma.2019.06.031},
}

@article {vem-deformations,
    AUTHOR = {Chi, H. and Beir\~ao da Veiga, L.  and Paulino, G. H.},
     TITLE = {Some basic formulations of the virtual element method ({VEM})
              for finite deformations},
   JOURNAL = {Comput. Methods Appl. Mech. Engrg.},
  FJOURNAL = {Computer Methods in Applied Mechanics and Engineering},
    VOLUME = {318},
      YEAR = {2017},
     PAGES = {148--192},
       DOI = {10.1016/j.cma.2016.12.020},
       URL = {https://doi.org/10.1016/j.cma.2016.12.020},
}

@article {vem-ep-deformations,
    AUTHOR = {Hudobivnik, B. and Aldakheel, F. and Wriggers, P.},
     TITLE = {A low order 3{D} virtual element formulation for finite
              elasto-plastic deformations},
   JOURNAL = {Comput. Mech.},
  FJOURNAL = {Computational Mechanics},
    VOLUME = {63},
      YEAR = {2019},
    NUMBER = {2},
     PAGES = {253--269},
       DOI = {10.1007/s00466-018-1593-6},
       URL = {https://doi.org/10.1007/s00466-018-1593-6},
}

@incollection {vem-dfn,
    AUTHOR = {Benedetto, Mat\'ias Fernando and Berrone, Stefano and Borio,
              Andrea},
     TITLE = {The virtual element method for underground flow situations in
              fractured data},
 BOOKTITLE = {Advances in discretization methods},
    SERIES = {SEMA SIMAI Springer Ser.},
    VOLUME = {12},
     PAGES = {167--186},
 PUBLISHER = {Springer, [Cham]},
      YEAR = {2016},
  MRNUMBER = {3585413},
}

@book{wriggers2024virtual,
  title={Virtual element methods in engineering sciences},
  author={Wriggers, P. and Aldakheel, F. and Hudobivnik, B.},
  year={2024},
  publisher={Springer Cham},
     PAGES = {XV, 452},
DOI = {https://doi.org/10.1007/978-3-031-39255-9}
}

@article {vem-imati,
    AUTHOR = {Attene, M. and Biasotti, S. and Bertoluzza, S. and
              Cabiddu, D. and Livesu, M. and Patan\`e, G. and
              Pennacchio, M. and Prada, D. and Spagnuolo, M.},
     TITLE = {Benchmarking the geometrical robustness of a virtual element
              {P}oisson solver},
   JOURNAL = {Math. Comput. Simulation},
  FJOURNAL = {Mathematics and Computers in Simulation},
    VOLUME = {190},
      YEAR = {2021},
     PAGES = {1392--1414},
       DOI = {10.1016/j.matcom.2021.07.018},
}

@article {vem-imati-genova,
    AUTHOR = {Sorgente, T. and Biasotti, S. and Manzini, G. and Spagnuolo,
              M.},
     TITLE = {Polyhedral mesh quality indicator for the virtual element
              method},
   JOURNAL = {Comput. Math. Appl.},
  FJOURNAL = {Computers \& Mathematics with Applications. An International
              Journal},
    VOLUME = {114},
      YEAR = {2022},
     PAGES = {151--160},
       DOI = {10.1016/j.camwa.2022.03.042},
       URL = {https://doi.org/10.1016/j.camwa.2022.03.042},
}

@article {vem-nn,
    AUTHOR = {Antonietti, P. F. and Manuzzi, E.},
     TITLE = {Refinement of polygonal grids using convolutional neural
              networks with applications to polygonal discontinuous
              {G}alerkin and virtual element methods},
   JOURNAL = {J. Comput. Phys.},
  FJOURNAL = {Journal of Computational Physics},
    VOLUME = {452},
      YEAR = {2022},
     PAGES = {Paper No. 110900, 20},
      ISSN = {0021-9991,1090-2716},
   MRCLASS = {65N30 (65N50)},
  MRNUMBER = {4361820},
       DOI = {10.1016/j.jcp.2021.110900},
       URL = {https://doi.org/10.1016/j.jcp.2021.110900},
}

@article {projectors,
    AUTHOR = {Ahmad, B. and Alsaedi, A. and Brezzi, F. and Marini, L. D. and
              Russo, A.},
     TITLE = {Equivalent projectors for virtual element methods},
   JOURNAL = {Comput. Math. Appl.},
  FJOURNAL = {Computers \& Mathematics with Applications. An International
              Journal},
    VOLUME = {66},
      YEAR = {2013},
    NUMBER = {3},
     PAGES = {376--391},
}

@article{BBMR16,
  title={Virtual element method for general second-order elliptic problems on polygonal meshes},
  author={Beirao da Veiga, L. and Brezzi, F. and Marini, L.D. and Russo, A.},
  journal={Mathematical Models and Methods in Applied Sciences},
  volume={26},
  number={04},
  pages={729--750},
  year={2016},
  publisher={World Scientific}
}

@article{BLRstab,
  title={Stability analysis for the virtual element method},
  author={Beir{\~a}o da Veiga, L. and Lovadina, C. and Russo, A.},
  journal={Mathematical Models and Methods in Applied Sciences},
  volume={27},
  number={13},
  pages={2557--2594},
  year={2017},
  publisher={World Scientific}
}

@article{brenner2018,
  title={Virtual element methods on meshes with small edges or faces},
  author={Brenner, S.C. and Sung, L.Y.},
  journal={Mathematical Models and Methods in Applied Sciences},
  volume={28},
  number={07},
  pages={1291--1336},
  year={2018},
  publisher={World Scientific}
}

@article{chen2018,
  title={Some error analysis on virtual element methods},
  author={Chen, L. and Huang, J.},
  journal={Calcolo},
  volume={55},
  number={1},
  pages={5},
  year={2018},
  publisher={Springer}
}

@article {vem-steklov,
    AUTHOR = {Mora, D. and Rivera, G. and Rodr\'iguez, R.},
     TITLE = {A virtual element method for the {S}teklov eigenvalue problem},
   JOURNAL = {Math. Models Methods Appl. Sci.},
  FJOURNAL = {Mathematical Models and Methods in Applied Sciences},
    VOLUME = {25},
      YEAR = {2015},
    NUMBER = {8},
     PAGES = {1421--1445},
       DOI = {10.1142/S0218202515500372},
       URL = {https://doi.org/10.1142/S0218202515500372},
}

@book{apel-book,
  title={Anisotropic finite elements: local estimates and applications},
  author={Apel, T.},
  year={1999},
  publisher={Teubner Stuttgart}
}

@article{smalledges,
  title={Sharper error estimates for virtual elements and a bubble-enriched version},
  author={Beir{\~a}o da Veiga, L. and Vacca, G.},
  journal={SIAM Journal on Numerical Analysis},
  volume={60},
  number={4},
  pages={1853--1878},
  year={2022},
  publisher={SIAM}
}

@article{Chen_anisotr:2018,
author = {Cao, S. and Chen, L.},
title = {Anisotropic Error Estimates of the Linear Virtual Element Method on Polygonal Meshes},
journal = {SIAM Journal on Numerical Analysis},
volume = {56},
number = {5},
pages = {2913-2939},
year = {2018},
doi = {10.1137/17M1154369}
}

@article {Chen_anisotr_nc,
    AUTHOR = {Cao, S. and Chen, L.},
     TITLE = {Anisotropic error estimates of the linear nonconforming
              virtual element methods},
   JOURNAL = {SIAM J. Numer. Anal.},
  FJOURNAL = {SIAM Journal on Numerical Analysis},
    VOLUME = {57},
      YEAR = {2019},
    NUMBER = {3},
     PAGES = {1058--1081},
       DOI = {10.1137/18M1196455},
       URL = {https://doi.org/10.1137/18M1196455},
}

@article {Nochetto_Liao,
    AUTHOR = {Liao, Xiaohai and Nochetto, Ricardo H.},
     TITLE = {Local a posteriori error estimates and adaptive control of
              pollution effects},
   JOURNAL = {Numer. Methods Partial Differential Equations},
  FJOURNAL = {Numerical Methods for Partial Differential Equations. An
              International Journal},
    VOLUME = {19},
      YEAR = {2003},
    NUMBER = {4},
     PAGES = {421--442},
      ISSN = {0749-159X,1098-2426},
   MRCLASS = {65N15 (35J25)},
  MRNUMBER = {1980188},
MRREVIEWER = {Sergey\ I.\ Repin},
       DOI = {10.1002/num.10053},
       URL = {https://doi.org/10.1002/num.10053},
}

@book {Grisvard,
    AUTHOR = {Grisvard, P.},
     TITLE = {Elliptic problems in nonsmooth domains},
    SERIES = {Classics in Applied Mathematics},
    VOLUME = {69},
 PUBLISHER = {Society for Industrial and Applied Mathematics (SIAM),
              Philadelphia, PA},
      YEAR = {2011},
     PAGES = {xx+410},
      ISBN = {978-1-611972-02-3},
}

@article{BurmanHansboLarsonMassing2015,
  author  = {Burman, E. and Hansbo, P. and Larson, M. G. and Massing, A.},
  title   = {Cut finite element methods for partial differential equations on embedded manifolds of arbitrary codimension},
  journal = {ESAIM: Mathematical Modelling and Numerical Analysis},
  volume  = {52},
  number  = {6},
  pages   = {2247--2282},
  year    = {2018},
  note    = {Often cited as CutFEM overview papers around 2015}
}

@article {Cut_Fem_Acta:2025,
    AUTHOR = {Burman, E. and Hansbo, P. and Larson, M.G. and Zahedi, S.},
     TITLE = {Cut finite element methods},
   JOURNAL = {Acta Numer.},
  FJOURNAL = {Acta Numerica},
    VOLUME = {34},
      YEAR = {2025},
     PAGES = {1--121},
}

@article{BeckerHansboStenberg2003,
  author  = {Becker, R. and Hansbo, P. and Stenberg, R.},
  title   = {A finite element method for domain decomposition with non-matching grids},
  journal = {ESAIM: Mathematical Modelling and Numerical Analysis},
  volume  = {37},
  number  = {2},
  pages   = {209--225},
  year    = {2003},
  doi     = {10.1051/m2an:2003023}
}

@article{BurmanHansbo2010,
  author  = {Burman, E. and Hansbo, P.},
  title   = {Fictitious domain finite element methods using cut elements: I. A stabilized Lagrange multiplier method},
  journal = {Computer Methods in Applied Mechanics and Engineering},
  volume  = {199},
  number  = {41--44},
  pages   = {2680--2686},
  year    = {2010},
  doi     = {10.1016/j.cma.2010.04.011}
}

@article {BoffiCavalliniGastaldi2015,
    AUTHOR = {Boffi, D. and Cavallini, N. and Gastaldi, L.},
     TITLE = {The finite element immersed boundary method with distributed
              {L}agrange multiplier},
   JOURNAL = {SIAM J. Numer. Anal.},
  FJOURNAL = {SIAM Journal on Numerical Analysis},
    VOLUME = {53},
      YEAR = {2015},
    NUMBER = {6},
     PAGES = {2584--2604},
     }

@article {BoffiGastaldiHeltai2007,
    AUTHOR = {Boffi, D. and Gastaldi, L. and Heltai, L.},
     TITLE = {Numerical stability of the finite element immersed boundary
              method},
   JOURNAL = {Math. Models Methods Appl. Sci.},
  FJOURNAL = {Mathematical Models and Methods in Applied Sciences},
    VOLUME = {17},
      YEAR = {2007},
    NUMBER = {10},
     PAGES = {1479--1505},
}

@article{McCrackenPeskin1980,
  author  = {McCracken, M. F. and Peskin, C. S.},
  title   = {A vortex method for blood flow through heart valves},
  journal = {Journal of Computational Physics},
  volume  = {35},
  number  = {2},
  pages   = {183--205},
  year    = {1980},
  doi     = {10.1016/0021-9991(80)90079-0}
}

@article{Peskin2002,
  author  = {Peskin, C. S.},
  title   = {The immersed boundary method},
  journal = {Acta Numerica},
  volume  = {11},
  pages   = {479--517},
  year    = {2002},
  doi     = {10.1017/S0962492902000077}
}

@article{GlowinskiPanPeriaux1994,
  author  = {Glowinski, R. and Pan, T.-W. and P{\'e}riaux, J.},
  title   = {A fictitious domain method for Dirichlet problem and applications},
  journal = {Computer Methods in Applied Mechanics and Engineering},
  volume  = {111},
  number  = {3--4},
  pages   = {283--303},
  year    = {1994},
  doi     = {10.1016/0045-7825(94)90135-X}
}

@article{GiraultGlowinski1995,
  author  = {Girault, V. and Glowinski, R.},
  title   = {Error analysis of a fictitious domain method applied to a Dirichlet problem},
  journal = {Japan Journal of Industrial and Applied Mathematics},
  volume  = {12},
  number  = {3},
  pages   = {487--514},
  year    = {1995},
  doi     = {10.1007/BF03167383}
}

@article{BarrettElliott1986,
  author  = {Barrett, J. W. and Elliott, C. M.},
  title   = {Fitted and Unfitted Finite-Element Methods for Elliptic Equations with Smooth Interfaces},
  journal = {IMA Journal of Numerical Analysis},
  volume  = {7},
  number  = {3},
  pages   = {283--300},
  year    = {1987},
  doi     = {10.1093/imanum/7.3.283}
}

@article{Maury2009,
  author  = {Maury, B.},
  title   = {Numerical analysis of a finite element/volume penalty method},
  journal = {SIAM Journal on Numerical Analysis},
  volume  = {47},
  number  = {2},
  pages   = {1126--1148},
  year    = {2009},
  doi     = {10.1137/070681180}
}

@article{MoesDolbowBelytschko1999,
  author  = {Mo{\"e}s, N. and Dolbow, J. and Belytschko, T.},
  title   = {A finite element method for crack growth without remeshing},
  journal = {International Journal for Numerical Methods in Engineering},
  volume  = {46},
  number  = {1},
  pages   = {131--150},
  year    = {1999},
  doi     = {10.1002/(SICI)1097-0207(19990910)46:1<131::AID-NME726>3.0.CO;2-J}
}

@article{FriesBelytschko2010,
  author  = {Fries, T.-P. and Belytschko, T.},
  title   = {The extended/generalized finite element method: An overview of the method and its applications},
  journal = {International Journal for Numerical Methods in Engineering},
  volume  = {84},
  number  = {3},
  pages   = {253--304},
  year    = {2010},
  doi     = {10.1002/nme.2914}
}

@article{ZhangGerstenbergerWangLiu2004,
  author  = {Zhang, L. and Gerstenberger, A. and Wang, X. and Liu, W. K.},
  title   = {Immersed finite element method},
  journal = {Computer Methods in Applied Mechanics and Engineering},
  volume  = {193},
  number  = {21--22},
  pages   = {2051--2067},
  year    = {2004},
  doi     = {10.1016/j.cma.2003.12.044}
}

@article{CockburnSolano2012,
  author  = {Cockburn, B. and Solano, M.},
  title   = {Solving Dirichlet boundary-value problems on curved domains by extensions from subdomains},
  journal = {SIAM Journal on Scientific Computing},
  volume  = {34},
  number  = {1},
  pages   = {A497--A519},
  year    = {2012},
  doi     = {10.1137/100787383}
}

@article {AtallahCanutoScovazzi:2021,
    AUTHOR = {Atallah, N. M. and Canuto, C. and Scovazzi, G.},
     TITLE = {Analysis of the shifted boundary method for the {P}oisson
              problem in domains with corners},
   JOURNAL = {Math. Comp.},
  FJOURNAL = {Mathematics of Computation},
    VOLUME = {90},
      YEAR = {2021},
    NUMBER = {331},
     PAGES = {2041--2069},
}

@article{MainScovazzi2018,
  author  = {Main, A. and Scovazzi, G.},
  title   = {The shifted boundary method for embedded domain computations. Part I: Poisson and Stokes problems},
  journal = {Journal of Computational Physics},
  volume  = {372},
  pages   = {972--995},
  year    = {2018},
  doi     = {10.1016/j.jcp.2018.07.029}
}

@article{DuprezLlerasLozinski2023,
  author  = {Duprez, M. and Lleras, V. and Lozinski, A.},
  title   = {$\phi$-FEM: a finite element method on domains defined by level sets},
  journal = {SIAM Journal on Numerical Analysis},
  volume  = {61},
  number  = {3},
  pages   = {1003--1030},
  year    = {2023},
  doi     = {10.1137/21M1453119}
}

@article{ParvizianDusterRank2007,
  author  = {Parvizian, J. and D{\"u}ster, A. and Rank, E.},
  title   = {Finite cell method: h- and p-extension for embedded domain problems in solid mechanics},
  journal = {Computational Mechanics},
  volume  = {41},
  number  = {1},
  pages   = {121--133},
  year    = {2007},
  doi     = {10.1007/s00466-007-0173-y}
}

@Book{Adams:1975,
  author    = {Adams, R. A.},
  title     = {Sobolev spaces},
  publisher = {Academic Press},
  year      = {1975},
  volume    = {65},
  series    = {Pure and Applied Mathematics},
  address   = {New York-London},
}

\end{document}